\documentclass[11pt]{amsart}
\usepackage[T1]{fontenc}
\usepackage{amssymb,amsmath}
\usepackage{microtype}
\usepackage{bbm}
\usepackage{bm}
\usepackage[]{graphicx}
\usepackage{color}
\usepackage{amssymb, epsfig}
\usepackage{enumitem}
\usepackage[toc,page]{appendix} 
\usepackage{hyperref}
\usepackage{enumitem}

\numberwithin{equation}{section}

\newcommand{\E}{\mathbb{E}}
\newcommand{\clr}{\mathcal{R}}
\newcommand{\PP}{\mathbb{P}}
\newcommand{\Lip}{\operatorname{Lip}}
\newcommand{\DD}{\mathbb{D}}
\newcommand{\NN}{\operatorname{N}}
\newcommand{\RR}{\mathbb{R}}
\newcommand{\ZZ}{\mathbb{Z}}

\renewcommand{\d }{\mathrm{d}}

\numberwithin{figure}{section}
\newtheorem{theorem}{Theorem}[section]
\newtheorem{lemma}[theorem]{Lemma}

\newtheorem{assumption}[theorem]{Assumption}
\newtheorem{corollary}[theorem]{Corollary}
\newtheorem{proposition}[theorem]{Proposition}
\newtheorem{remark}[theorem]{Remark}
\begin{document}

\title[ On the density of the supremum of nonlinear SPDEs ] {On the density of the supremum of nonlinear SPDEs}

 \author[Karali, Stavrianidi, Tzirakis, Zoubouloglou]{G. Karali$^{\dag, \ddag}$, A. Stavrianidi$^{\# \S}$, K. Tzirakis$^{\ddag *}$, P. Zoubouloglou$^{\# \S}$}

 \thanks
 {$^{\dag}$ Department of Mathematics, National and Kapodistrian
 University of Athens, Panepistimiopolis, 15784, Greece.}
 \thanks
 {$^{\ddag}$ Institute of Applied and Computational Mathematics,
 FORTH, GR--711 10 Heraklion, Greece.}
 \thanks
 {$^{*}$ Department of Mathematics and Applied Mathematics,
 University of Crete, GR--714 09 Heraklion, Greece.}
 \thanks
 {$^{\#}$ Department of Mathematics, University of M\"unster, Einsteinstr. 62, D-48149 M\"unster, Germany.}
 \thanks
 {$^{\S}$ Corresponding Authors.}
 \thanks{E-mails: gkarali@math.uoa.gr, alex.st@uni-muenster.de, kostas.tzirakis@gmail.com, p.zoubouloglou@uni-muenster.de}

\subjclass{}
%
%

\begin{abstract}
We study the one-dimensional stochastic partial differential equation
\begin{equation*}
\frac{\partial u}{\partial t}(t,x)
=
-\kappa \frac{\partial^4 u}{\partial x^4}(t,x)
+ \rho \frac{\partial^2 u}{\partial x^2}(t,x)
+ b(u(t,x))
+ \sigma(u(t,x))\, \dot W(t,x),
\end{equation*}
posed on a bounded spatial domain, where $u$ is understood in the random field sense and $\dot W(t,x)$ is space-time white noise. Depending on the value of $\kappa$, this equation includes the nonlinear stochastic heat equation with Dirichlet or Neumann boundary conditions, as well as the linearized stochastic Cahn–Hilliard equation with Neumann boundary conditions. We prove that the supremum of the solution admits a density with respect to the Lebesgue measure. Our approach is based on Malliavin calculus, and in particular on the version of the Bouleau--Hirsch criterion for suprema developed by Nualart and Vives. One of the main difficulties lies in the analysis of the argmax set of the solution and in showing that the Malliavin derivative is almost surely nondegenerate on this set. As part of our arguments, we establish H\"older continuity properties for the Malliavin derivative of the solution as an $L^2-$valued process in the regimes considered in this work.
\end{abstract}

\keywords{Stochastic partial differential equations, stochastic heat equation, Malliavin calculus, Radon-Nikodym derivative, supremum of random field}

\subjclass{Primary 60H15; Secondary 60H07, 60G60, 60H30, 35K55}
\maketitle

\pagestyle{myheadings}
\thispagestyle{plain}



%
\thispagestyle{plain}
%
%

\tableofcontents

\section{Introduction} \label{sbi}
\allowdisplaybreaks

We study the existence of a density for the supremum of solutions to nonlinear stochastic partial differential equations (SPDEs) in one spatial dimension. More precisely, we consider the SPDE
\begin{equation}\label{sCH} 
\frac{\partial u}{\partial t}(t,x)
=
-\kappa \frac{\partial^4 u}{\partial x^4}(t,x)
+ \rho \frac{\partial^2 u}{\partial x^2}(t,x)
+ b(u(t,x))
+ \sigma(u(t,x))\, \dot W(t,x),
\end{equation}
for $(t,x)\in [0,T]\times[0,1]$, where $\kappa,\rho\ge 0$ and $\kappa+\rho>0$. The coefficients $b:\RR\to\RR$ and $\sigma:\RR\to\RR$ are assumed to be Lipschitz and $\sigma$ is upper bounded and lower bounded by a positive constant (see Assumption~\ref{assum:b-s} below). We impose an initial condition $u_0\in C([0,1])$, together with additional assumptions that will be specified in Assumption~\ref{ass:uo-holder}. Finally, $\dot{W}(t,x)$ denotes a space-time white noise in the sense of Walsh \cite{W6}, defined on some probability space $(\Omega,\mathcal{F}, \PP)$.

Our goal is to study the existence of a density with respect to Lebesgue measure, for the supremum of the solution to \eqref{sCH}. We consider three different regimes, corresponding to different choices of the parameters $\kappa$ and $\rho$ and to different boundary conditions. These will be referred to throughout as \emph{Regimes} \ref{:dirichlet}, \ref{case:neumann} and \ref{case:fourthorder}:
\begin{enumerate}[label=(\roman*)]
\item \label{:dirichlet}
$\kappa=0$, corresponding to the stochastic heat equation with Dirichlet boundary conditions
\[
u(t,0)=u(t,1)=0,\qquad t\in[0,T].
\]
For consistency, this requires $u_0(0)=u_0(1)=0$. We may take $\rho=1$ without loss of generality. We denote by $G^{\operatorname{D}}$ the associated Dirichlet heat kernel:
\begin{equation}\label{eq:def-GD}
G_t^{\operatorname{D}}(x,y)
=
2\sum_{k=1}^{\infty} e^{-k^2\pi^2 t}\sin(k\pi x)\sin(k\pi y).
\end{equation}

\item \label{case:neumann}
$\kappa=0$, corresponding to the stochastic heat equation with Neumann boundary conditions
\[
\frac{\partial u}{\partial x}(t,0) = \frac{\partial u}{\partial x}(t,1) =0,\qquad t\in[0,T].
\]
Again, we take $\rho=1$ without loss of generality and denote by $G^{\operatorname{N}}$ the associated Neumann heat kernel:
\begin{equation}\label{eq:neumann-kernel}
G_t^{\operatorname{N}}(x,y)
=
1+2\sum_{k=1}^{\infty} e^{-k^2\pi^2 t}\cos(k\pi x)\cos(k\pi y).
\end{equation}

\item \label{case:fourthorder}
$\kappa>0$, corresponding to a semilinear fourth-order equation with Neumann boundary conditions
\[
\frac{\partial u}{\partial x}(t,0) = \frac{\partial u}{\partial x}(t,1) = \frac{\partial^3 u}{\partial x^3}(t,0) = \frac{\partial^3 u}{\partial x^3}(t,1) = 0, \qquad t\in[0,T].
\]
In this case, we may assume $\kappa=1$ while keeping a general $\rho\ge0$. The Green's function associated with the operator $-\partial_x^4+\rho\partial_x^2$ on $[0,1]$ with these boundary conditions is given by
\begin{equation}\label{eq:H-def}
H_t(x,y)
=
1+2\sum_{k=1}^{\infty}
e^{-(k^4\pi^4+\rho k^2\pi^2)t}
\cos(k\pi x)\cos(k\pi y).
\end{equation}
This regime includes the linearized Cahn--Hilliard equation.
\end{enumerate}

The mild solution to \eqref{sCH} is given by
\begin{equation}\label{weakCH}
\begin{split}
u(t,x)
={}&
\int_0^1 \mathcal{G}_t(x,y)u_0(y)\,dy
+
\int_0^t \int_0^1 \mathcal{G}_{t-s}(x,y)b(u(s,y))\,dy\,ds
\\
&\quad
+\int_0^t\int_0^1 \mathcal{G}_{t-s}(x,y)\sigma(u(s,y))\,W(dy,ds),
\end{split}
\end{equation}
where $\mathcal{G}_t(x,y)$ denotes the corresponding Green's function, that is, $\mathcal{G}=G^{\operatorname{D}}, G^{\operatorname{N}}$ or $H$.

\subsection{Main results and assumptions}\label{subsec:main-results}

We now state the assumptions on the drift and diffusion coefficients appearing in \eqref{sCH}.

\begin{assumption}\label{assum:b-s}
The following conditions hold:
\begin{enumerate}
\item The function $b:\RR\to\RR$ is Lipschitz continuous.
\item The function $\sigma:\RR\to\RR$ is Lipschitz continuous, bounded, and uniformly elliptic, in the sense that there exists a constant $C_\sigma>1$ such that
\[
\frac{1}{C_\sigma}\le \sigma(x) \le C_\sigma,
\qquad x\in\RR.
\]
\end{enumerate}
\end{assumption}

Assumption~\ref{assum:b-s} is already sufficient to establish one of our main results, namely the existence of a density for the supremum when the latter is taken over an arbitrary nonempty compact subset of $(0,T]\times(0,1)$. In order to extend this result to the full set $[0,T]\times[0,1]$ in \emph{Regimes} \ref{:dirichlet}, \ref{case:neumann}, we impose the following local condition on the initial datum $u_0$.

\begin{assumption}\label{ass:uo-holder}
The function $u_0:[0,1]\to\RR$ is locally H\"older continuous of order $\alpha>\frac12$ near a maximizer. More precisely, there exist a maximizer $x^*\in[0,1]$, constants $\alpha>\frac12$, $C_0>0$, and $r_0>0$ such that
\begin{equation}\label{eq:local-holder}
\sup_{y, z \in (x^{*}-r_0, x^{*}+r_0) \cap [0,1]} \frac{|u_0(z)-u_0(y)|}{|y-z|^\alpha}\le C_0.
\end{equation}
\end{assumption}

Observe that Assumption~\ref{ass:uo-holder} is automatically satisfied, for instance, when $u_0\equiv 0$, which is the initial condition considered in the important work \cite{dalangpu}.

We are now ready to state our first main result, which concerns the case $\kappa=0$, that is, the stochastic heat equation. Throughout the paper, we denote by $\lambda$ the Lebesgue measure on $\RR$.

\begin{theorem}\label{maintheorem}
Let $u$ be the solution to \eqref{sCH} with $\kappa=0$ and $\rho=1$, subject to the boundary conditions of either Regime~\ref{:dirichlet} or Regime~\ref{case:neumann}. Suppose that Assumption~\ref{assum:b-s} holds. Then
\begin{equation}\label{eq:sup-u-d12}
\sup_{(t,x)\in[0,T]\times[0,1]} u(t,x)\in \mathbb\mathbb{D}^{1,2},
\end{equation}
and the law of $\sup_{(t,x)\in K}u(t,x)$ is absolutely continuous with respect to $\lambda$ for every nonempty compact set $K\subset(0,T]\times(0,1)$ in Regime \ref{:dirichlet} or $K\subset(0,T]\times [0,1]$ in Regime \ref{case:neumann}. If, in addition, Assumption~\ref{ass:uo-holder} holds, then the law of
\[
\sup_{(t,x)\in[0,T]\times[0,1]} u(t,x)
\]
is absolutely continuous with respect to $\lambda$.
\end{theorem}

The statement \eqref{eq:sup-u-d12} is proved in Section~\ref{subsec:second-cond-k=0}; see in particular Theorem~\ref{thm:nualart-criterion-existence}. The absolute continuity of the supremum over compact subsets of $(0,T)\times(0,1)$ follows from Sections~\ref{subsubsec:gammaholder}, \ref{subsubsec:k=0-small-ball}, and \ref{subsubsec:k=0sup-compacts}. The extension to the full domain $[0,T]\times[0,1]$ is established in Section~\ref{subsubsec:k=0-whole-space}.

For Regime~\ref{case:fourthorder}, we first present an assumption that allows us to pass to the full space $[0,T] \times [0,1]$, analogous to Assumption \ref{ass:uo-holder}.

\begin{assumption}\label{ass:u0-fourth}
The function $u_0:[0,1]\to\RR$ is locally in $C^{1,\alpha}$ near a maximizer with $\alpha > \frac{1}{2}$, i.e., there exists a maximizer $x^*\in[0,1]$, constants $C_0>0$, and $r_0>0$ such that $u_0$ is differentiable for all $y \in [0,1] \cap (x^*-r_0,x^*+r_0)$ and, moreover,
\begin{equation}\label{eq:local-holder-4th}
\sup_{y,z \in (x^{*}-r_0, x^{*}+r_0) \cap [0,1]} \frac{|u'_0(y)-u'_0(z)|}{|y-z|^\alpha} \le C_0.
\end{equation}
Moreover, if $x^*  \in \{0,1\}$, then $u_0'(x^*) = 0$, where $u_0'(x)$ denotes the one-sided derivative.
\end{assumption}


We obtain the following analogous statement for Regime \ref{case:fourthorder}.

\begin{theorem}\label{thm:4th-order}
Let $u$ be the solution to \eqref{sCH} with $\kappa=1$ and $\rho\ge 0$, subject to the boundary conditions of Regime~\ref{case:fourthorder}. Suppose that Assumption~\ref{assum:b-s} holds. Then
\[
\sup_{(t,x)\in[0,T]\times[0,1]} u(t,x)\in \mathbb\mathbb{D}^{1,2},
\]
and, for every nonempty compact set $K\subset(0,T]\times[0,1]$, the law of $\sup_{(t,x)\in K}u(t,x)$ is absolutely continuous with respect to $\lambda$. If, in addition, Assumption~\ref{ass:u0-fourth} holds, then the law of
\[
\sup_{(t,x)\in[0,T]\times[0,1]} u(t,x)
\]
is absolutely continuous with respect to $\lambda$.
\end{theorem}

The proof of Theorem~\ref{thm:4th-order} closely follows that of Theorem~\ref{maintheorem}. We point out the main differences in Section~\ref{sec:kappa-neq-0}.

\begin{remark}

We consider Neumann boundary conditions for the linearized Cahn–Hilliard equation because they are the natural boundary conditions for this model. The Cahn–Hilliard dynamics is a mass-conserving gradient flow, and on a bounded interval its standard formulation is coupled with no-flux conditions, which in the linearized equation reduce to the setting of Regime \ref{case:fourthorder}.

\end{remark}

\subsection{Literature review and proof technique} \label{subsec:lit-review}

By now, a substantial literature has developed on the existence and quantitative properties of densities for pointwise evaluations of random-field solutions to SPDEs. For Regimes \ref{:dirichlet} and \ref{case:neumann}, that is, for the stochastic heat equation under Dirichlet or Neumann boundary conditions, this line of investigation was initiated in \cite{pardouxzhang}, where it was shown that for each fixed $(t,x)\in(0,T]\times(0,1)$ (respectively, $(t,x) \in (0,T]\times [0,1]$ for Neumann conditions), the random variable $u(t,x)$ is Malliavin differentiable and admits a density with respect to Lebesgue measure. Subsequent works established that this density is in fact smooth and strictly positive away from the boundary; see \cite{muellernualart,ballypardoux}. These results were later extended to solutions of the stochastic heat equation with measure-valued initial data in \cite{ChenHuNualart2021RegularityPositivitySHE}.

Going beyond the stochastic heat equation, Malliavin differentiability and the existence of a density for pointwise evaluations of the stochastic Cahn--Hilliard equation were established in \cite{weber,CuiHong20AbsoluteContinuity}. See also \cite{AFK}, where the authors proved the existence of a density for pointwise evaluations of the stochastic Cahn--Hilliard/Allen--Cahn combined model. We note that \eqref{sCH} under Regime \ref{case:fourthorder}, that is, when $\kappa>0$, corresponds to the linearized Cahn-Hilliard for appropriate $b,\sigma$. In particular, a straightforward adaptation of the argument in \cite{AFK} yields that, for every fixed $(t,x)\in(0,T]\times[0,1]$, the corresponding solution $u(t,x)$ also admits a density.

A natural question, once the density of pointwise evaluations has been understood, is whether analogous results hold for the supremum of the underlying stochastic process or random field. In general, obtaining a closed-form expression for the density of the supremum is an almost impossible task. Nevertheless, even qualitative or quantitative information on that density, such as smoothness, or upper and lower bounds, can be useful in applications, for instance in estimating exit probabilities from prescribed domains. 
At the same time, even proving the mere existence of a density for the supremum is typically a delicate problem.

For classical processes, some results are available. The density of the supremum of Brownian motion is explicitly known. For fractional Brownian motion, the existence of a density for the supremum was established in \cite{nualartvives}, and its smoothness was later proved in \cite{zadinualart}. In the context of diffusion processes, Hayashi and Kohatsu-Higa showed in \cite{HayashiKohatsuHiga2013} that, for a multidimensional diffusion with smooth, uniformly elliptic, and commutative diffusion vector fields, the random vector
\[
(X^1_{\theta_1},\dots,X^d_{\theta_1}),
\]
where $\theta_1$ is the almost surely unique time at which $X^1$ attains its maximum on $[0,T]$, admits a smooth density. In particular, this implies that $\sup_{0\le t\le T} X^1_t$ itself has a smooth density.

The question becomes substantially more delicate for random fields. The existence of a smooth density for the supremum of the Brownian sheet over a bounded rectangle was established in \cite{floritnualart}. Perhaps the work closest to ours is that of Dalang and Pu \cite{dalangpu}, who derived estimates for the density of the supremum of the linear stochastic heat equation, that is, \eqref{sCH} in the case $\kappa=b=0$ and $\sigma\equiv 1$. In that setting, the solution is a Gaussian random field. Moreover, the initial condition in \cite{dalangpu} is $u_0\equiv 0$, and therefore Assumption \ref{ass:uo-holder} is satisfied. Since \cite{dalangpu} obtains stronger conclusions in a less general setting, the methods used there are substantially different from ours. In particular, the motivation for studying the density of the supremum of the solution in \cite{dalangpu} stems from obtaining upper bounds for hitting probabilities for the solution to the linear stochastic heat equation; see also 
\cite{DalangKhoshnevisanNualart2007Additive,DalangKhoshnevisanNualart2009Multiplicative} for related works on hitting probabilities. As already noted in \cite[Remark 1.6(2)]{dalangpu}, the nonlinear case is considerably more challenging. It is unclear whether estimates on the density as in \cite{dalangpu} can be obtained for the nonlinear equation by extending the tools used herein. To the best of our knowledge, the present paper provides the first result on the existence of a density for the supremum of solutions to nonlinear SPDEs.

As in all of the works mentioned above, Malliavin calculus lies at the heart of our proofs. In particular, the Bouleau--Hirsch criterion for the existence of a density was adapted in \cite{nualartvives} to the case where the random variable is the supremum of a stochastic process or random field; we recall this criterion in Theorem~\ref{thm:nualart-criterion-existence} below.

To establish the existence of a density in our setting, we must verify two key conditions, namely Conditions (ii) and (iii) in Theorem~\ref{thm:nualart-criterion-existence}. The first of these requires suitable integrability of the supremum of the squared norm of the Malliavin derivative of the underlying random field, viewed as an element of the relevant Gaussian Hilbert space. The main technical difficulty comes from the fact that the supremum appears inside the expectation. We overcome this by proving a Kolmogorov-type continuity result for the Malliavin derivative. We thereby obtain the H\"older continuity of the Malliavin derivative as an $L^{2}(dsdy)$ process, with exponents matching those of the underlying field in the regimes considered here.

The second condition requires proving that the Malliavin derivative is nondegenerate on the $\arg\max$ set of the solution. This naturally leads us to study the argmax set associated with the random field $(t,x)\mapsto u(t,x)$. A standard strategy for verifying this condition is to prove that the maximizer is almost surely unique, and then approximate it by maxima over a dense countable subset; see, for instance, Lemmas 2.1.8 and 2.1.9 in \cite{Nua06book} for an implementation of this argument for the Brownian sheet. In our setting, however, establishing almost sure uniqueness of the maximizer appears to be a difficult problem in its own right. Existing criteria in the literature are largely restricted either to one-parameter processes (see, e.g., \cite{Pimentel2014,Adler1990,AdlerTaylor2007,Piterbarg1996}) or to Gaussian random fields (see, e.g., \cite{LopezPimentel2018} and the references discussed in \cite{dalangpu}), and therefore do not apply to our SPDEs. Moreover, in the regimes considered here, the initial condition is deterministic, which in particular forces, for example, $Du(0,x)=0$ for all $x\in[0,1]$. In addition, Dirichlet boundary conditions force $Du(t,0)=Du(t,1)=0$ for all $t \in [0,T]$. As a result, the verification of this condition is naturally reduced to two main tasks: first, proving that the Malliavin derivative is nondegenerate on every compact subset of the interior $(0,T)\times(0,1)$; and second, showing that the $\arg\max$ set of $u$ almost surely does not intersect $\{t=0\}$ or $\{x=0\}$ and $\{x=1\}$ in the case of Dirichlet boundary conditions.

A further difficulty, compared with most of the existing literature, is the lack of Gaussianity of the underlying random field. With the exception of \cite{HayashiKohatsuHiga2013}, the works discussed above concern Gaussian processes or Gaussian random fields. Although the process considered in \cite{HayashiKohatsuHiga2013} is not necessarily Gaussian, the authors impose additional structural assumptions in order to guarantee almost sure uniqueness of the maximizer; see also \cite{Nakatsu2016} for a related problem. In our setting, this absence of Gaussian structure makes several parts of the analysis substantially more involved.

\subsection{Notation} \label{subsec:notation}

Throughout the paper, we work with the jointly continuous version of the random field $\{u(t,x):(t,x)\in [0,T]\times[0,1]\}$.

We write $C$ and $c$ for positive constants whose value may change from line to line. When needed, subscripts indicate the dependence of a constant on the relevant parameters, such as $T$, $b$, $\sigma$, or $\rho$.

We write $\mathcal C([0,T])=\mathcal C([0,T];\RR)$ to be the space of continuous functions from $[0,T]$ to $\RR$, equipped with the usual norm $\|f\|_\infty:=\sup_{t\in[0,T]}|f(t)|$.

For $p\in[1,\infty]$, we denote by $L^p(0,1)$ the usual space on $(0,1)$, with norm $\|\cdot\|_{L^p}$. If $f:\RR \to\RR$, we denote by
\[
\Lip(f):=\sup_{x\neq y}\frac{|f(x)-f(y)|}{|x-y|}
\]
its Lipschitz constant, whenever $f$ is Lipschitz continuous.

We denote by $C_p^\infty(\RR^n)$ the class of all smooth functions $\varphi:\RR^n\to\RR$ such that $\varphi$ and all of its partial derivatives have at most polynomial growth at infinity.

We write
\begin{equation} \label{eq:heatkernelonR}
p_t(x,y)=p_t(x-y):=\frac{1}{\sqrt{4\pi t}}\exp\left(-\frac{|x-y|^2}{4t}\right),
\qquad t>0,\ \ x,y\in\RR,
\end{equation}
for the standard one-dimensional heat kernel on $\RR$. 

Moreover, for $t>0$ and $z\in\mathbb R$, we write
\begin{equation} \label{eq:gtrho}
    g_t^\rho(z):=\frac{1}{2\pi}\int_{\mathbb R}e^{-t(\xi^4+\rho\xi^2)}e^{iz\xi}\,d\xi, \qquad t >0,\ \ z \in \mathbb{R}
\end{equation}
for the fundamental solution to the fourth order equation $u_{t}=-u_{xxxx}+\rho u_{xx}$  with $\rho \geq 0$ on $\RR.$ 

Finally, for $a,b > 0$, define the beta function
\begin{equation*}
    \operatorname{B}(a,b) := \int_{0}^{1}x^{a-1}(1-x)^{b-1}dx =\int_{0}^{\infty}x^{a-1}(x+1)^{-(a+b)}dx, \quad a, b > 0.
\end{equation*}

\subsection{Organization of the paper} \label{subsec:organization}

The paper is organized as follows. In Section \ref{sec:prelim} we recall some preliminaries, including elements of the Malliavin calculus for our context (in Section \ref{subsec:malliavin}) and the anisotropic Kolmogorov criterion, as well as a criterion for the existence of density for the supremum of a continuous random field (in Section \ref{subsec:Kolmogorov}); we also recall some well-known technical lemmas for the Green's functions under considerations in Section \ref{subsec:green-estimates}. In Section \ref{sec:auxiliary}, we prove some auxiliary lemmas that are needed for the proof of Theorem \ref{maintheorem} (see Section \ref{subsec:kappa=0}) and Theorem \ref{thm:4th-order} (see Section \ref{subsec:kappa>0}). Finally, Sections \ref{sec:proof-kappa=0} and \ref{sec:kappa-neq-0} deal respectively with the proofs of Theorems \ref{maintheorem} and \ref{thm:4th-order}.

\section{Preliminaries}\label{sec:prelim}

\subsection{Elements of Malliavin calculus} \label{subsec:malliavin}
Let \(W=\{W(h),\,h\in\mathcal H\}\) be the isonormal Gaussian process associated with the space--time white noise, where $\mathcal H=L^2([0,T]\times[0,1];\RR)$. Let \(\mathcal S\) denote the family of smooth random variables of the form
\[
F=f\bigl(W(h_1),\ldots,W(h_n)\bigr),
\]
where \(n\geq 1\), \(h_1,\ldots,h_n\in\mathcal H\), and \(f\in C_p^\infty(\RR^n)\). For \(F\in\mathcal S\), its Malliavin derivative is the $\mathcal{H}$-valued stochastic process \(DF=(D_{t,x}F)_{(t,x)\in[0,T]\times[0,1]}\) defined by
\[
D_{t,x}F=\sum_{i=1}^n \partial_i f\bigl(W(h_1),\ldots,W(h_n)\bigr)h_i(t,x).
\]
More generally, for any integer \(k\geq 1\), the derivative of order \(k\), denoted by \(D^kF\), is defined in the usual way and takes values in \(\mathcal H^{\otimes k}\). Then, for \(k,p\geq 1\), the space \(\mathbb D^{k,p}\) of $k$-th order Malliavin differentiable functions with integrability of order $p$ is defined as the closure of \(\mathcal S\) under the norm
\[
\|F\|_{k,p}^p
=
\E\bigl[|F|^p\bigr]
+
\sum_{j=1}^k
\E\bigl[\|D^jF\|_{\mathcal H^{\otimes j}}^p\bigr].
\]

More precisely, the Malliavin derivative of the mild solution given in \eqref{weakCH} for $s<t$ is the solution to the evolution equation 
\begin{equation}\label{PPddfo3}
\begin{split}
 D_{s,y}u(t,x) = \,& \mathcal{G}_{t-s}(x,y)\sigma(u(s,y)) 
 \;+\;
\int_{s}^{t}\int_{0}^{1}\mathcal{G}_{t-\theta}(x,r)m(\theta,r)D_{s,y}u(\theta,r) dr d\theta  
\\
&\;+\; \int_{s}^{t}\int_{0}^{1} \mathcal{G}_{t-\theta}(x,r) \hat{m}(\theta,r)D_{s,y}u(\theta,r) W(dr,d \theta) ,
\end{split} 
\end{equation}
and
$$D_{s,y} u(t,x)=0,\qquad \mbox{for}\;\; s> t.$$
Proposition 1.2.4 from \cite{Nua06book} applies and gives us that there exist random functions $m,\hat{m}$ such that 
\begin{eqnarray*}\label{pod1}
D_{s,y}[b(u(\theta,r))]&=&m(\theta,r)D_{s,y}u(\theta,r) ,
\\
D_{s,y}[\sigma(u(\theta,r))]&=&\hat{m}(\theta,r)D_{s,y}u(\theta,r),
\end{eqnarray*}
with $|m(t,x)| \le \Lip(b)$, $|\hat{m}(t,x)| \le \Lip(\sigma) $ a.s. for any $t \in [0,T]$ and $x \in [0,1]$.

\subsection{Anisotropic Kolmogorov criterion and a Bouleau-Hirsch type criterion} \label{subsec:Kolmogorov}

In this section, we recall two key well-known theorems that will be used to prove our main results. 

The first result is a variant of the Kolmogorov-Chentsov theorem for SPDEs, often referred to as the anisotropic Kolmogorov criterion. The following is a simpler formulation of the one given in Theorem A.3.1 from \cite{dalangsole}, coupled with Remark A.3.2(a) therein.

\begin{theorem}[Anisotropic Kolmogorov criterion] \label{thm:kolmogorov}
Let $\{v(t,x): (t,x) \in [0,T] \times [0,1]\}$ be a stochastic process taking values in a Banach space $(\mathbb{B}, \|\cdot\|_{\mathbb{B}})$ and fix $\lambda_1, \lambda_2 \in (0,1]$. Suppose that there exists a constant $C>0$ and a $p > \frac{1}{\lambda_1} + \frac{1}{\lambda_2}$ such that
\begin{equation} \label{eq:Kolmogorov}
\E \|v(t,x) - v(s,y)\|_{\mathbb{B}}^p \le C \left( |t-s|^{\lambda_1} + |x-y|^{\lambda_2}  \right)^p, \quad s,t \in [0,T] , \; x,y \in [0,1].
\end{equation}
Then, there exists a version $\widetilde v$ of $v$ on $[0,T] \times [0,1]$ such that, a.s., $\widetilde v$ is $\mu_1-$H\"older continuous in time for all $\mu_1 < \lambda_1 \left(1 -  \frac{\lambda_1 + \lambda_2}{\lambda_1 \lambda_2 p} \right)$ and $\mu_2-$H\"older continuous in space for all $\mu_2 < \lambda_2 \left(1 -  \frac{\lambda_1 + \lambda_2}{\lambda_1 \lambda_2 p}\right)$.

In particular $\{\widetilde v(t,x): t \in [0,T], x \in [0,1]\}$ is continuous (in $\mathbb{B}$) and, moreover, if  $ \E\|\widetilde v (t_0,x_0)\|_{\mathbb{B}}^p < \infty$ for some $t_0 \in [0,T], x_0 \in [0,1]$, then 
\begin{equation}
\E \left( \sup_{t \in [0,T]} \sup_{x \in [0,1]} \|\widetilde v (t,x)\|_{\mathbb{B}}^p \right) < +\infty .
\end{equation}
\end{theorem}

The following result is a combination of Propositions 2.1.10 and 2.1.11 of \cite{Nua06book}. For a random field $v$ on $K \subseteq [0,T] \times [0,1]$, where $K$ is compact, we define the argmax set to be
\begin{equation} \label{eq:arg-max-set}
    \mathcal{S}_{K} := \arg \max _{(t,x) \in K} v(t,x).
\end{equation}

The following theorem is a version of the Bouleau-Hirsch criterion (see, e.g., Theorem 7.2.1 \cite{Nualart_Nualart_2018}) for the Malliavin differentiability and the existence of density of the supremum of a random field; it was first developed in the work of Nualart and Vives \cite{nualartvives}. In its general formulation, this criterion can be found as Theorem 2.1.3 of \cite{Nua06book}, and we adjust it here to the case of SPDEs with spatial dimension $1$.

\begin{theorem}\label{thm:nualart-criterion-existence}
Let $\{v(t,x): (t,x) \in [0,T] \times [0,1]\}$ be a continuous stochastic process and let $K$ denote any nonempty compact subset of $[0,T]\times [0,1]$. Consider the following three conditions:
\begin{enumerate}[label=(\roman*), ref=(\roman*)]
\item \label{crit:nualart-1} $\E \left( \sup_{(t,x) \in K} |v(t,x)|^2 \right) < +\infty$ and $v(t,x) \in \mathbb{D}^{1,2}$ for all $(t,x) \in K$.
\item \label{crit:nualart-2} The process $\left\{D_{\cdot,\cdot} v(t,x), (t,x) \in K \right\}$  possesses a continuous version as an $L^{2}(ds \, dy)$ valued process and $\E \left( \sup_{(t,x) \in K} \int_0^T \int_0^1 |D_{s,y} v(t,x)|^2 dy ds \right) < +\infty$.
\item \label{crit:nualart-3} $\PP\left(\text{for all }(t,x) \in  \mathcal{S}_K: \int_0^T \int_0^1 |D_{s,y} v(t,x)|^2 dy ds  > 0 \right) = 1$, where $\mathcal{S}_K$ is defined in \eqref{eq:arg-max-set}.
\end{enumerate}
If $v$ satisfies conditions \ref{crit:nualart-1} and \ref{crit:nualart-2}, then the random variable $\sup_{(t,x) \in K} v(t,x) \in \mathbb{D}^{1,2}$. If, in addition, $v$ satisfies Condition \ref{crit:nualart-3}, then it is absolutely continuous with respect to $\lambda$. 
\end{theorem}

\subsection{Green's function estimates} \label{subsec:green-estimates}

The following representations for the Green's functions will be useful in the sequel.
\begin{lemma}
For all $t>0$ and $x,y\in[0,1]$, the Dirichlet and Neumann heat kernels and the Green's function $H_t$ admit the alternative representation
\begin{equation}\label{eq:images-Neumann}
\begin{split}
G_t^{\operatorname{D}}(x,y)=\sum_{m\in\ZZ}\left(p_t(x-y+2m)-p_t(x+y+2m)\right), \\
G_t^{\NN}(x,y)=\sum_{m\in\ZZ}\left(p_t(x-y+2m)+p_t(x+y+2m)\right), \\
H_t(x,y) = \sum_{m \in \ZZ} \left( g_t^\rho(x-y + 2m )+g_t^\rho (x+y + 2m)\right),
\end{split}
\end{equation}
where the series $\sum_{m\in\ZZ} p_t(x-y+2m)$, $\sum_{m\in\ZZ} p_t(x+y+2m)$, $\sum_{m \in \ZZ} g_t^\rho(x-y + 2m )$, and $\sum_{m \in \ZZ} g_t^\rho(x+y + 2m )$ are absolutely convergent, hence the representations in \eqref{eq:images-Neumann} are well-defined.
\end{lemma}

\begin{proof}
The representations for the heat kernels $G_t^{\operatorname{D}}$ and $G_t^{\NN}$ are well-known, see, e.g., Remark 2.1 and (2.4) in \cite{ballypardoux}. We were not able to find the proof for the representation of $H_t$ in the literature, so we include it here for completeness.

First, using that $\cos(k\pi x)\cos(k\pi y)=\frac12\bigl(\cos(k\pi(x-y))+\cos(k\pi(x+y))\bigr)$, we can rewrite $H_t$ as
\begin{equation} \label{eq:meow12}
H_t(x,y)
= 1 + 
2\sum_{k=0}^\infty e^{-((k\pi)^4+\rho(k\pi)^2)t}\cos(k\pi x)\cos(k\pi y) =
\frac12\Bigl( P_t(x-y)+P_t(x+y)\Bigr),
\end{equation}
where
\[
P_t(s):=1+2\sum_{k=1}^\infty e^{-((k\pi)^4+\rho(k\pi)^2)t}\cos(k\pi s).
\]
Recalling the form of $g_t^\rho(z)$ in \eqref{eq:gtrho}, its Fourier transform is given by $\widehat{g_t^\rho}(\xi)=e^{-t(\xi^4+\rho\xi^2)}$. Then, evaluating the Poisson summation formula at $\pi\mathbb Z$ gives
\begin{align*}
P_t(s) =
\sum_{k\in\mathbb Z} e^{-((k\pi)^4+\rho(k\pi)^2)t}e^{ik\pi s} =
2\sum_{m\in\mathbb Z} g_t^\rho(s+2m).
\end{align*}
which completes the proof.
\end{proof}

The following lemma establishes some Gaussian upper bounds for the Dirichlet and Neumann heat kernels, as well as some of their derivatives.

\begin{lemma}
\label{lem:Neumann-Gaussian-|x-y|}
 Let $G$ denote either $G^{\operatorname{D}}$ or $G^{\operatorname{N}}$.
There exists a constant $C_T \in(0,\infty)$ such that for all
$0<t\le T$ and all $x,y\in[0,1]$,
\begin{align}
    G_t(x,y)
\;&\le\;
C_T p_t(x,y),  \label{eq:G-lower-upper-T}\\
|\partial_x G_t(x,y)|\;&\le\; C_T\,t^{-1}\exp\!\left(-\frac{|x-y|^2}{8\,t}\right), \label{eq:G-deriv-x-upper-T} \\
|\partial_t G_t (x,y)|\;&\le\; C_T\,t^{-3/2}\exp\!\left(-\frac{|x-y|^2}{8\,t}\right). \label{eq:Dirichlet-tder-upper-T}
\end{align}
\end{lemma}

\begin{proof}

In the rest of the proof, we write $G$ to denote $G^{\operatorname{N}}$. The proofs for $G^{\operatorname{D}}$ are analogous.

Set $a:=|x-y|\in[0,1]$. For the inequality in \eqref{eq:G-lower-upper-T}, we provide estimates for both series in \eqref{eq:images-Neumann}. Let $u\ge 0$ and $n\ge 0$. Since $(u+2n)^2 \ge u^2 +4n^2$, we have, for $t>0$,
\begin{equation}\label{eq:gauss-shift-T}
\exp\!\left(-\frac{(u+2n)^2}{4t}\right)
\le
\exp\!\left(-\frac{u^2}{4t}\right)\exp\!\left(-\frac{n^2}{t}\right).
\end{equation}

We first bound the ``$x-y$'' part in \eqref{eq:images-Neumann}. Note that since $a\le 1$, for every $k\in\ZZ\setminus\{0\}$ we have
\begin{equation} \label{eq:090}
    |x-y+2k|\ge a+2(|k|-1).
\end{equation}
Then, with explanations given below,
\begin{equation} \label{eq:091}
\begin{split}
    \sum_{k\in\ZZ} p_t(x-y+2k)
&\le
p_t(a)
+
2 (4\pi t)^{-1/2}\sum_{n=1}^\infty \exp\!\left(-\frac{(a+2n-2)^2}{4t}\right) \\
&\le p_t(a)
+
2 (4\pi t)^{-1/2}\sum_{n=1}^\infty \exp\!\left(-\frac{a^2}{4t}\right)\exp\!\left(-\frac{(n-1)^2}{t}\right) \\
&\le p_t(a)
+
2 (4\pi t)^{-1/2} \exp\!\left(-\frac{a^2}{4t}\right) \sum_{n=0}^\infty \exp\!\left(-\frac{n^2}{t}\right) \\
&\le \frac{C_T}{\sqrt{t}}\,\exp\!\left(-\frac{a^2}{4t}\right).
\end{split}
\end{equation}
In the estimates above, the first line follows by \eqref{eq:090}, the second line follows by \eqref{eq:gauss-shift-T} with $u=a$ and $n-1$ in place of $n$. In the last line, $C_T$ is a constant that depends only on $\sum_{m\ge 0}e^{-m^2/T}$.

Next, let $b := \min_{k\in\ZZ} |x+y+2k| = \min\{x+y,\,2-x-y\} \in [0,1]$ because, for $x+y \in [0,2]$, the minimum is achieved either for $k=0$ or $k=-1$.
Since $x+y\ge |x-y|=a$ and $2-x-y\ge a$ for $x,y\in[0,1]$, we have $b\ge a$.
Choose $k_0\in\{0,-1\}$ such that $|x+y+2k_0|=b$. Then for every $m\in\ZZ$, 
\[
|x+y+2(k_0+m)| \ge b + 2(|m|-1)_+,
\]
and arguing as in \eqref{eq:091} (with $b$ in place of $a$) we obtain
\begin{equation}\label{eq:xyplus-part-T}
\sum_{k\in\ZZ} p_t(x+y+2k)
\le
\frac{C_T}{\sqrt{t}}\,\exp\!\left(-\frac{b^2}{4t}\right)
\le
\frac{C_T}{\sqrt{t}}\,\exp\!\left(-\frac{a^2}{4t}\right),
\qquad 0<t\le T,
\end{equation}
with the same constant $C_T$ as in \eqref{eq:091}.
Combining \eqref{eq:images-Neumann}, \eqref{eq:091}, and \eqref{eq:xyplus-part-T} yields
\[
G_t(x,y)\le 2 \frac{C_T}{\sqrt{t}}\,\exp\!\left(-\frac{a^2}{4t}\right),
\qquad 0<t\le T,
\]
which proves the inequality in \eqref{eq:G-lower-upper-T}.

For the estimate in \eqref{eq:G-deriv-x-upper-T}, recall again that the series in \eqref{eq:images-Neumann} is absolutely and locally uniformly convergent for each $t>0$,
so we may differentiate termwise:
\[
\partial_x G_t(x,y)
=
\sum_{n\in\ZZ}\left(\partial_x p_t(x-y+2n)+\partial_x p_t(x+y+2n)\right).
\]
Moreover,
\begin{equation} \label{eq:0980}
   \partial_x p_t(z)= -\frac{z}{2t}\,p_t(z), \quad |\partial_x p_t(z)|
=
\frac{|z|}{2t}\,(4\pi t)^{-1/2}e^{-z^2/(4t)}
\le
C\,t^{-1}\,e^{-z^2/(8t)}, 
\end{equation}
where we used the pointwise inequality $|z|e^{-z^2/(4t)}\le C\sqrt{t}\,e^{-z^2/(8t)}$ for $z\in\RR,\ t>0$. Therefore, combining the last two equations, 
\[
|\partial_x G_t(x,y)|
\le
C\,t^{-1}\sum_{n\in\ZZ}\left(
e^{-\frac{(x-y+2n)^2}{8t}}
+
e^{-\frac{(x+y+2n)^2}{8t}}
\right).
\]
Analogous calculations leading to \eqref{eq:091} and \eqref{eq:xyplus-part-T} respectively give that, for $0<t\le T$,
\begin{equation} \label{eq:4567}
    \sum_{n\in\ZZ} e^{-\frac{(x-y+2n)^2}{8t}}
\le
C_T\,e^{-\frac{a^2}{8t}}, \quad \sum_{n\in\ZZ} e^{-\frac{(x+y+2n)^2}{8t}} \le 
C_T\,e^{-\frac{a^2}{8t}}.
\end{equation}
Putting these estimates together yields
\[
|\partial_x G_t(x,y)|
\le
C_T\,t^{-1}\exp\!\left(-\frac{|x-y|^2}{8t}\right),
\qquad x,y\in[0,1],\ 0<t\le T,
\]
which is the inequality in \eqref{eq:G-deriv-x-upper-T}.

For \eqref{eq:Dirichlet-tder-upper-T}, we again use absolute and locally uniform convergence to differentiate termwise:
\[
\partial_t G_t (x,y)
=
\sum_{n\in\ZZ}\left(\partial_t p_t(x-y+2n)+\partial_t p_t(x+y+2n)\right).
\]
Arguing as in \eqref{eq:0980}, a direct computation gives
\begin{equation}\label{eq:gt-t-der-bound-D-2}
\partial_t p_t(z)=\left(-\frac{1}{2t}+\frac{z^2}{4t^2}\right)p_t(z), \quad 
|\partial_t p_t(z)|
\le
C\,t^{-3/2}e^{-z^2/(8t)},
\end{equation}
where we used the pointwise bound $|z|^2 e^{-z^2/(4t)}\le C t\,e^{-z^2/(8t)}$ for $z\in\RR,\ t>0$. Therefore, from \eqref{eq:4567}, for $x,y\in[0,1],\ 0<t\le T, $
\[
|\partial_t G_t(x,y)|
\le
C\,t^{-3/2}\sum_{n\in\ZZ}\left(
e^{-\frac{(x-y+2n)^2}{8t}}
+
e^{-\frac{(x+y+2n)^2}{8t}}
\right) \le
C_T\,t^{-3/2}\exp\!\left(-\frac{|x-y|^2}{8t}\right),
\]
which is \eqref{eq:Dirichlet-tder-upper-T}.
\end{proof}

In the next lemma we collect some elementary $L^2$ estimates for the kernels $G^{\operatorname{D}}_\tau, G^N_\tau, p_\tau $ for small $\tau > 0$. The estimates in (i) and (ii) are direct corollaries of the form of the heat kernel $p$ and the representations in \eqref{eq:images-Neumann}, so we omit the proofs.

\begin{lemma}\label{lem:G-L2}
\begin{enumerate}
\renewcommand{\labelenumi}{(\roman{enumi})}

\item For the heat kernel $p_\tau(x)$ on $\RR$, we have that 
\begin{equation}\label{eq:p-L2}
\int_0^1 p^2_{\tau}(x,y) \,dy \;\le\; C \,\tau^{-1/2},
\end{equation}
where $C $ can be chosen as $\frac{1}{2\sqrt{2\pi}}$. 

\item Let $G$ denote either $G^{\operatorname{D}}$ or $G^{\operatorname{N}}$. For all $\tau\in(0,T]$ and all $x\in[0,1]$,
\begin{equation}\label{eq:G-L2}
\int_0^1 G^2_{\tau}(x,y) \,dy \;\le\; C_T \tau^{-1/2}.
\end{equation}
\item \label{lem:L2lower}
Take $x(t)=t^{\theta}$ with $0 < \theta< \frac{1}{2}$. There exist constants $c>0$ and $t_0\in(0,1/16)$ such that for all
$0<r\le t\le t_0$,
\begin{equation}\label{eq:L2lower-claim}
\int_0^1 G_r^{\operatorname{D}}(x(t),y)^2\,dy \ge c\,r^{-1/2}.
\end{equation}
\end{enumerate}
\end{lemma}

\begin{proof}[Proof of item (iii)]
Fix $r\in(0,t] \le 1/16$ and set $N:=\lfloor r^{-1/2}\rfloor\ge 4$. Using \eqref{eq:def-GD} and the orthogonality of $\{ \sin(n \pi y) \}$ in $L^2(0,1)$, we have
\begin{equation}\label{eq:L2lower-step1}
\int_0^1 G_r^{\operatorname{D}}(x,y)^2\,dy
=2\sum_{n=1}^\infty e^{-2n^2\pi^2 r}\sin^2(n\pi x) \ge 2e^{-2\pi^2}\sum_{n=1}^N \sin^2(n\pi x) = 2e^{-2\pi^2}\left(\frac{N}{2}-\frac12\sum_{n=1}^N \cos(2n\pi x)\right),
\end{equation}
where the first equality follows by taking the $L^2(0,1)$ norm of $\{\sin(n\pi \cdot)\}$. Moreover, the cosine sum has the closed form
\begin{equation}\label{eq:cos-sum-bound}
    \sum_{n=1}^N \cos(2n\pi x)
=\frac{\sin(N\pi x)\cos((N+1)\pi x)}{\sin(\pi x)}, \quad \text{hence} \; \Big|\sum_{n=1}^N \cos(2n\pi x)\Big|\le \frac{1}{|\sin(\pi x)|}.
\end{equation}
If $x\in[1/N,1/2]$, then $\sin(\pi x)\ge 2x$ and since $N \ge 4$, by plugging \eqref{eq:cos-sum-bound} into \eqref{eq:L2lower-step1},
\begin{equation}\label{eq:sin2-sum-lower}
\sum_{n=1}^N \sin^2(n\pi x)\ge \frac{N}{2}-\frac{1}{4x} \ge \frac{N}{4}.
\end{equation}
By symmetry, the same bound holds when $x\in[1/2,1 -\frac{1}{N}]$ by using $\sin(\pi x)=\sin(\pi(1-x))$.
Combining \eqref{eq:L2lower-step1} and \eqref{eq:sin2-sum-lower} yields, for $x \in [1/N, 1 - 1/N]$,
\begin{equation}\label{eq:L2lower-step2}
\int_0^1 G_r^{\operatorname{D}}(x,y)^2\,dy \ge c_0\,N \ge c_1\,r^{-1/2}.
\end{equation}

Setting $x=x(t)=t^\theta$ and since $0<r\le t\le t_0$, $N=\lfloor r^{-1/2}\rfloor$, for $r\le 1/4$ we have
$N\ge r^{-1/2}-1 \ge \frac12 r^{-1/2}$, hence
$
\frac1N \le 2\sqrt r.
$
Since $\theta < 1/2$, we can choose $t_0$ small so that for all $t\le t_0$ $t^\theta > 2 \sqrt{t}$, $t_0^{\theta}<\frac{1}{2}$ and thus
\[
x(t)=t^\theta \ge 2\sqrt t \ge 2\sqrt r \ge \frac{1}{N},
\qquad\text{and}\qquad 1-x(t)\ge \frac12\ge \frac{1}{N}.
\]
Hence $x(t)\ge 1/N$ and $1-x(t)\ge 1/N$, and \eqref{eq:L2lower-step2}
gives \eqref{eq:L2lower-claim}. 
\end{proof}

\section{Auxiliary Lemmas} \label{sec:auxiliary}

\subsection{For the case $\kappa = 0$}\label{subsec:kappa=0} In this case, the Green's functions $G^{\operatorname{D}}$ and $G^{\operatorname{N}}$ correspond to Regimes \ref{:dirichlet} and \ref{case:neumann} respectively, while the Malliavin derivative is given by \eqref{PPddfo3}, with $G^{\operatorname{D}}$ (resp. $G^{\operatorname{N}}$) replacing $\mathcal{G}$. In Section \ref{subsubsec:estimates-kernels-kappa=0}, we provide some useful estimates for $G$. In Section \ref{subsubsec:estimates-malliavin-kappa=0}, we prove some bounds for the Malliavin derivative.

\subsubsection{Estimates for the kernels} \label{subsubsec:estimates-kernels-kappa=0}

The following lemma provides us with $L^2$ estimates for space-time increments of $G_t$. The proofs of items \ref{lem:G-space-increment-L2} and \ref{lem:time-increment-weighted} of Lemma \ref{lemma:space-time-increments-G} resemble these of Lemmas B.2.1(3) and B.3.1(3) in \cite{dalangsole}. 

\begin{lemma} \label{lemma:space-time-increments-G}

Let $G$ denote either $G^{\operatorname{D}}$ or $G^{\operatorname{N}}$. The following are true.
\begin{enumerate}[label=(\roman*)]
\item \label{lem:G-space-increment-L2}
For every $\alpha\in(0,1)$ there exists $C_{\alpha,T}>0$ such that for all
$t\in(0,T]$ and all $x,\bar x\in[0,1]$,
\begin{equation}\label{eq:G-space-increment-L2}
\int_0^1 \left(G_t(x,r)-G_t(\bar x,r)\right)^2\,dr
\;\le\;
C_{\alpha, T}\,|x-\bar x|^{2\alpha}\,t^{-1/2-\alpha}.
\end{equation}

\item \label{lem:product-bound}
Let $0\le s<t\le T$ and $x,\bar x,y\in[0,1]$. For every $\alpha\in(0,1/2)$ there exists
$C_{\alpha,T}>0$ such that
\begin{equation}\label{eq:product-bound}
\int_s^t\int_0^1
\left(G_{t-\theta}(x,r)-G_{t-\theta}(\bar x,r)\right)^2\,p_{\theta-s}(y,r)^2\,dr\,d\theta
\;\le\;
C_{\alpha,T}\,|x-\bar x|^{2\alpha}\,(t-s)^{-1/2-\alpha}.
\end{equation}

\item \label{lem:time-increment-weighted}
Fix $0\le s< t,\bar t\le T$ and set $t_0:=t\wedge \bar t$.
Then for every $\beta\in(0,1/4)$ there exists $C_{\beta, T}>0$ such that,
for every $\theta \in [0,t_0)$ and $\bar x\in[0,1]$,
\begin{equation}\label{eq:weighted-time-increment}
\int_0^1
\left(G_{t-\theta}(\bar x,r)-G_{\bar t-\theta}(\bar x,r)\right)^2\,
\,dr
\;\le\;
C_{\beta, T}\,|t-\bar t|^{2\beta}\,(t_0-\theta)^{-1/2-2\beta}.
\end{equation}

\item \label{lem:time-increment-product-correct}
Let $0\le s<t,\bar t\le T$ and set $t_0:=t\wedge \bar t$.
Then for every $\beta\in(0,1/4)$ there exists $C_{\beta,T}>0$ such that for all
$\bar x,y\in[0,1]$,
\begin{equation}\label{eq:time-increment-product-correct}
\int_s^{t_0}\int_0^1
\left(G_{t-\theta}(\bar x,r)-G_{\bar t-\theta}(\bar x,r)\right)^2\,
p_{\theta-s}(y,r)^2\,dr\,d\theta
\;\le\;
C_{\beta,T}\,|t-\bar t|^{2\beta}\,(t_0-s)^{-1/2-2\beta}.
\end{equation}

\end{enumerate}
\end{lemma}

\begin{proof}
Since the proofs for the two kernels are very similar, we prove the lemma for $G = G^{\operatorname{N}}$ (denoted simply by $G$ for the rest of the proof). At the end of each item, we highlight the differences for the proof when $G = G^{\operatorname{D}}$.

\textit{Proof of item \emph{(i)}.}
Using \eqref{eq:neumann-kernel} and the orthogonality
\(
\int_0^1 \cos(k\pi r)\cos(\ell\pi r)\,dr = \frac12\,\mathbf 1_{\{k=\ell\}}
\)
for $k,\ell\ge 1$, we can write
\[
\int_0^1 (G_t(x,r)-G_t(\bar x,r))^2\,dr
=
2\sum_{k\ge1} e^{-2k^2\pi^2 t}\left(\cos(k\pi x)-\cos(k\pi\bar x)\right)^2.
\]
Since $|\cos a-\cos b|\le 2 \min(1,|a-b|)$ and $\min(1,u)\le u^\alpha$ for $\alpha\in[0,1]$,
we obtain
\begin{equation} \label{eq:cos-sin-1}
    \left(\cos(k\pi x)-\cos(k\pi\bar x)\right)^2
\le (2 \min(1,|k\pi x - k\pi \bar{x}|))^2 \le 4 k^{2\alpha} \pi^{2\alpha} |x-\bar x|^{2\alpha}.
\end{equation}
Therefore,
\begin{flalign} \label{eq:0087}
\int_0^1 (G_t(x,r)-G_t(\bar x,r))^2\,dr
&\le
C_\alpha \,|x-\bar x|^{2\alpha}\sum_{k\ge1} k^{2\alpha} e^{-2k^2\pi^2 t} \notag\\
&\le C_\alpha \,|x-\bar x|^{2\alpha} \int_{0}^{\infty} (u+1)^{2\alpha} e^{-2 \pi^2 t u^{2}}\,du
\notag\\
&\le C_\alpha \,|x-\bar x|^{2\alpha} 2^{2\alpha}\int_{0}^{\infty}\left(u^{2\alpha}+1\right)e^{-2 \pi^2 t u^{2}}\,du
\notag\\
&= C_\alpha \,|x-\bar x|^{2\alpha} \left(t^{-\alpha-\frac12}\int_{0}^{\infty}u^{2\alpha}e^{-2 \pi^2 u^{2}}\,du
\;+\;t^{-\frac12}\int_{0}^{\infty}e^{-2 \pi^2 u^{2}}\,du\right)
\notag\\
&\le C_\alpha \,|x-\bar x|^{2\alpha}\,\left(t^{-\alpha-\frac12}+t^{-\frac12}\right)
\notag\\
&\le C_{\alpha,T}\,|x-\bar x|^{2\alpha}\,t^{-\alpha-\frac12}.
\end{flalign}
This
yields \eqref{eq:G-space-increment-L2}.

For $G^{\operatorname{D}}$, the same proof holds with $\sin$ replacing $\cos$. In particular, \eqref{eq:cos-sin-1} holds in this case as~well.

\medskip
\textit{Proof of item \emph{(ii)}.}
Let $u\doteq t-\theta\in(0,t-s)$ and $v\doteq \theta-s\in(0,t-s)$ so that $u+v=t-s$.
We claim that for all $\alpha\in(0,1/2)$, all $u\in(0,T]$, and all $r\in[0,1]$,
\begin{equation}\label{eq:holder-pointwise-gaussian}
|G_u(x,r)-G_u(\bar x,r)|
\le
C_T \,|x-\bar x|^\alpha\,u^{-(1+\alpha)/2}\left(
e^{-|x-r|^2/(8 u)}+e^{-|\bar x-r|^2/(8 u)}
\right).
\end{equation}
Indeed, if $|x-\bar x|>\sqrt{u}$, then
$|x-\bar x|^\alpha u^{-\alpha/2}\ge 1$ and the bound follows immediately from \eqref{eq:G-lower-upper-T} (for the Neumann kernel) applied to $G_u(x,r)+G_u(\bar x,r)$ (by replacing the $4u$ in the denominator of the exponent with $8u$ for an upper bound).

Assume, on the other hand, that $|x-\bar x|\le \sqrt{u}$. By the mean value theorem and \eqref{eq:G-deriv-x-upper-T}, for some $\xi = \lambda x + (1-\lambda) \bar x, \; \lambda \in (0,1)$,
\begin{equation}\label{eq:MVT}
\begin{split}
|G_u(x,r)-G_u(\bar x,r)|
&\le
|x-\bar x|\,|\partial_x G_u(\xi,r)| \\
&\le
C_T \,u^{(1-\alpha)/2}|x-\bar x|^\alpha \,u^{-1}\exp\!\left(-\frac{|\lambda x + (1-\lambda) \bar x-r|^2}{8\,u}\right) \\
&\le C_T \,|x-\bar x|^\alpha\,u^{-(1+\alpha)/2}\left(
e^{-|x-r|^2/(8u)}+e^{-|\bar x-r|^2/(8 u)}
\right) .
\end{split}
\end{equation}
To justify the last line in \eqref{eq:MVT}, note that, if $r\notin [x,\bar x]$ (or $[\bar x,x]$), then $\xi=\lambda x+(1-\lambda)\bar x$ lies between $x$ and $\bar x$, so
\[
|\xi-r|\ge \min\{|x-r|,|\bar x-r|\}, \quad \text{and therefore }e^{-|\xi-r|^2/(8u)}\le e^{-|x-r|^2/(8u)}+e^{-|\bar x-r|^2/(8u)}.
\]
If $r\in [x,\bar x]$, then $|x-r|,|\bar x-r|\le |x-\bar x|\le \sqrt u$, hence
\[
e^{-|x-r|^2/(8u)}+e^{-|\bar x-r|^2/(8u)}\ge 2e^{-1/8}\ge 1\ge e^{-|\xi-r|^2/(8u)}.
\]
This proves \eqref{eq:holder-pointwise-gaussian}.

Squaring \eqref{eq:holder-pointwise-gaussian} and multiplying by $p_v(y,r)^2$ gives
\begin{align}
\int_0^1 \left(G_u(x,r)-G_u(\bar x,r)\right)^2\,p_v(y,r)^2\,dr
&\le
C_T\,|x-\bar x|^{2\alpha}\,u^{-(1+\alpha)}\,v^{-1}
\nonumber\\
&\quad\times
\int_0^1
\left(
e^{-|x-r|^2/(4 u)}+e^{-|\bar x-r|^2/(4 u)}
\right)
e^{-|y-r|^2/(4 v)}\,dr\nonumber \\
&\le C_T \,|x-\bar x|^{2\alpha}\,u^{-1/2-\alpha}\,v^{-1/2}\,(u+v)^{-1/2}. 
\label{eq:two-gaussians-integral}
\end{align}
For the last line \eqref{eq:two-gaussians-integral}, we have used the Gaussian estimate
\begin{equation} \label{eq:gaus-estim}
    \int_0^1 e^{-|x-r|^2/(c u)}e^{-|y-r|^2/(c v)}\,dr
\;\le\;
C\,\sqrt{\frac{uv}{u+v}},
\end{equation}
and the respective bound with $x$ replaced by $\bar x$.

Returning to $\theta$ (so $u=t-\theta$, $v=\theta-s$, $u+v=t-s$), recasting \eqref{eq:two-gaussians-integral} and integrating over $\theta\in(s,t)$,
\begin{flalign}
\int_s^t \int_0^1 \left(G_{t-\theta}(x,r)-G_{t-\theta}(\bar x,r)\right)^2&\,p_{\theta-s}(y,r)^2\,dr  d \theta  
\nonumber\\
&\le \int_s^t C_T \,|x-\bar x|^{2\alpha}\,(t-s)^{-1/2}\,(t-\theta)^{-1/2-\alpha}\,(\theta-s)^{-1/2} d \theta 
\nonumber\\
&\le C_T \,|x-\bar x|^{2\alpha}\,(t-s)^{-1/2} \int_s^t (t-\theta)^{-1/2-\alpha}(\theta-s)^{-1/2}\,d\theta 
\nonumber\\
&\le C_T \,|x-\bar x|^{2\alpha}\,(t-s)^{-1/2 - \alpha} B\left(\frac12-\alpha,\frac12\right)\nonumber\\
&= C_{\alpha,T} |x-\bar x|^{2\alpha}\,(t-s)^{-1/2 - \alpha}, \label{eq:Beta-example}
\end{flalign}
proving \eqref{eq:product-bound}.

We remark that, for $G^{\operatorname{D}}$, the same proof holds by using the respective estimates in Lemma \ref{lem:Neumann-Gaussian-|x-y|} for the kernel $G^{\operatorname{D}}$.

\medskip
\textit{Proof of item \emph{(iii)}.}
For $\theta<t_0$ we have $t-\theta>0$ and $\bar t-\theta>0$, and thus \eqref{eq:neumann-kernel} gives
\begin{equation} \label{eq:diff-g}
G_{t-\theta}(\bar x,r)-G_{\bar t-\theta}(\bar x,r)
=
2\sum_{k\ge1}\left(e^{-k^2\pi^2 (t-\theta)}-e^{-k^2\pi^2 (\bar t-\theta)}\right)
\cos(k\pi \bar x)\cos(k\pi r).
\end{equation}
By orthogonality in $r$, for $\theta\in(s,t_0)$,
\begin{flalign}\label{eq:ortho-time}
\|G_{t-\theta}(\bar x,r)-G_{\bar t-\theta}(\bar x,r) \|^2_{L^2([0,1])} &=  \int_0^1 \left(G_{t-\theta}(\bar x,r)-G_{\bar t-\theta}(\bar x,r)\right)^2\,dr
\nonumber\\
&\le
2\sum_{k\ge1}
\left(e^{-k^2\pi^2 (t-\theta)}-e^{-k^2\pi^2 (\bar t-\theta)}\right)^2 \nonumber\\
&\le C_\beta\,|t-\bar t|^{2\beta}\,\sum_{k\ge1} k^{4\beta}\,e^{-2k^2\pi^2 (t_0-\theta)} \nonumber\\
&\le C_{\beta,T}\,|t-\bar t|^{2\beta}\,(t_0-\theta)^{-1/2-2\beta}.
\end{flalign}
For the third line, we have used the fact that, for fixed $\beta\in(0,1)$, $\lambda\ge0$, and $a,b>0$,
\begin{equation}\label{eq:frac-exp-diff}
\big|e^{-\lambda a}-e^{-\lambda b}\big|
\le
C_\beta\,|a-b|^\beta\,\lambda^\beta\,e^{-\lambda(a\wedge b)},
\end{equation}
and we have applied this estimate with $\lambda=k^2\pi^2$, $a=t-\theta$, $b=\bar t-\theta$. The fourth line follows by $\sum_{k\ge1} k^{4\beta}e^{-c k^2 \tau}\le C\,\tau^{-1/2-2\beta}$ for $\tau\in(0,T]$, which in turn is a result of identical calculations
that we used to bound the series in the first line of \eqref{eq:0087} (with $4\beta$ replacing $2\alpha$). 

For $G^{\operatorname{D}}$, the same proof holds with $\sin$ replacing $\cos$. This completes the proof of (iii).

\medskip
\textit{Proof of item \emph{(iv)}.}
\medskip
Fix $\theta\in(s,t_0)$ and set $u:=t-\theta, \bar u:= \bar t-\theta,\tau:=u\wedge \bar u=t_0-\theta, v:= \theta-s$. We first treat the case $|u-\bar u|\le \tau$. By the mean value theorem,
\[
\big|G_u(\bar x,r)-G_{\bar u}(\bar x,r)\big|
\le
|u-\bar u|\sup_{w\in[\tau,\,2\tau]}\big|\partial_w G_w(\bar x,r)\big|.
\]
Using \eqref{eq:Dirichlet-tder-upper-T}, 
\[
|\partial_w G_w(\bar x,r)|
\le
C_T\,w^{-3/2}\exp\!\left(-\frac{|\bar x-r|^2}{8w}\right)
\le
C_T\,\tau^{-3/2}\exp\!\left(-\frac{|\bar x-r|^2}{16\tau}\right),
\]
since $w\in[\tau,2\tau]$. Hence
\begin{equation}\label{eq:small-time-diff-pointwise-beta}
\begin{split}
\big|G_u(\bar x,r)-G_{\bar u}(\bar x,r)\big|
&\le
C_T\,|u-\bar u|\,\tau^{-3/2}\exp\!\left(-\frac{|\bar x-r|^2}{16\tau}\right) \\
&\le C_T\,|u-\bar u|^\beta\,\tau^{-1/2-\beta}
\exp\!\left(-\frac{|\bar x-r|^2}{16\tau}\right),
\end{split}
\end{equation}
where we used that, since $|u-\bar u|\le \tau$, then for every $\beta\in(0,1)$, we have $|u-\bar u|\,\tau^{-3/2}
\le
|u-\bar u|^\beta\,\tau^{-1/2-\beta}$.

Squaring \eqref{eq:small-time-diff-pointwise-beta}, multiplying by $p_v(y-r)^2$, and integrating in $r$, we get
\begin{align}
\int_0^1 \big(G_u(\bar x,r)-G_{\bar u}(\bar x,r)\big)^2 p_v(y-r)^2\,dr
&\le
C_T\,|u-\bar u|^{2\beta}\tau^{-1-2\beta}
\int_0^1 e^{-|\bar x-r|^2/(8\tau)}\,p_v(y-r)^2\,dr \nonumber\\
&\le
C_T\,|u-\bar u|^{2\beta}\tau^{-1-2\beta}v^{-1}
\int_0^1 e^{-|\bar x-r|^2/(8\tau)}e^{-|y-r|^2/(4v)}\,dr \nonumber\\
&\le
C_T\,|u-\bar u|^{2\beta}\tau^{-1-2\beta}v^{-1}
\int_0^1 e^{-|\bar x-r|^2/(8\tau)}e^{-|y-r|^2/(8v)}\,dr \nonumber\\
&\le
C_T\,|u-\bar u|^{2\beta}\tau^{-1/2-2\beta}(\tau+v)^{-1/2}v^{-1/2},
\label{eq:small-case-weighted}
\end{align}
where in the last line we used \eqref{eq:gaus-estim}.

We now consider the case $|u-\bar u|>\tau$. Using the bound
\[
\big(G_u(\bar x,r)-G_{\bar u}(\bar x,r)\big)^2
\le
2G_u(\bar x,r)^2+2G_{\bar u}(\bar x,r)^2,
\]
we get that
\begin{multline}
\int_0^1 \big(G_u(\bar x,r)-G_{\bar u}(\bar x,r)\big)^2 p_v(y-r)^2\,dr \\ 
\le
2\int_0^1 G_u(\bar x,r)^2 p_v(y-r)^2\,dr
+
2\int_0^1 G_{\bar u}(\bar x,r)^2 p_v(y-r)^2\,dr.
\label{eq:large-case-split}
\end{multline}
We estimate the first term in the second line in \eqref{eq:large-case-split}; the second is identical. By \eqref{eq:G-lower-upper-T},
\[
G_u(\bar x,r)
\le
C_T\,u^{-1/2}\exp\!\left(-\frac{|\bar x-r|^2}{8u}\right),
\]
and since $p_v(y-r)^2\le C\,v^{-1}e^{-|y-r|^2/(4v)}$, it follows that
\begin{align}
\int_0^1 G_u(\bar x,r)^2 p_v(y-r)^2\,dr
&\le
C_T\,u^{-1}v^{-1}\int_0^1
e^{-|\bar x-r|^2/(4u)}e^{-|y-r|^2/(4v)}\,dr \nonumber\\
&\le
C_T\,u^{-1}v^{-1}\int_0^1
e^{-|\bar x-r|^2/(8u)}e^{-|y-r|^2/(8v)}\,dr \nonumber\\
&\le
C_T\,u^{-1/2}(u+v)^{-1/2}v^{-1/2},
\label{eq:large-case-one-term}
\end{align}
where again we used \eqref{eq:gaus-estim}. Since $u\ge \tau$ and $|u-\bar u|>\tau$, we have $u^{-1/2}\le \tau^{-1/2}\le |u-\bar u|^{2\beta}\tau^{-1/2-2\beta}.$
Using also that $u+v=t-s\ge t_0-s$, \eqref{eq:large-case-one-term} yields
\[
\int_0^1 G_u(\bar x,r)^2 p_v(y-r)^2\,dr
\le
C_T\,|u-\bar u|^{2\beta}\tau^{-1/2-2\beta}(t_0-s)^{-1/2}v^{-1/2}.
\]
The same bound holds with $\bar u$ in place of $u$. Thus, by \eqref{eq:large-case-split},
\begin{equation}\label{eq:large-case-weighted}
\int_0^1 \big(G_u(\bar x,r)-G_{\bar u}(\bar x,r)\big)^2 p_v(y-r)^2\,dr
\le
C_T\,|u-\bar u|^{2\beta}\tau^{-1/2-2\beta}(t_0-s)^{-1/2}v^{-1/2}.
\end{equation}

Combining \eqref{eq:small-case-weighted} and \eqref{eq:large-case-weighted}, and recalling that $u-\bar u=t-\bar t,\qquad \tau=t_0-\theta,\qquad v=\theta-s$, we obtain for all $\theta\in(s,t_0)$ that
\begin{flalign*}
\int_s^{t_0}\int_0^1
\big(G_{t-\theta}(\bar x,r)-&G_{\bar t-\theta}(\bar x,r)\big)^2
p_{\theta-s}(y-r)^2\,dr\,d\theta \\
&\le
C_T\,|t-\bar t|^{2\beta}(t_0-s)^{-1/2}
\int_s^{t_0}
(\theta-s)^{-1/2}(t_0-\theta)^{-1/2-2\beta}\,d\theta \\ 
&=
C_T\,|t-\bar t|^{2\beta}(t_0-s)^{-1/2-2\beta}
B\!\left(\frac12,\frac12-2\beta\right) \\
&= C_{\beta,T}\,|t-\bar t|^{2\beta}\,(t_0-s)^{-1/2-2\beta}.
\label{eq:final-iv-beta}
\end{flalign*}
For the third line above, we used that $\beta<1/4$, hence the Beta function is finite, proving~\eqref{eq:time-increment-product-correct}.

For $G^{\operatorname{D}}$, the same proof applies using the corresponding bounds from Lemma~\ref{lem:Neumann-Gaussian-|x-y|}. This completes the proof. \end{proof}

The following is an immediate consequence of Lemma \ref{lemma:space-time-increments-G}.

\begin{corollary} \label{cor:space-time-diff}
For any $\alpha \in (0,1/2)$, $0  \le t \le \bar t \le T$, and $x, \bar x \in [0,1]$, there exists $C_{\alpha,T} > 0$ such that
\begin{equation*}  
    \int_{0}^{t} \int_{0}^{1}|G_{t-s}(x, y)-G_{\bar{t}-s}(\bar{x}, y)|^{2} \mathrm{~d} y \mathrm{~d} s \le C_{\alpha,T} \left( |\bar{x}-x|^{2\alpha}+|\bar{t}-t|^{\alpha} \right) .
\end{equation*}
\end{corollary}

\begin{proof}
From \eqref{eq:G-space-increment-L2}--\eqref{eq:weighted-time-increment}, note that, for $\alpha \in (0,1/2), \beta \in (0,1/4)$
\begin{equation} \label{eq:estimate-i1}
\begin{split}
\int_0^t \int_0^1 \left(G_{t-s}(x,r)-G_{t-s}(\bar x,r)\right)^2\,dr ds
&\;\le\;
C_{\alpha,T} \,|x-\bar x|^{2\alpha}\,\int_0^t (t-s)^{-1/2-\alpha} ds \\
&\;\le C_{\alpha,T} |x-\bar x|^{2\alpha}, \\
\int_0^t \int_0^1
\left(G_{t-s}(\bar x,r)-G_{\bar t-s}(\bar x,r)\right)^2\,
\,dr \,ds
&\;\le\;
C_{\beta,T} \,|t-\bar t|^{2\beta}\int_0^t \,(t_0-s)^{-1/2-2\beta} ds \\
&\;\le C_{\beta,T} \,|t-\bar t|^{2\beta}.
\end{split}
\end{equation}
Moreover, the triangle inequality for the $L^{2}(dy ds)$ norm gives us that 
\begin{flalign*}
\|G_{t-s}(x,y)-&G_{\bar t-s}(\bar x,y)\|_{L^2(d y d s )} 
\\
& =\|\left(G_{t-s}(x,y)-G_{t-s}(\bar{x},y)\right)+ \left(G_{t-s}(\bar{x},y)-G_{\bar{t}-s}(\bar{x},y)\right) \|_{L^2(dy ds )} 
\\
& \leq \|G_{t-s}(x,y)-G_{t-s}(\bar{x},y)\|_{L^2(d y ds)} +\|G_{ t-s}(\bar{x},y)-G_{\bar{t}-s}(\bar{x},y) \|_{L^2(dy ds )}, 
\end{flalign*}
and thus squaring this inequality and leveraging \eqref{eq:estimate-i1} with $\alpha = 2 \beta \in (0,1/2)$,
\begin{flalign*}
\int_0^t \int_0^1&|G_{t-s}(x, y)-G_{\bar{t}-s}(\bar{x}, y)|^{2} \mathrm{d} y \mathrm{d}s =  \|G_{t-s}(x,y)-G_{\bar t-s}(\bar x,y)\|^2_{L^2(dyds )}  \\
&\leq C \left(\|G_{t-s}(x,y)-G_{t-s}(\bar{x},y)\|^{2}_{L^2(d y ds)} +\|G_{ t-s}(\bar{x},y)-G_{\bar{t}-s}(\bar{x},y)\|^2_{L^2(dy ds )}\right) \\
&\leq C_{\alpha,T} \left( |\bar{x}-x|^{2\alpha}+|\bar{t}-t|^{\alpha} \right) .
\end{flalign*}
This concludes the proof.
\end{proof}

\subsubsection{Estimates for the Malliavin derivative} \label{subsubsec:estimates-malliavin-kappa=0}

The following key estimate provides a pointwise bound for the $k$-th moment of the Malliavin derivative in terms of the heat kernel on $\RR$. The proof follows  a standard strategy (see, e.g., Theorem~6.4 in \cite{ChenKhoshnevisanNualartPu2021EJP690}), with three differences: 
(i) our model includes a drift term $b$, 
(ii) it is driven by space--time white noise in spatial dimension $d=1$, and 
(iii) it is posed on the compact spatial interval $[0,1]$ with Dirichlet or Neumann heat kernel.

Since we work on a finite time horizon, both the Neumann and Dirichlet heat kernels on $[0,1]$ are bounded above by a constant multiple of the heat kernel on $\RR$ (see, e.g., Lemma~\ref{lem:Neumann-Gaussian-|x-y|} for the Neumann case and \cite[(2.86)]{Nua06book} for the Dirichlet case). For the Dirichlet kernel, one can take this multiplicative constant equal to $1$. We prove the estimate using the Dirichlet kernel; the Neumann case follows by the same calculations.

For all $T>0$ and $k\ge 2$, set
\begin{align}\label{eq:CTk-gamma}
C_{T,k} 
&\doteq 
\sup_{n\ge 0}\ \sup_{(t,x)\in[0,T]\times[0,1]}
\left(\mathbb{E} |\sigma\left(u_n(t,x)\right) |^{k}\right)^{\frac{1}{k}}
\;<\;\infty,
\nonumber\\
\gamma 
&\doteq 6\left(T\,\Lip(b)^2+\left(z_k\Lip(\sigma)\right)^2\right),
\end{align}
where $u_n$ denotes the $n$-th Picard iterate, and $\{z_k\}_{k\ge 2}$ are the constants in the BDG inequality, satisfying
\begin{equation}\label{eq:zk}
z_2=1
\qquad\text{and}\qquad
z_k\le 2\sqrt{k}
\qquad\text{for every }k>2.
\end{equation}
Notice that under our assumption of $\sigma$ being upper bounded we have $C_{T,k} \leq \sup_{u \in \mathbb{R}} \sigma(u).$

Finally, define a sequence of functions $\{h_n\}_{n\ge 0}$ by $h_0(t)\equiv 1$ and
\begin{equation}\label{eq:hn}
h_n(t) 
\doteq 
\int_0^t h_{n-1}(s)\,p_{t-s}(0)\,ds
=
\frac{2^{-n}\,t^{n/2}}{\Gamma(\frac n2+1)},
\qquad t>0,\ n\ge 1.
\end{equation}
The explicit expression follows from the identity
\[
\int_0^t s^{\alpha-1}(t-s)^{-1/2}\,ds
=
t^{\alpha-1/2}\,
\frac{\Gamma(\alpha)\Gamma(\tfrac12)}{\Gamma(\alpha+\tfrac12)},
\qquad \alpha>0,
\]
and the fact that $p_{t}(0)=(4\pi t)^{-1/2}$. By induction, $h_n$ is nondecreasing for all $n\ge 0$.

We also use the following lemma (verbatim Lemma~6.6 in \cite{ChenKhoshnevisanNualartPu2021EJP690}), rewritten for the heat kernel associated with $\Delta$.

\begin{lemma}\label{lem:inequality-g}
Fix $T>0$. For every nondecreasing function $g:[0,T]\to\RR_+$ and all $t\in(0,T)$ and $y\in\RR $,
\[
\int_0^t g(s)\,p_{\,s(t-s)/t}(y)\,ds
\le
2^{3/2}\int_0^t g(s)\,p_{t-s}(y)\,ds.
\]
\end{lemma}

The following theorem plays a crucial role in providing bounds for the Malliavin derivative in the sequel.

\begin{theorem}\label{thm:malliavin-derivative}
Let $u$ denote the mild solution to \eqref{sCH}, and let $G_{t}$ denote either $G_{t}^{\operatorname{D}}$ or $G_{t}^{\operatorname{N}}$. Then,
\[
u(t,x)\in\bigcap_{k\ge 2}\mathbb{D}^{1,k}
\qquad\text{for every }(t,x)\in(0,T)\times[0,1].
\]
Moreover, for fixed $T>0$, if $t\in(0,T)$ and $x\in[0,1]$, then
\begin{equation}\label{eq:Lp_norm}
\E\!\left(\left|D_{s,y}u(t,x)\right|^{k}\right)^{1/k}
\le
\sqrt{6}\,C_{T,k}\,
\exp\!\left(\gamma^2 t/2\right)\,
p_{t-s}(x-y),
\end{equation}
for almost every $(s,y)\in(0,t)\times[0,1]$,
where $C_{T,k}$ and $\gamma$ are defined in \eqref{eq:CTk-gamma} and $z_k$ in \eqref{eq:zk}.
\end{theorem}

\begin{proof} Let $u_0(t\,,x) \doteq \int_0^1 G_{t}(x,y) u_0(y) dy$ for all $t>0$ and $x\in [0,1]$, and introduce the Picard iterations:
\begin{flalign*}
u_{n+1}(t\,,x) :=& u_0(t\,,x) + \int_{[0,t]\times [0,1]} G_{t-s} (x,y) b(u_n(s\,,y)) \d s\,\d y \\
&+ \int_{[0,t]\times [0,1]} G_{t-s} (x,y) \sigma(u_n(s\,,y)) \,W(\d s\,\d y),
\end{flalign*}
for all $n\ge0$, $t>0$, and $x\in[0,1]$. 
	
We claim that $u_{n}(t\,,x)\in \mathbb{D}^{1,k}$ for every $(t\,,x) \in (0\,,T)\times [0,1]$ and $k\ge 2$ and that, moreover,
	\begin{equation} \label{e2}
		\E\!\left(\left| D_{s,y} u_{n}(t\,,x) \right|^{k}\right)^{1/k} \le \sqrt{3} C_{T,k}  \,
		p_{t-s}(x-y)\left(\sum_{i = 0}^{n}\gamma^ih_i(t - s)\right)^{1/2},
	\end{equation}
	for almost every $(s\,,y) \in (0\,,t) \times [0,1]$. We proceed with induction. The left-hand side of \eqref{e2} is equal to zero, so the statement in \eqref{e2} is true for $n= 0$ since the process $u_0(t,x)$ is deterministic.
	Now suppose \eqref{e2} holds for  $n$, and we derive its validity for $n+1$. 
	
As in the proof of Proposition 2.4.4 of Nualart \cite{Nua06book}, $b(u_n(t\,,x)), \sigma(u_n(t\,,x))\in\mathbb{D}^{1,k}$  for every $(t\,,x) \in(0,T)\times [0,1]$;
moreover,
\[
D( \sigma( u_n(t\,,x))) = \Sigma_n Du_n(t\,,x), \quad 		D( b( u_n(t\,,x))) = B_n Du_n(t\,,x), \qquad\text{a.s.,}
\]
where $\Sigma_n,B_n$ are adapted processes (on $[0,T] \times [0,1]$), uniformly bounded by the Lipschitz constants of $\sigma$ and $b$ respectively.
    
Furthermore, we have the identity 
	\begin{align*}
		D_{r,z} u_{n+1}(t\,,x) 
		 &= G_{t-r} (x,z) \sigma(u_n(r\,,z)) + \int_{(r,t)\times [0,1]} G_{t-s} (x,y)
		 	\Sigma_n  D_{r,z} u_n(s\,,y )\,W(\d s \,\d y) \\
      &\quad + \int_{(r,t)\times [0,1]} G_{t-s} (x,y)
		 	B_n  D_{r,z} u_n(s\,,y )\, \d s \,\d y
			\qquad\text{a.s.,}
	 \end{align*}
	 and so 
	 \begin{align*}
		 & \E\!\left(\left|   D_{r,z} u_{n+1}(t\,,x)  \right|^{k}\right)^{1/k} \le
			G_{t-r} (x,z)  \E\!\left(\left| \sigma(u_n(r\,,z))  \right|^{k}\right)^{1/k} \\
		&\qquad \qquad + \E\!\left(\left|  \int_{(r,t)\times [0,1]} G_{t-s} (x,y)
		 	B_n  D_{r,z} u_n(s\,,y )\, \d s \,\d y \right|^{k}\right)^{1/k} \\
        &\qquad \qquad+ \E\!\left(\left|  \int_{(r,t)\times [0,1]}
			G_{t-s} (x,y)  \Sigma_n  D_{r,z} u_n(s\,,y )\,W(\d s \,\d y)
			\right|^{k}\right)^{1/k},
	\end{align*}
	for every integer $k\ge2$. We deal first with the stochastic integral term. The Burkholder--Davis--Gundy and Minkowski integral inequality imply that
	\begin{align*}
		& \E \left( \left|  \int_{(r,t)\times [0,1]} G_{t-s} (x,y)
			\Sigma_n  D_{r,z} u_n(s\,,y )\,W(\d s \,\d y)\right |^k \right)  \\
		& \le \left[z_k{\rm Lip} (\sigma)\right]^k  \E \left( \left|
			\int_r^t  \d s \int_{0}^1 \d y \
			G^2_{t-s}(x,y)
			|D_{r,z} u_n(s\,,y)|^2  \right| ^{k/2} \right) \\
        &\le \left[ z_k {\rm Lip} (\sigma)\right]^k   \left[  \int_r^t 
			\int_{0}^1   
			G^2_{t-s}(x,y)
			\E\!\left(\left| D_{r,z} u_n(s\,,y) \right|^{k}\right)^{2/k} \d y  \d s \right] ^{k/2}.
	\end{align*}
    In addition, the Minkowski integral inequality yields
    \begin{multline*}
    \E\!\left(\left|  \int_{(r,t)\times [0,1]} G_{t-s} (x,y)
		 	B_n  D_{r,z} u_n(s\,,y ) \, \d s \, \d y
			\right|^{k}\right)^{1/k}  \\
            \le {\rm Lip} (b)  \int_r^t 
			\int_{0}^1   
			G_{t-s}(x,y)
			\E\!\left(\left| D_{r,z} u_n(s\,,y) \right|^{k}\right)^{1/k} \d y \, \d s.
    \end{multline*}
The last three displays give the following inequality on the Malliavin derivative of
	$u_{n+1}(t\,,x)$:
\begin{align}
\E\!\left(\left|D_{r,z} u_{n+1}(t,x)\right|^{k}\right)^{1/k}
&\le C_{T,k}\,G_{t-r}(x,z) \nonumber\\
&\quad + {\rm Lip}(b)\int_r^t \int_{0}^1
G_{t-s}(x,y)\, \E\!\left(\left|D_{r,z}u_n(s,y)\right|^{k}\right)^{1/k}\,\mathrm{d}y\,\mathrm{d}s \nonumber\\
&\quad + z_k\,{\rm Lip}(\sigma)\left[
\int_r^t \int_0^1
G_{t-s}^2(x,y)\,\E\!\left(\left|D_{r,z}u_n(s,y)\right|^{k}\right)^{2/k}\,\mathrm{d}y\,\mathrm{d}s
\right]^{1/2}. \label{eq:555}
\end{align}
Since \eqref{e2} holds for $n$, we have
\begin{equation}\label{W}
\begin{split}
&\int_r^t\int_0^1
G_{t-s}(x,y)^2\,\E\!\left(\left|D_{r,z}u_n(s,y)\right|^{k}\right)^{2/k}\,dy\,ds \\
&\le 3C_{T,k}^2\sum_{i=0}^{n}\gamma^i
\int_r^t h_i(s-r)\int_0^1 p_{t-s}(x-y)^2\,p_{s-r}(y-z)^2\,dy\,ds \\
&= 3C_{T,k}^2\,p_{t-r}(x-z)^2\sum_{i=0}^{n}\gamma^i
\int_r^t h_i(s-r)\int_0^1
p_{\frac{(s-r)(t-s)}{t-r}} \!\left(y-z-\frac{s-r}{t-r}(x-z)\right)^{\!2}\,dy\,ds \\
&\le 3C_{T,k}^2\,p_{t-r}(x-z)^2\sum_{i=0}^{n}\gamma^i
\int_r^t h_i(s-r)\int_{\RR}
p_{\frac{(s-r)(t-s)}{t-r}} \!\left(y-z-\frac{s-r}{t-r}(x-z)\right)^{\!2}\,dy\,ds \\
&= 3C_{T,k}^2\,p_{t-r}(x-z)^2\sum_{i=0}^{n}\gamma^i
\int_r^t h_i(s-r)\,p_{2(s-r)(t-s)/(t-r)}(0)\,ds \\
&= 3C_{T,k}^2\,p_{t-r}(x-z)^2\sum_{i=0}^{n}\gamma^i
\int_0^{t-r} h_i(s)\,p_{2s(t-r-s)/(t-r)}(0)\,ds.
\end{split}
\end{equation}
In the first inequality we used the pointwise bound $G_{t-s}(x,y)\le p_{t-s}(x-y)$, together with \eqref{e2}.
For the third line we used the following identity: for $0<a<b$ and $x,y,z\in\RR$,
\[
p_{b-a}(x-y)\,p_a(y-z)
=
p_b(x-z)\,p_{a(b-a)/b}\!\left(y-z-\frac{a}{b}(x-z)\right).
\]
The fourth line follows by extending the integral from $[0,1]$ to $\RR$.
For the fifth line we used that, for every $\tau>0$,
\[
\int_{\RR} p_\tau(x)^2\,dx
=
\int_{\RR} p_\tau(x)\,p_\tau(-x)\,dx
=
(p_\tau * p_\tau)(0)
=
p_{2\tau}(0),
\]
by symmetry and the semigroup (convolution) property of the heat kernel. The last line
follows from the change of variables $s\mapsto s-r$.

Next, since each $h_i$ is nondecreasing, Lemma~\ref{lem:inequality-g} (with $y=0$)
implies, for $\tau=t-r$,
\begin{equation}\label{eq:1001}
\begin{split}
\int_0^{\tau} h_i(s)\,p_{2s(\tau-s)/\tau}(0)\,ds
&=2^{-1/2}\int_0^{\tau} h_i(s)\,p_{s(\tau-s)/\tau}(0)\,ds \\
&\le 2^{-1/2}\cdot 2^{3/2}\int_0^{\tau} h_i(s)\,p_{\tau-s}(0)\,ds \\
&= 2\int_0^{\tau} h_i(s)\,p_{\tau-s}(0)\,ds
=2\,h_{i+1}(\tau),
\end{split}
\end{equation}
where we used the heat kernel scaling $p_{2a}(0)=2^{-1/2}p_a(0)$ and the recursive definition of $h_i$ in \eqref{eq:hn}.
Combining \eqref{W} and \eqref{eq:1001} yields
\begin{equation}\label{eq:1002}
\int_r^t\int_0^1
G_{t-s}(x,y)^2\,\E\!\left(\left|D_{r,z}u_n(s,y)\right|^{k}\right)^{2/k}\,dy\,ds
\le
6\,C_{T,k}^2\,p_{t-r}(x-z)^2\sum_{i=0}^{n}\gamma^i\,h_{i+1}(t-r).
\end{equation}

For the drift term, Minkowski's inequality together with Cauchy--Schwarz gives
\begin{flalign}\label{eq:1003}
\E\!&\left(\left|
\int_{(r,t)\times[0,1]} G_{t-s}(x,y)\,B_n(s,y)\,D_{r,z}u_n(s,y)\,ds\,dy
\right|^{k}\right)^{2/k} \nonumber\\
&\le \Lip(b)^2\left(\int_r^t\int_0^1 G_{t-s}(x,y)\,\E\!\left(\left|D_{r,z}u_n(s,y)\right|^{k}\right)^{1/k}\,dy\,ds\right)^{\!2} \nonumber\\
&\le (t-r)\,\Lip(b)^2\int_r^t\int_0^1
G_{t-s}(x,y)^2\,\E\!\left(\left|D_{r,z}u_n(s,y)\right|^{k}\right)^{2/k}\,dy\,ds \nonumber\\
&\le 6\,T\,\Lip(b)^2\,C_{T,k}^2\,p_{t-r}(x-z)^2
\sum_{i=0}^{n}\gamma^i\,h_{i+1}(t-r),
\end{flalign}
where we used $(t-r)\le T$ and \eqref{eq:1002} in the last step.

Finally, using $(a+b+c)^2\le 3a^2+3b^2+3c^2$ for $a,b,c\ge 0$ and \eqref{eq:555},
together with \eqref{eq:1002}--\eqref{eq:1003}, we obtain
\begin{align*}
\E\!\left(\left|D_{r,z}u_{n+1}(t,x)\right|^{k}\right)^{2/k}
&\le 3C_{T,k}^2\,p_{t-r}(x-z)^2
+ 3\E\!\left(\left|\int_{(r,t)\times[0,1]} G_{t-s}(x,y)\,B_n\,D_{r,z}u_n\,ds\,dy\right|^{k}\right)^{2/k} \\
&\quad + 3\left(z_k\Lip(\sigma)\right)^2
\int_r^t\int_0^1 G_{t-s}(x,y)^2\,\E\!\left(\left|D_{r,z}u_n(s,y)\right|^{k}\right)^{2/k}\,dy\,ds \\
&\le 3C_{T,k}^2\,p_{t-r}(x-z)^2 \\
&\quad + 18\,T\,\Lip(b)^2\,C_{T,k}^2\,p_{t-r}(x-z)^2
\sum_{i=0}^{n}\gamma^i\,h_{i+1}(t-r) \\
&\quad + 18\left(z_k\Lip(\sigma)\right)^2 C_{T,k}^2\,p_{t-r}(x-z)^2
\sum_{i=0}^{n}\gamma^i\,h_{i+1}(t-r).
\end{align*}
More precisely,
\begin{equation}
\begin{split}
\E\!\left(\left|D_{r,z}u_{n+1}(t,x)\right|^{k}\right)^{2/k}
&\le 3C_{T,k}^2\,p_{t-r}(x-z)^2\\
&\qquad \times 
\left(1 + 6\left(T\,\Lip(b)^2+\left(z_k\Lip(\sigma)\right)^2\right)
\sum_{i=0}^{n}\gamma^i\,h_{i+1}(t-r)\right) \\
& = 3C_{T,k}^2\,p_{t-r}(x-z)^2
\left(1 + \sum_{i=0}^{n}\gamma^{i+1}\,h_{i+1}(t-r)\right) \\
&= 3C_{T,k}^2\,p_{t-r}(x-z)^2
\sum_{i=0}^{n+1}\gamma^{i}\,h_{i}(t-r),
\end{split}
\end{equation}
which is precisely \eqref{e2} for $n+1$.

Moreover, for $\gamma\ge0$ fixed, we can derive the uniform in $n$ bound 
\begin{equation} \label{eq:bound-gamma-hn} 
\sum_{n=0}^\infty \gamma^n h_n(t)\le 2\exp\!\left(\frac{\gamma^2 t}{4}\right). 
\end{equation} 
To obtain the bound in \eqref{eq:bound-gamma-hn}, write $a=\gamma\sqrt t/2$ and by using the form of $h_n$ in \eqref{eq:hn}, 
\[ \sum_{n=0}^\infty \gamma^n h_n(t) =\sum_{n=0}^\infty \frac{a^n}{\Gamma\!\left(\frac n2+1\right)} =\sum_{m=0}^\infty \frac{a^{2m}}{m!} +\sum_{m=0}^\infty \frac{a^{2m+1}}{\Gamma(m+3/2)} =e^{a^2}+e^{a^2}\operatorname{erf}(a), 
\]
where the odd-series identity is standard and can be verified by differentiating with respect to $a$ and using the value at $a=0$. Since $0\le \operatorname{erf}(a)\le 1$ for $a\ge0$, the bound follows. 

Finally, it remains to show that the estimate \eqref{eq:Lp_norm} holds for $u(t,x)$,
for fixed $(t,x)\in(0,T)\times[0,1]$. Fix $(t,x)$ and $k\ge 2$. By \eqref{e2} and the bound \eqref{eq:bound-gamma-hn},
\begin{equation}\label{eq:uniform-Dun}
\sup_{n\ge 0}\E\left(\|Du_n(t,x)\|_{\mathcal H}^k\right)<\infty,
\qquad \mathcal H \doteq L^2([0,T]\times[0,1]).
\end{equation}
Since $u_n(t,x)\to u(t,x)$ in $L^k(\Omega)$, the closability of the Malliavin derivative
(e.g.\ \cite[Lemma~1.5.3]{Nua06book}) implies that $u(t,x)\in\mathbb{D}^{1,k}$.
Moreover, extracting a subsequence if necessary, 
\begin{equation}\label{eq:weak-Dun}
Du_n(t,x)\rightarrow Du(t,x)\quad\text{weakly in }L^k(\Omega;\mathcal H).
\end{equation}

Let $\{\psi_\varepsilon\}_{\varepsilon>0}\subset C_c^\infty(\RR^2)$ be a standard approximation of the identity:
$\psi_\varepsilon\ge 0$ and $\int_{\RR^2}\psi_\varepsilon=1$.
For $(s,y)\in(0,t)\times(0,1)$ and $\varepsilon>0$, define
\[
D^\varepsilon_{s,y}u(t,x)
\doteq
\int_0^t\int_0^1 D_{s',y'}u(t,x)\,\psi_\varepsilon(s'-s,y'-y)\,dy'\,ds'
=
\big\langle Du(t,x),\psi_\varepsilon(\cdot-s,\cdot-y)\big\rangle_{\mathcal H}.
\]
For each fixed $\varepsilon>0$ and $(s,y)$, H\"older's inequality yields
\begin{equation}\label{eq:mollifier-duality}
\E\!\left(\left|D^\varepsilon_{s,y}u(t,x)\right|^{k}\right)^{1/k}
=
\sup_{\E\!\left(\left|F\right|^{k/(k-1)}\right)^{(k-1)/k}\le 1}
\left|\E\!\left[F\,D^\varepsilon_{s,y}u(t,x)\right]\right|.
\end{equation}
Fix $F\in L^{k/(k-1)}(\Omega)$ with $\E\!\left(\left|F\right|^{k/(k-1)}\right)^{(k-1)/k}\le 1$.
By \eqref{eq:weak-Dun}, for every $\varepsilon>0$,
\begin{equation}\label{eq:765}
\E\!\left[F\,D^\varepsilon_{s,y}u(t,x)\right]
=
\lim_{n\to\infty}
\E\!\left[F\,D^\varepsilon_{s,y}u_n(t,x)\right].
\end{equation}
Moreover, using \eqref{e2} and the fact that $\psi_\varepsilon$ has unit mass, we obtain
\begin{align}\label{eq:098}
\left|\E\!\left[F\,D^\varepsilon_{s,y}u_n(t,x)\right]\right|
&\le \E\!\left(\left|D^\varepsilon_{s,y}u_n(t,x)\right|^k\right)^{1/k} \nonumber\\
&\le \int_0^t\int_0^1 \E\!\left(\left|D_{s',y'}u_n(t,x)\right|^k\right)^{1/k}\,\psi_\varepsilon(s'-s,y'-y)\,dy'\,ds' \nonumber\\
&\le \sqrt{3}\,C_{T,k}\int_0^t\int_0^1
p_{t-s'}(x-y')\left(\sum_{i=0}^\infty \gamma^i h_i(t-s')\right)^{1/2}
\,\psi_\varepsilon(s'-s,y'-y)\,dy'\,ds'.
\end{align}
 Hence, letting $n\to\infty$ in \eqref{eq:098}
and using \eqref{eq:765}, we obtain for every $\varepsilon>0$,
\[
\left|\E\!\left[F\,D^\varepsilon_{s,y}u(t,x)\right]\right|
\le
\sqrt{3}\,C_{T,k}\,
\left(\psi_\varepsilon * \Phi\right)(s,y),
\quad \text{ where }
\Phi(s',y')\doteq p_{t-s'}(x-y')\left(\sum_{i=0}^\infty \gamma^i h_i(t-s')\right)^{1/2}.
\]

Taking the supremum over $\E\!\left(\left|F\right|^{k/(k-1)}\right)^{(k-1)/k}\le 1$ and using \eqref{eq:mollifier-duality} gives
\[
\E\!\left(\left|D^\varepsilon_{s,y}u(t,x)\right|^{k}\right)^{1/k}
\le
\sqrt{3}\,C_{T,k}\,(\psi_\varepsilon * \Phi)(s,y).
\]
Since $\Phi$
is continuous on $(0,t)\times(0,1)$, the convolution with $\psi_\varepsilon$ converges pointwise to its value
at $(s,y)$ as $\varepsilon\downarrow 0$.
Finally, along a sequence $\varepsilon_m\downarrow 0$, we have
$D^{\varepsilon_m}_{s,y}u(t,x)\to D_{s,y}u(t,x)$ for a.e.\ $(s,y, \omega)$
(after choosing a representative of $Du(t,x)$), and $(\psi_{\varepsilon_m}*\Phi)(s,y)\to \Phi(s,y)$.
Therefore, by Fatou's lemma,
\[
\E\!\left(\left|D_{s,y}u(t,x)\right|^{k}\right)^{1/k}
\le
\sqrt{3}\,C_{T,k}\,
p_{t-s}(x-y)\left(\sum_{i=0}^\infty \gamma^i h_i(t-s)\right)^{1/2},
\qquad\text{for a.e.\ }(s,y)\in(0,t)\times(0,1).
\]
Applying \eqref{eq:bound-gamma-hn} yields \eqref{eq:Lp_norm}, completing the proof.
\end{proof}

\subsection{For the case $\kappa > 0$} \label{subsec:kappa>0}

In this case, the Green's function $H$ is given by \eqref{eq:H-def}, while the Malliavin derivative is still given by \eqref{PPddfo3}, with $H$ replacing $\mathcal{G}$. In Section \ref{subsubsec:green-H}, we provide some useful estimates for $H$. In Section \ref{subsubsec:Malliavin-H}, we prove some bounds for the Malliavin derivative. 

\subsubsection{Estimates for the kernel} \label{subsubsec:green-H}
The following lemmas give some $L^\infty$, $L^2$ and pointwise estimates for $H$.

\begin{lemma} \label{pointwiseboundonH}
For every $T>0$ there exist constants $C,c>0$ such that for all
$0<t\le T$ and all $x,y\in[0,1]$,
\[
|H_t(x,y)|
\le
C_{T,\rho} t^{-1/4}
\left(
e^{-c|x-y|^{4/3}/t^{1/3}}
+
e^{-c\,m(x,y)^{4/3}/t^{1/3}}
\right),
\]
where
\begin{equation} \label{def:m}
    m(x,y):=\min(x+y,\,2-x-y).
\end{equation}
\end{lemma}

\begin{proof}
From \eqref{eq:images-Neumann}, 
\begin{equation} \label{eq:meow-1}
H_t(x,y) = \sum_{m \in \ZZ} \left( g_t^\rho(x-y + 2m )+g_t^\rho (x+y + 2m)\right)
\end{equation}
where recall that for $t>0$ and $z\in\mathbb R$, $g_t^\rho$ was defined in \eqref{eq:gtrho}. Moreover, equation (1.2) in \cite{BARBATISDAVIES} for $m=2,N=1$ says that 
\begin{equation} \label{eq:meow19}
    |g_t^\rho(z)|
\le
C_{T,\rho} t^{-1/4}\exp\!\left(-c\frac{|z|^{4/3}}{t^{1/3}}\right),
\qquad z\in\mathbb R,\ \ 0<t\le T.
\end{equation}

We now obtain an upper bound for $\sum_{m\in\mathbb Z}\exp\!\left(-c\frac{|s+2m|^{4/3}}{t^{1/3}}\right)$.
Since the sum is $2$-periodic in $s$, it suffices to consider $s\in[-1,1]$.
Then $\operatorname{dist}(s,2\mathbb Z)=|s|$, and
\begin{align}
\sum_{m\in\mathbb Z}\exp\!\left(-c\frac{|s+2m|^{4/3}}{t^{1/3}}\right)
&=
\exp\!\left(-c\frac{|s|^{4/3}}{t^{1/3}}\right) \nonumber \\
&\quad+
\sum_{n=1}^\infty \exp\!\left(-c\frac{(2n-s)^{4/3}}{t^{1/3}}\right)
+
\sum_{n=1}^\infty \exp\!\left(-c\frac{(2n+s)^{4/3}}{t^{1/3}}\right) \nonumber \\
&\le C_T \exp\!\left(-c\frac{|s|^{4/3}}{t^{1/3}}\right)
\left(
1+\sum_{n=2}^\infty \exp\!\left(-c_2\frac{n^{4/3}}{t^{1/3}}\right)
\right)  \label{eq:meow1}\\
&\le
C_T \exp\!\left(-c\frac{|s|^{4/3}}{t^{1/3}}\right) \nonumber \\
&=  C_T \exp\!\left(-c\frac{\operatorname{dist}(s,2\mathbb Z)^{4/3}}{t^{1/3}}\right).  \label{eq:meow2}
\end{align}
For \eqref{eq:meow1}, we used that, for $s\in[-1,1]$ and $n\ge2$, $2n\pm s\ge n$ and so $(2n\pm s)^{4/3}\ge |s|^{4/3}+c_1 n^{4/3}$, where we can take e.g., $c_1:= \frac{3^{4/3} - 1}{2^{4/3}}$ and $c_2 := c c_1$. After decreasing $c_1>0$ if necessary, we have 
\begin{align*}
\exp\!\left(-c\frac{(2n\pm s)^{4/3}}{t^{1/3}}\right)
&\le
\exp\!\left(-c\frac{|s|^{4/3}}{t^{1/3}}\right)
\exp\!\left(-c c_1\frac{n^{4/3}}{t^{1/3}}\right).
\end{align*}

Finally, returning to $H_t$, \eqref{eq:meow-1}, \eqref{eq:meow19}, and \eqref{eq:meow2} say that
\begin{align}
|H_t(x,y)|
&\le
C_{T,\rho} t^{-1/4}
\left(
\exp\!\left(-c\frac{\operatorname{dist}(x-y,2\mathbb Z)^{4/3}}{t^{1/3}}\right)
+
\exp\!\left(-c\frac{\operatorname{dist}(x+y,2\mathbb Z)^{4/3}}{t^{1/3}}\right)
\right) \nonumber\\
&\le C_{T,\rho} t^{-1/4}
\left(
e^{-c|x-y|^{4/3}/t^{1/3}}
+
e^{-c\,m(x,y)^{4/3}/t^{1/3}}
\right). \label{eq:mew5}
\end{align}
where $m$ was defined in \eqref{def:m}.
The inequality in \eqref{eq:mew5} follows because $x,y\in[0,1]$ implies that $\operatorname{dist}(x-y,2\mathbb Z)=|x-y|$ and $\operatorname{dist}(x+y,2\mathbb Z)=\min(x+y,2-x-y)$. This finishes the proof.
\end{proof}

\begin{corollary} \label{4thgreenbounds}
There exists a constant $C_T>0$
such that, uniformly in $x$ and for all $0<t \le T$,
\begin{equation}
\|H_t(x,\cdot)\|_{L^\infty(0,1)}\le C_T\,t^{-1/4},
\qquad
\|H_t(x,\cdot)\|_{L^2(0,1)}\le C_T \,t^{-1/8}.
\end{equation}
\end{corollary}

\begin{proof}
The $\| \cdot \|_{L^\infty}$ bound follows immediately from the pointwise bounds given in Lemma \ref{pointwiseboundonH}.

For the $L^2$ bound, we one can also use the pointwise bound; here we use Parseval's inequality in the $y$-variable to obtain
\[
\|H_t(x,\cdot)\|_{L^2(0,1)}^2 \le 1+2\sum_{k=1}^\infty e^{-2\pi^4 k^4 t} \le 1 + 2\int_0^\infty e^{-2\pi^4 t u^4}\,du = 1 + C t^{-1/4} \le C_T t^{-1/4}.
\]

\end{proof}

\begin{lemma} \label{integralboundholderforH}
Let $\alpha > -1$, $r_0>0$, and $x^*\in[0,1]$. Then for $0<t\le T$,
\begin{equation} \label{eq:Ht-x*-y-bound}
    \int_{|y-x^*|\le r_0} |H_t(x^*,y)|\,|x^*-y|^{\alpha+1}\,dy
\le
C_{T,\rho,\alpha,r_0} \, t^{(\alpha+1)/4}.
\end{equation}
\end{lemma}

\begin{proof}
Using the pointwise estimate of Lemma \ref{pointwiseboundonH} (see also \eqref{def:m}),
\[
\int_{|y-x^*|\le r_0} |H_t(x^*,y)|\,|x^*-y|^{\alpha+1}\,dy
\le C_{T,\rho}(A_1 + A_2),
\]
where
\[
A_1 :=
\int_{|y-x^*|\le r_0}
t^{-1/4} e^{-c|x^*-y|^{4/3}/t^{1/3}}
|x^*-y|^{\alpha+1}\,dy,
\]
\[
A_2 :=
\int_{|y-x^*|\le r_0}
t^{-1/4} e^{-c\,m(x^*,y)^{4/3}/t^{1/3}}
|x^*-y|^{\alpha+1}\,dy.
\]

We first estimate $A_1$. We have that
\begin{align*}
    A_1
&=
\int_{|z|\le r_0}
t^{-1/4} e^{-c|z|^{4/3}/t^{1/3}} |z|^{\alpha+1}\,dz \\
&= t^{(\alpha+1)/4}
\int_{|u|\le r_0 t^{-1/4}}
e^{-c|u|^{4/3}} |u|^{\alpha+1}\,du \\
&\le
C_{\alpha, r_0} t^{(\alpha+1)/4},
\end{align*}
where the first line follows by setting $z = y - x^*$ and the second by setting $z = t^{1/4}u$. The integrability follows since $\alpha>-1$.

For $A_2$ we consider separately three cases. First, take $x^*\in[r_0,1-r_0]$.
Then $m(x^*,y)\ge \delta>0$ for $|y-x^*|\le r_0$, hence
\[
A_2 \le t^{-1/4} e^{-c\delta^{4/3}/t^{1/3}}
\int_{|y-x^*|\le r_0} |x^*-y|^{\alpha+1}\,dy
\le C_{T,\alpha, r_0} t^{(\alpha+1)/4}.
\]

Second, let $x^*<r_0$. Then $m(x^*,y)=x^*+y$ and $|x^*-y|\le x^*+y$, so
\begin{align}
    A_2
&\le
\int_{|y-x^*|\le r_0}
t^{-1/4} e^{-c(x^*+y)^{4/3}/t^{1/3}}
(x^*+y)^{\alpha+1}\,dy \nonumber \\
&\le
\int_0^{3r_0}
t^{-1/4} e^{-c u^{4/3}/t^{1/3}} u^{\alpha+1}\,du \nonumber \\
&\le  C_\alpha t^{(\alpha+1)/4}.
\end{align}
The second line follows by setting $u=x^*+y$, and the third by $v=t^{1/4}u$. The last case $x^*>1-r_0$ follows similarly as above, by taking $m(x^*,y)=2-x^*-y$ and $|x^*-y|\le 2-x^*-y$, so the same argument as above applies.

Combining the estimates above yields \eqref{eq:Ht-x*-y-bound}.
\end{proof}

The following lemma on the increments of $H$ is also needed. This is an easy modification of Lemma 1.8 from \cite{weber} where it is shown for the Green's function of the bi-Laplacian. 
\begin{lemma} \label{greenincrementsintegral}
For any $0<\gamma<\frac{3}{4}$, there exists a constant $C_{\gamma,\rho}$ such that 
$$
\begin{aligned}
& \int_{0}^{t} \int_{0}^{1} |H_{t-s}(x,y)-H_{\bar{t}-s}(\bar{x},y)|^{2} dy ds  \leq  C_{\gamma,\rho} \left(|t-\bar{t}|^{\gamma}+|x-\bar{x}|^{2}\right).
\end{aligned}
$$
\end{lemma}
\begin{proof}
 To prove these inequalities, we use the series decomposition of $H_{t}(x,y)$, in particular the orthonormality of $\sqrt{2}\cos(k \pi x)$ in $[0,1]$. Starting with the spatial increment, we have 
$$
\begin{aligned}
 \int_{0}^{t} \int_{0}^{1}&|H_{t-s}(x, y)-H_{t-s}(\bar{x}, y)|^{2} dy ds \\
&= \int_{0}^{t} \int_{0}^{1}\left|\sum_{k=1} ^{\infty} e^{-\left(k^4\pi^4+\rho k^2 \pi^2\right) (t-s)} \cos( k \pi y)\left[\cos(k \pi x) -\cos( k \pi \bar{x})\right]\right|^{2} d y ds \\
& =\int_{0}^{t} \sum_{k=1} ^{\infty} e^{-2\left(k^4\pi^4+\rho k^2 \pi^2\right) (t-s)} \left[\cos(k \pi x) -\cos( k \pi \bar{x})\right]^{2} ds\\
&=\sum_{k=1}^{\infty} \left|\cos( k \pi x)-\cos( k \pi \bar{x})\right|^{2} \int_{0}^{t} e^{- 2\left(k^4\pi^4+\rho k^2 \pi^2\right) (t-s)}  \mathrm{~d} s   \\
& =\sum_{k=1}^{\infty} \left|\cos( k \pi x)-\cos( k \pi \bar{x})\right|^{2} \frac{\left[1-e^{-2 \left(k^4\pi^4+\rho k^2 \pi^2\right) t}\right]}{\left(k^4\pi^4+\rho k^2 \pi^2\right)} .
\end{aligned}
$$
Note that, since $ k \neq 0$, we have that
$|\cos(k\pi x)-\cos(k \pi\bar{x})| \leq C \min \{k|x-\bar{x}|,1\}$
for a constant $C$ that doesn't depend on $k$.
Thus,
\begin{flalign}
\int_0^t \int_0^1 |H_{t-s}(x, y)-H_{t-s}(\bar{x}, y)|^{2} \mathrm{~d} y \mathrm{~d} s 
&\leq C |x-\bar{x}|^{2} \sum_{k=1}^{\infty}  \left[1-e^{-2 \left(k^4\pi^4+\rho k^2 \pi^2\right) t}\right]\frac{k^2}{\left(k^4\pi^4+\rho k^2 \pi^2\right)}
\nonumber\\ &\leq  C |x-\bar{x}|^{2} \sum_{k=1}^{\infty}  \frac{1}{\left(k^2\pi^4+\rho\pi^2\right)} \leq C_\rho |x-\bar{x}|^{2}, \label{eq:H-space-L2}
\end{flalign}
where the last inequality follows from the fact that the series $\sum_{k=1}^{\infty}  \frac{1}{\left(k^2\pi^4+\rho\pi^2\right)}$ converges. For the time increment, we have
\begin{flalign*}
& \int_0^t \int_0^1|H_{t-s}(x, y)-H_{\bar{t}-s}(x, y)|^{2} dy ds = \\
&=2\int_0^t \int_0^1 \left|\sum_{k=1} ^{\infty}  e^{-\left(k^4\pi^4+\rho k^2 \pi^2\right)(t-s)} \left( 1- e^{-\left(k^4\pi^4+\rho k^2 \pi^2\right)(\bar{t}-t)}\right) \cos( k \pi y) \cos( k \pi x)\right|^{2} dy ds \\
&=2\int_{0}^{t} \sum_{k=1} ^{\infty}  e^{-2\left(k^4\pi^4+\rho k^2 \pi^2\right)(t-s)} \left( 1- e^{-\left(k^4\pi^4+\rho k^2 \pi^2\right)(\bar{t}-t)}\right)^{2}  \cos^{2}( k \pi x) ds \\
&= 2\sum_{k=1} ^{\infty}  \left( 1- e^{-\left(k^4\pi^4+\rho k^2 \pi^2\right)(\bar{t}-t)}\right)^{2}  \cos^{2}( k \pi x) \int_{0}^{t}  e^{-2\left(k^4\pi^4+\rho k^2 \pi^2\right)(t-s)} ds \\
&=2 \sum_{k=1} ^{\infty}  \left( 1- e^{-\left(k^4\pi^4+\rho k^2 \pi^2\right)(\bar{t}-t)}\right)^{2}  \cos^{2}( k \pi x) \frac{\left( 1-  e^{-2\left(k^4\pi^4+\rho k^2 \pi^2\right)t} \right)}{2\left(k^4\pi^4+\rho k^2 \pi^2\right)}.
\end{flalign*}
From $e^x \geq x+1$, we see that, since $\bar t \ge t$,
$$
\left( 1- e^{-\left(k^4\pi^4+\rho k^2 \pi^2\right)(\bar{t}-t)} \right)^{2} \leq \min \{1, \left(k^4\pi^4+\rho k^2 \pi^2\right)|\bar{t}-t|\}\leq |\bar{t}-t|^{\gamma} \left(k^4\pi^4+\rho k^2 \pi^2\right)^{\gamma}.
$$
Therefore,
\begin{flalign}
\int_0^t \int_0^1|H_{t-s}(x, y)&-H_{\bar{t}-s}(x, y)|^2 dy ds \nonumber\\
 &\leq 2\sum_{k=1} ^{\infty}  \left( 1- e^{-\left(k^4\pi^4+\rho k^2 \pi^2\right)(\bar{t}-t)}\right)^{2}  \cos^{2}( k \pi x) \frac{\left( 1-  e^{-2\left(k^4\pi^4+\rho k^2 \pi^2\right)t} \right)}{2\left(k^4\pi^4+\rho k^2 \pi^2\right)} \nonumber\\
&\leq C_\gamma |\bar{t}-t|^{\gamma}\sum_{k=1} ^{\infty} \frac{1}{\left(k^4\pi^4+\rho k^2 \pi^2\right)^{1-\gamma}}.
\end{flalign}
The series converges as long as $4(1-\gamma)>1$, or equivalently as long as $\gamma < \frac{3}{4}$. We can thus conclude that for any $0<\gamma < \frac{3}{4},$
\begin{equation} \label{eq:H-time-L2}
    \int_0^t \int_0^1 |H_{t-s}(x, y)-H_{\bar{t}-s}(x, y)|^{2} \mathrm{~d} y \mathrm{~d} s \leq C_{\gamma,\rho} |\bar{t}-t|^{\gamma} .
\end{equation}
We can now combine \eqref{eq:H-space-L2}--\eqref{eq:H-time-L2} and the triangle inequality for the $L^{2}(dy ds)$ norm to get 
\begin{equation*}
\|H_{t-s}(x,y)-H_{\bar t-s}(\bar x,y)\|_{L^2(d y d s )}  \leq \|H_{t-s}(x,r)-H_{t-s}(\bar{x},y)\|_{L^2(d y ds)} +\|H_{ t-s}(\bar{x},y)-H_{\bar{t}-s}(\bar{x},y) \|_{L^2(dy ds )} ,
\end{equation*}
and thus squaring this inequality, we get 
$$
\begin{aligned}
& \int_{0}^{t} \int_{0}^{1}|H_{t-s}(x, y)-H_{\bar{t}-s}(\bar{x}, y)|^{2} \mathrm{~d} y \mathrm{~d} s =  \|\big(H_{t-s}(x,y)-H_{\bar t-s}(\bar x,y)\big)\|^{2}_{L^2(d y d s )}  \\
&\leq C_{\gamma,\rho} \left(\|H_{t-s}(x,r)-H_{t-s}(\bar{x},y) \|^{2}_{L^2(d y ds)} +\|H_{ t-s}(\bar{x},y)-H_{\bar{t}-s}(\bar{x},y) \|^{2}_{L^2(dy ds )}\right) \\
&\leq C_{\gamma,\rho} \left( |\bar{x}-x|^{2}+|\bar{t}-t|^{\gamma} \right) .
\end{aligned}
$$
This completes the proof of Lemma \ref{greenincrementsintegral}. 
\end{proof}

The following lemma provides analogous estimates for weighted increments of the Green function.

\begin{lemma}\label{lemmagdif}
Let $0\le s<t,\bar t\le T$, $x, \bar x \in [0,1]$, and set $t_0:=t\wedge \bar t$.
There  exist constants $C_\rho,C_{\rho,\alpha}$ so that the following bounds hold for any $0<\alpha< \frac{3}{4}$
\begin{align}
    & \int_s^t \int_0^1
\big(H_{t-\theta}(x,r)-H_{t-\theta}(\bar x,r)\big)^2
\frac{1}{(\theta-s)^{\frac{1}{2}}}\,dr\,d\theta \leq C_\rho \frac{|x-\bar{x}|^{2}}{\sqrt{t-s}}, \label{desiredboundspatial} \\
&\int_s^{t_0}\int_0^1
\big(H_{t-\theta}(x,r)-H_{\bar t-\theta}(x,r)\big)^2
\frac{1}{(\theta-s)^{\frac{1}{2}}} dr d\theta \leq  C_{\rho,\alpha} \frac{ |t-\bar{t}|^{\alpha}}{\sqrt{t_0-s}}.\label{desiredtimebound}
\end{align}
\end{lemma}
\begin{proof}[Proof of Lemma \ref{lemmagdif}]
We start by proving the spatial increment bound in \eqref{desiredboundspatial}. We have 
\begin{equation}\label{eq:grdifer}
H_{t-\theta}(x,r)-H_{t-\theta}(\bar x,r)=2\sum_{k \in \mathbb{N}} e^{-(k^{4}\pi^4 +\rho k^2 \pi^2)(t-\theta)}(\cos(k \pi x)-\cos(k \pi \bar{x}))\cos(k\pi r).
\end{equation}
Using \eqref{eq:grdifer} and Parseval's identity, it follows that
\begin{flalign}
J(t,x,\bar x, s, y) & :=  \int_s^t \int_0^1
\big(H_{t-\theta}(x,r)-H_{t-\theta}(\bar x,r)\big)^2
\frac{1}{(\theta-s)^{\frac{1}{2}}}\,dr\,d\theta 
\nonumber\\
&  =   \int_s^t \frac{1}{(\theta-s)^{\frac12}}\int_0^1
\big(H_{t-\theta}(x,r)-H_{t-\theta}(\bar x,r)\big)^2
  dr\,d\theta
\nonumber\\
& \leq 4\int_s^{t} \frac{1}{(\theta-s)^{\frac{1}{2}}}\int_0^1
\left(\sum_{k \in \mathbb{N}} e^{-(k^{4}\pi^4 +\rho k^2 \pi^2)(t-\theta)}(\cos(k \pi x)-\cos(k \pi \bar{x}))\cos(k r) \right)^2 dr\,d\theta 
\nonumber\\
&=  4\int_s^{t} \frac{1}{(\theta-s)^{\frac{1}{2}}}
\sum_{k \in \mathbb{N}} e^{-2(k^{4}\pi^4 +\rho k^2 \pi^2) (t-\theta)}(\cos(k\pi x)-\cos(k \pi \bar{x}))^2 d\theta
\nonumber\\
&= 4\sum_{k \in \mathbb{N}} \big(\cos(k\pi x)-\cos(k \pi \bar{x})\big)^2 \int_s^t \frac{1}{(\theta-s)^{\frac{1}{2}}}
\, e^{-2(k^{4}\pi^4 +\rho k^2 \pi^2) (t-\theta)} \,d\theta.\label{eq:Jest}
\end{flalign}
We focus now on computing the integral term. We write, with explanations given below,
\begin{align}
    \int_s^t &\frac{1}{(\theta-s)^{\frac{1}{2}}}
e^{-2(k^{4}\pi^4 +\rho k^2 \pi^2) (t-\theta)} d\theta  = e^{-2(k^{4}\pi^4 +\rho k^2 \pi^2)(t-s)}\int_0^{t-s} \frac{1}{u^{\frac{1}{2}}}
e^{2(k^{4}\pi^4 +\rho k^2 \pi^2)u}\,du \label{eq:3455} \\
&= e^{-2(k^{4}\pi^4 +\rho k^2 \pi^2)(t-s)}\int_0^{\sqrt{2(k^{4}\pi^4 +\rho k^2 \pi^2)(t-s)}} \frac{\sqrt{2(k^{4}\pi^4 +\rho k^2 \pi^2)}}{z}
e^{z^2} \frac{z}{(k^{4}\pi^4 +\rho k^2 \pi^2)} dz \label{eq:3456}
\\
&=  e^{-2(k^{4}\pi^4 +\rho k^2 \pi^2)(t-s)}\frac{\sqrt{2}}{\sqrt{(k^{4}\pi^4 +\rho k^2 \pi^2)}}\int_{0}^{\sqrt{2(k^{4}\pi^4 +\rho k^2 \pi^2)(t-s)}}
e^{z^2} dz \nonumber\\
&= \frac{\sqrt{2}}{\sqrt{2(k^{4}\pi^4 +\rho k^2 \pi^2)}} W\big(\sqrt{2(k^{4}\pi^4 +\rho k^2 \pi^2)} \sqrt{2(t-s)}\big), \label{eq:3458}
\end{align}
where we have defined Dawson's integral 
$$
W(x):=e^{-x^2}\int_0^x e^{z^2} dz. 
$$
The equality in \eqref{eq:3455} follows by the change of variables $u=\theta-s$; \eqref{eq:3456} follows by the change of variables $z=\sqrt{2(k^{4}\pi^4 +\rho k^2 \pi^2)u}$.
By \eqref{eq:Jest} and \eqref{eq:3458}, we obtain
\begin{equation*}
J(t,x,\bar x, s)\leq 
\sqrt{2}\sum_{k \in \mathbb{N}} \frac{1}{\sqrt{2(k^{4}\pi^4 +\rho k^2 \pi^2)}} W(\sqrt{2(k^{4}\pi^4 +\rho k^2 \pi^2)} \sqrt{2(t-s)})  \big(\cos(kx)-\cos(k\bar{x})\big)^2.
\end{equation*}
We estimate Dawson's integral as follows.
For \(x>0\),
\[
W(x)=e^{-x^{2}}\!\int_{0}^{x} e^{z^{2}}\,dz
=\int_{0}^{x} e^{-(x^{2}-z^{2})}\,dz
=\int_{0}^{x} e^{-(x-z)(x+z)}\,dz 
\le \int_{0}^{x} e^{-x(x-z)}\,dz,
\]
where we used that, since \(0\le z\le x\) implies \(x+z\ge x\), we have that $e^{-(x-z)(x+z)} \le e^{-x(x-z)}$.
With the change of variables \(u=x-z\),
\begin{equation} \label{eq:W-bound}
    W(x) \le \int_{0}^{x} e^{-x(x-z)}\,dz
=\int_{x}^{0} e^{-xu}(-du)
=\int_{0}^{x} e^{-xu}\,du
=\frac{1-e^{-x^{2}}}{x}
<\frac{1}{x}.
\end{equation}
Recalling the estimate $|\cos(k \pi x)-\cos(k \pi \bar{x})| \leq C \min \{k|x-\bar{x}|,1\}$ from the proof of Lemma \ref{greenincrementsintegral}, it holds that
\begin{align*}
J(t,x,\bar x, s)
&\leq C
\sum_{k \in \mathbb{N}} \frac{1}{ 2(k^{4}\pi^4 +\rho k^2 \pi^2) \sqrt{t-s}} (\cos(k \pi x)-\cos(k \pi \bar{x}))^{2} 
\\
&\leq  C \sum_{k \in \mathbb{N}} \frac{1}{ 2(k^{4}\pi^4 +\rho k^2 \pi^2) \sqrt{t-s}} k^2 |x-\bar{x}|^2 \\
&=C  \frac{|x-\bar{x}|^{2}}{\sqrt{t-s}}  \sum_{k \in \mathbb{N}} \frac{1}{2(k^{2}\pi^4 +\rho \pi^2) } \leq  C_\rho \frac{|x-\bar{x}|^{2}}{\sqrt{t-s}} ,
\end{align*}
which finishes the proof of \eqref{desiredboundspatial}.

Now, we prove the time increment bound in \eqref{desiredtimebound}. Assume without loss of generality that $t_0 = t$. We have 
\begin{align}
H_{t-\theta}(x,r)-H_{\bar{t}-\theta}(x,r) &=
\sum_{ k \in \mathbb{N}} 2 \cos(k \pi x) \cos(k \pi r) \left( e^{-(k^{4}\pi^4 +\rho k^2 \pi^2)(t-\theta)} - e^{-(k^{4}\pi^4 +\rho k^2 \pi^2)(\bar{t}-\theta)} \right) \nonumber
\\
&= \sum_{ k \in \mathbb{N}} 2\cos(k \pi x) \cos(k \pi r) e^{-(k^{4}\pi^4 +\rho k^2 \pi^2)(t-\theta)} \left( 1- e^{-(k^{4}\pi^4 +\rho k^2 \pi^2)(\bar{t}-t)} \right), \nonumber
\end{align}
so we get 
\begin{align}
\int_0^1 &\big(H_{t-\theta}(x,r)-H_{\bar{t}-\theta}(x,r)\big)^2 dr \nonumber
\\
&= \int_0^1 4 \left( \sum_{ k \in \mathbb{N}} \cos(k \pi x) \cos(k \pi r) e^{-(k^{4}\pi^4 +\rho k^2 \pi^2)(t-\theta)} \left( 1- e^{-(k^{4}\pi^4 +\rho k^2 \pi^2)(\bar{t}-t)} \right) \right)^{2} dr \nonumber
\\
&=4\sum_{ k \in \mathbb{N}} \cos^{2}(k \pi x)e^{-2(k^{4}\pi^4 +\rho k^2 \pi^2)(t-\theta)} \left( 1- e^{-(k^{4}\pi^4 +\rho k^2 \pi^2)(\bar{t}-t)} \right)^2, \nonumber
\end{align}
where we used Parseval's identity in the last step. Moreover,
\begin{equation}
\begin{split}
&\int_s^{t} \frac{1}{(\theta-s)^{\frac{1}{2}}} \sum_{ k \in \mathbb{N}} \cos^{2}(k \pi x)e^{-2(k^{4}\pi^4 +\rho k^2 \pi^2)(t-\theta)} \left( 1- e^{-(k^{4}\pi^4 +\rho k^2 \pi^2)(\bar{t}-t)} \right)^{2}\,d\theta 
\\
&= \sum_{ k \in \mathbb{N}} \cos^{2}(k \pi x) \left( 1- e^{-(k^{4}\pi^4 +\rho k^2 \pi^2)(\bar{t}-t)} \right)^{2}\int_s^{t} \frac{1}{(\theta-s)^{\frac{1}{2}}}e^{-2(k^{4}\pi^4 +\rho k^2 \pi^2)(t-\theta)} \,d\theta \\
&=\sum_{ k \in \mathbb{N}} \cos^{2}(k \pi x) \left( 1- e^{-(k^{4}\pi^4 +\rho k^2 \pi^2)(\bar{t}-t)} \right)^{2}  \frac{1}{\sqrt{(k^{4}\pi^4 +\rho k^2 \pi^2)}} W(\sqrt{(k^{4}\pi^4 +\rho k^2 \pi^2)} \sqrt{2(t-s)}) \\
&\le \sum_{ k \in \mathbb{N}} \cos^{2}(k \pi x) \left( 1- e^{-(k^{4}\pi^4 +\rho k^2 \pi^2)(\bar{t}-t)} \right)^{2}\frac{1}{(k^{4}\pi^4 +\rho k^2 \pi^2)\sqrt{t-s}}. \label{eq:8765}
\end{split}
\end{equation}
For the third line we have used the estimate for Dawson's integral given in \eqref{eq:W-bound}.
Since all the terms are smaller than $1, $ for any $\alpha \leq 1$
$$
\left( 1- e^{-(k^{4}\pi^4 +\rho k^2 \pi^2)(\bar{t}-t)} \right)^{2} \leq |\bar{t}-t|^{\alpha} (k^{4}\pi^4 +\rho k^2 \pi^2)^{\alpha}.$$
Using \eqref{eq:8765}, we get the bound 
\begin{multline}
\int_s^{t}\int_0^1
\big(H_{t-\theta}(x,r)-H_{\bar t-\theta}(x,r)\big)^2
\frac{1}{(\theta-s)^{\frac{1}{2}}} dr d\theta \\
\leq  C |t-\bar{t}|^{\alpha} \frac{1}{\sqrt{t-s}} \sum_{ k \in \mathbb{N}} \frac{1}{(k^{4}\pi^4 +\rho k^2 \pi^2)^{1-\alpha}} \leq  C_{\rho,\alpha} \frac{ |t-\bar{t}|^{\alpha}}{\sqrt{t-s}}
\end{multline}
where we used that $0<\alpha<\frac{3}{4}$, so the series converges.
This completes the proof of Lemma \ref{lemmagdif}.
\end{proof}




The bounds from Lemma \ref{lemmagdif} can be combined to obtain, for $s  \le t \le \bar t \le T$, $x , \bar x \in [0,1]$.
\begin{eqnarray}
\int_s^{t}\int_0^1
\big(H_{t-\theta}(x,r)-H_{\bar t-\theta}(\bar{x},r)\big)^2
\frac{1}{(\theta-s)^{\frac{1}{2}}} dr d\theta \leq C_{\rho,\alpha} \frac{ \left(|t-\bar{t}|^{\alpha}+|x-\bar{x}|^2 \right)}{\sqrt{t-s}}.\label{combinedspatialtimeboundfor4th}
\end{eqnarray}

\subsubsection{Estimates for the Malliavin derivative} \label{subsubsec:Malliavin-H}

We first prove the following Gr\"onwall-type lemma. 

\begin{lemma}\label{lemmaA}
Fix $s\in[0,T)$ and suppose $A:[s,T]\to[0,\infty)$ satisfies, for all $t\in(s,T]$,
\begin{equation}\label{eq:base_volterra}
A(t)
\;\le\;
C_0\,(t-s)^{-1/2}
\;+\;
c\!\int_s^t \frac{A(\theta)}{(t-\theta)^{1/4}}\,d\theta.
\end{equation}
Then, we have $ A(t)\le C_T\,(t-s)^{-1/2}, $ for all $t\in(s,T]$.
\end{lemma}

\begin{proof}
Define
\[
Q(t):=\int_s^t \frac{A(\tilde{s})}{(t-\tilde{s})^{1/4}}\,d\tilde{s}.
\]
From \eqref{eq:base_volterra}, for every $\tilde{s}\in(s,t]$ we have
\begin{equation}\label{eq:Atildes}
A(\tilde{s})
\;\le\;
C_0\,(\tilde{s}-s)^{-1/2}
\;+\;
c\!\int_s^{\tilde{s}}\frac{A(\theta)}{(\tilde{s}-\theta)^{1/4}}\,d\theta.
\end{equation}
Multiply \eqref{eq:Atildes} by $(t-\tilde{s})^{-1/4}$ and integrate $\tilde{s}\in[s,t]$:
\begin{align}
Q(t)
&=\int_s^t \frac{A(\tilde{s})}{(t-\tilde{s})^{1/4}}\,d\tilde{s} \nonumber\\
&\le
C_0\int_s^t \frac{(\tilde{s}-s)^{-1/2}}{(t-\tilde{s})^{1/4}}\,d\tilde{s}
\;+\;
c\int_s^t \frac{1}{(t-\tilde{s})^{1/4}}
\Big(\int_s^{\tilde{s}}\frac{A(\theta)}{(\tilde{s}-\theta)^{1/4}}\,d\theta\Big)d\tilde{s}.
\label{eq:Hineq}
\end{align}
A change of variable $u=\tilde{s}-s$ gives
\begin{equation} \label{eq:beta-example-2}
    \int_s^t \frac{(\tilde{s}-s)^{-1/2}}{(t-\tilde{s})^{1/4}}\,d\tilde{s}
=\int_0^{t-s}\frac{u^{-1/2}}{(t-s-u)^{1/4}}\,du
=B\!\Big(\tfrac12,\tfrac34\Big)\,(t-s)^{1/4}.
\end{equation}
Expanding and using Fubini:
\begin{align}
\int_s^t \frac{1}{(t-\tilde{s})^{1/4}}
\Big(\int_s^{\tilde{s}}\frac{A(\theta)}{(\tilde{s}-\theta)^{1/4}}\,d\theta\Big)d\tilde{s} \nonumber
&=
\int_s^t A(\theta)\Big(\int_{\tilde{s}=\theta}^{t}
\frac{d\tilde{s}}{(t-\tilde{s})^{1/4}(\tilde{s}-\theta)^{1/4}}\Big)d\theta \nonumber\\
&=
B\!\Big(\tfrac34,\tfrac34\Big)
\int_s^t (t-\theta)^{1/2}A(\theta)\,d\theta
\;\le\;C_T\!\int_s^t A(\theta)\,d\theta. \label{eq:beta-example-2.5}
\end{align}
Inserting the bounds \eqref{eq:beta-example-2} and \eqref{eq:beta-example-2.5} into \eqref{eq:Hineq} yields
\[ Q(t) \le C_T+C_T\int_s^t A(\theta)\,d\theta. \]
Return to \eqref{eq:base_volterra} at time $t$:
\[
A(t)
\;\le\;
C_0\,(t-s)^{-1/2}
\;+\;
c\,Q(t)
\;\le\;
C_T\,(t-s)^{-1/2}
\;+\;
C_T\!\int_s^t A(\theta)\,d\theta,
\]
Then, the usual Gr\"onwall inequality with integral terms gives
$$
A(t)\le C_T\,(t-s)^{-1/2},
$$
as claimed, and the proof of Lemma \ref{lemmaA} is completed.
\end{proof}

The following is a lemma bounding the $k$-th norm of the Malliavin derivative.

\begin{lemma}\label{lmalden} There exists a constant $C_{T,k,\sigma,b}$ such that, for all $0\leq s <t \leq T$ and $ x, y \in [0,1],$
\begin{equation}\label{lpnormmalliavin} 
\mathbb{E}\left(|D_{s,y} u (t,x)|^k\right) ^{\frac{1}{k}}\leq \frac{C_{T,k,\sigma,b}}{(t-s)^{\frac{1}{4}}} .
\end{equation}
\end{lemma}
\begin{proof}

We set
\[
U(t):=\sup_{x\in[0,1]}\big(\E|D_{s,y}u(t,x)|^k\big)^{\frac{1}{k}},\qquad
M(t):=U(t)^2,
\]
and we will prove that 
\begin{equation}\label{eq:Volterra_core}
M(t)\;\le\; C_{k,b,\sigma}\,(t-s)^{-1/2} \;+\; C_{k,b,\sigma}\!\int_s^t (t-r)^{-1/4}\,M(r)\,dr,
\end{equation}
Recalling the form of the Malliavin derivative in \eqref{PPddfo3} with $H$ replacing $\mathcal{G}$, we have
\begin{equation*}
\begin{split}
| D_{s,y}u(t,x)|^{k} \leq  & C_k |H_{t-s}(x,y)\sigma(u(s,y)) |^{k}
 \;+\;
C_k \left|\int_{s}^{t}\int_{0}^{1}H_{t-\theta}(x,r)m(\theta,r)D_{s,y}u(\theta,r) dr d\theta \right|^{k} 
\\
&\;+\; C_k \left|\int_{s}^{t}\int_{0}^{1} H_{t-\theta}(x,r) \hat{m}(\theta,r)D_{s,y}u(\theta,r) W(dr,d \theta)\right|^{k} ,
\end{split} 
\end{equation*}
and so 
\begin{align}
\E(| D_{s,y}u(t,x)|^k)^{\frac{1}{k}} 
 \leq &   C_k \E \left(|H_{t-s}(x,y)\sigma(u(s,y)) |^{k}\right)^{\frac{1}{k}} \nonumber
\\
& +\,
C_k \E \left( \left|\int_s^t\int_0^1H_{t-\theta}(x,r)m(\theta,r)D_{s,y}u(\theta,r) dr d\theta \right|^{k}\right)^{\frac{1}{k}} \nonumber
\\
& +\; C_k \E \left(\left|\int_s^t\int_0^1 H_{t-\theta}(x,r) \hat{m}(\theta,r)D_{s,y}u(\theta,r) W(dr,d \theta)\right|^{k}\right)^{\frac{1}{k}}. \label{eq:M-comb-I1I2I3}
\end{align}
We estimate each of these summands in $L^k(\Omega)$ uniformly in $x$. Denoting
\begin{align*}
\mathcal{I}_1(t,x)&:=H_{t-s}(x,y)\sigma(u(s,y)),
\\
\mathcal{I}_2(t,x)&:=\int_s^t\int_0^1H_{t-\theta}(x,r)m(\theta,r)D_{s,y}u(\theta,r) dr d\theta,
\\
\mathcal{I}_3(t,x)&:=\int_s^t\int_0^1 H_{t-\theta}(x,r) \hat{m}(\theta,r)D_{s,y}u(\theta,r) W(dr,d \theta).
\end{align*}

By boundedness of $\sigma$ and Corollary \ref{4thgreenbounds},
\begin{equation}\label{eq:I1bound}
\sup_x (\E|\mathcal{I}_1(t,x)|^k)^{1/k}
\;\le\;
\|\sigma\|_\infty\,\sup_x  \|H_{t-s}(x,\cdot)\|_{L^\infty}
\;\le\;
C_{\sigma,T}\,(t-s)^{-1/4}.
\end{equation}

Using Minkowski inequality, the Cauchy–Schwarz inequality in the spatial variable, the a.s.\ boundedness of $m$, and Corollary \ref{4thgreenbounds}, we obtain
$$
\begin{aligned}
(\E|\mathcal{I}_2(t,x)|^k)^{1/k}
&\le\int_s^t\int_0^1H_{t-\theta}(x,r)\big(\E[|m(\theta,r)D_{s,y}u(\theta,r)|^k]\big)^{1/k} dr d\theta
\\
&\le
C_b \int_s^t 
\|H_{t-\theta}(x,\cdot)\|_{L^2}
\big(\sup_{r}(\E|D_{s,y}u(\theta,r)|^k)^{1/k}\big)d \theta \\
&\le
C_{b,T}
\int_s^t (t-\theta)^{-1/8}\,U(\theta)\,d\theta.
\end{aligned}
$$
Squaring both sides and taking the $\sup$ over $x$, we obtain
\begin{equation}\label{eq:I2bound}
\sup_x (\E|\mathcal{I}_2(t,x)|^k)^{2/k}
\;\le\;
C_{b,T}\left(\int_s^t (t-\theta)^{-1/8}\,U(\theta)\,d\theta\right)^2
\;\le\;
C_{b,T}\!\int_s^t (t-\theta)^{-1/4}\,U(\theta)^2\,d\theta,
\end{equation}
where the last inequality follows from Cauchy–Schwarz inequality.

We now turn to estimating $\mathcal{I}_3$. By Burkholder–Davis–Gundy inequality, Minkowski inequality and Corollary \ref{4thgreenbounds}, we have
$$
\begin{aligned}
(\E|\mathcal{I}_3(t,x)|^k)^{1/k}
&\le
C_{k} \,\|\Lip(\sigma)\|_\infty
\Bigg(\int_s^t
\|H_{t-\theta}(x,\cdot)\|_{L^2}^2
\big(\sup_r (\E|D_{s,y}u(\theta,r)|^k)^{2/k}\big)d\theta\Bigg)^{1/2}\\
&\le
C_{k,\sigma,T}
\Bigg(\int_s^t (t-\theta)^{-1/4}\,U(\theta)^2\,d\theta\Bigg)^{1/2}.
\end{aligned}
$$ Squaring yields
\begin{equation}\label{eq:I3bound}
\sup_x (\E|\mathcal{I}_3(t,x)|^k)^{2/k}
\;\le\;
C_{k,\sigma,T}\!\int_s^t (t-\theta)^{-1/4}\,U(\theta)^2\,d\theta.
\end{equation}
Plugging \eqref{eq:I1bound}, \eqref{eq:I2bound}, and \eqref{eq:I3bound} into \eqref{eq:M-comb-I1I2I3}, we obtain
\[
M(t)
=\sup_x (\E|D_{s,y}u(t,x)|^k)^{2/k}
\;\le\;
C_{\sigma,T,k} \,(t-s)^{-1/2}
\;+\;
C_{\sigma,T,k,b}\!\int_s^t (t-\theta)^{-1/4}\,M(\theta)\,d \theta,
\]
This completes the proof of \eqref{eq:Volterra_core}.

Now, we apply Lemma \ref{lemmaA} to $M(t)$ to conclude that, for all $x,y \in [0,1],$
$$
\E\left(|D_{s,y}u(t,x)|^k\right)^{\frac{1}{k}} \leq \frac{C_{T,\sigma,k,b}}{(t-s)^{\frac{1}{4}}},
$$
which completes the proof of Lemma \ref{lmalden}.
\end{proof}

\section{Proof of the Theorem for $\kappa=0$}\label{sec:proof-kappa=0}

Theorem \ref{maintheorem} is a consequence of Theorem \ref{thm:nualart-criterion-existence}, applied to the random field $u$ solving \eqref{sCH} under Regimes \ref{:dirichlet} and \ref{case:neumann}. The first condition in Theorem \ref{thm:nualart-criterion-existence} for $K = [0,T] \times [0,1]$  is well-known. The fact that $\E \left( \sup_{(t,x) \in K} |u(t,x)|^2 \right) < \infty$  follows by splitting $u(t,x)$ into its deterministic component including $u_0$ and the stochastic component including the nonlinear and stochastic terms, and obtaining the finiteness for each of these terms separately. The finiteness of the first term follows immediately by the continuity of $u_0$, while for the other terms by an application of Kolmogorov's criterion. The fact that $u(t,x) \in \mathbb{D}^{1,2}$ was shown in \cite{pardouxzhang}.

By proving the second condition in Theorem \ref{thm:nualart-criterion-existence} for $K = [0,T] \times [0,1]$, we can infer that $\sup_{(t,x) \in [0,T] \times [0,1]} u(t,x) \in \mathbb{D}^{1,2}$. This is the subject of Section \ref{subsec:second-cond-k=0}.

If conditions (i) and (ii) of Theorem \ref{thm:nualart-criterion-existence} are true for the random field $\{u(t,x): (t,x) \in [0,T] \times [0,1]\}$, they are also true for the restricted field $\{u(t,x): (t,x) \in K\}$, where $K$ is an arbitrary compact subset of $[0,T] \times [0,1]$. In particular, $ \sup_{(t,x) \in K} u(t,x) \in \DD^{1,2}$ for any compact subset $K$ of $[0,T] \times [0,1]$. Moreover, to obtain that the law of $\sup_{(t,x) \in K} u(t,x)$ is absolutely continuous with respect to the Lebesgue measure for any compact $K$ specified in the statement of Theorem \ref{maintheorem}, it suffices to check the third condition in Theorem \ref{thm:nualart-criterion-existence} for the set $K$. This is the subject of Section \ref{subsubsec:k=0sup-compacts}, while Sections \ref{subsubsec:gammaholder} and \ref{subsubsec:k=0-small-ball} give some useful estimates. 

In proving the absolute continuity of the law of $\sup_{(t,x) \in [0,T] \times [0,1]} u(t,x)$, we impose in addition Assumption \ref{ass:uo-holder}. The proof of the third condition in Theorem \ref{thm:nualart-criterion-existence} for $K = [0,T] \times [0,1]$ is conducted in Section \ref{subsubsec:k=0-whole-space}.

In this section, we will use $G$ to denote either $G^{\operatorname{D}}$ or $G^{\operatorname{N}}$. When an argument is only true or necessary for one of these two kernels, we will use the superscript to avoid notational ambiguity.

\subsection{Proof of the second condition} \label{subsec:second-cond-k=0}

To derive the second condition of Theorem \ref{thm:nualart-criterion-existence}, we must verify the continuity of the map
\begin{equation} \label{eq:process-cont}
    (t,x) \mapsto D_{\cdot,\cdot} u(t,x) \in L^2([0,T] \times [0,1]),
\end{equation}
up to modification of the process. To do so, we aim to apply Kolmogorov's criterion (Theorem \ref{thm:kolmogorov}) with $\mathbb{B} = L^2([0,T] \times [0,1])$, equipped with the usual norm of this space. More precisely, it suffices to prove that, for all $ \, t, \bar{t} \in [0,T],\; x, \bar{x}\in [0,1],$ there exist some $\lambda_1,\lambda_2 \in (0,1]$ such that, for some $p > \frac{1}{\lambda_1} + \frac{1}{\lambda_2}$,
\begin{equation} \label{eq:cond-2-kolmogorov}
\begin{split}
    \E \| D_{\cdot,\cdot} u(t,x)-  D_{\cdot,\cdot} u(\bar t,\bar x) \|_{L^2([0,T] \times [0,1])}^p &= \mathbb{E}\bigg(\bigg(\int_0^T \int_0^1|D_{s,y}u(t,x)-D_{s,y}u(\bar{t},\bar{x})|^2 dy ds \bigg)^{\frac{p}{2}}\bigg) \\
    &\leq C \left(\left|t-\bar{t}\right|^{\lambda_1} +\left|x-\bar{x}\right|^{\lambda_2}\right)^p.
\end{split}
\end{equation}
Then, there exists a $\mathbb{B}-$valued modification of the process in \eqref{eq:process-cont} that is continuous.

We suppose without loss of generality that $ \, \bar{t} < t \,$  and we have
\begin{align}\label{e26625}
D_{s,y}u(t,x)- D_{s,y}u(\bar{t}, \bar{x})
&=\left(G_{t-s}(x,y)-G_{\bar{t}-s}( \bar{x},y) \right)\sigma(u(s,y)) 
\nonumber\\
&  +\int_{s}^{t}\int_{0}^{1}\left(G_{t-\theta}(x,r)-G_{\bar{t}-\theta}( \bar{x},r)\right)m(r,\theta)D_{s,y}u(\theta,r) drd \theta 
\nonumber\\
&  +\! \int_s^{t}\!\!\!\int_{0}^{1} \left(G_{t-\theta}(x,r)-G_{\bar{t}-\theta}( \bar{x},r) \right) \hat{m}(r,\theta)D_{s,y}u(\theta,r) W(dr,d \theta),
\end{align}
\noindent
where we set $G_{t}(x,y)=0$ for $t<0$. This means, for instance, that in the interval $\theta \in (\bar{t},t)$ we have $G_{\bar{t}-\theta}(x,y)=0$ since $\bar{t}-\theta<0$ and only the term $G_{t-\theta}$ remains in the difference.

Using the fact that $\sigma$ is bounded, we write
\begin{equation}\label{choco1}
\begin{split}
 \E&\left[\left( \int_{0}^{T} \int_{0}^{1} |D_{s,y} u(t,x)- D_{s,y}u(\bar{t}, \bar{x})|^2 dy ds  \right)^{\frac{p}{2}}\right]
\\
&\le 
\;C_p  \left(\int_{0}^{T} \int_0^1 |G_{t-s}(x,y)-G_{\bar{t}-s}(\bar{x},y)|^{2} dy ds \right)^{\frac{p}{2}} 
\\
& \;+ C_p \E \left( \int_0^T \int_0^1 \left(\int_{s}^{t}\int_{0}^{1} \left( G_{t-\theta}(x,r)-G_{\bar{t}-\theta}( \bar{x},r)\right) m(r,\theta) D_{s,y}u(\theta,r) drd \theta \right)^2 dy ds \right)^{\frac{p}{2}} 
\\
& \;+ C_p \E \left(\int_0^T \int_0^1 \left(  \int_{s}^{t}\int_{0}^{1} \left(G_{t-\theta}(x,r)-G_{\bar{t}-\theta}( \bar{x},r)\right) \hat{m}(r,\theta)D_{s,y}u(\theta,r) W(dr,d \theta) \right)^{2} dy ds \right)^{\frac{p}{2}} 
\\
&=I_1 + I_2 + I_3,
\end{split}
\end{equation}
where we have set
\begin{eqnarray*}
I_1 &:=& \left(\int_{0}^{T} \int_0^1 |G_{t-s}(x,y)-G_{\bar{t}-s}(\bar{x},y)|^{2} dy ds \right)^{\frac{p}{2}} ,
\\
I_2 &:=& \E \left( \int_0^T \int_0^1 \left(\int_{s}^{t}\int_{0}^{1} \left( G_{t-\theta}(x,r)-G_{\bar{t}-\theta}( \bar{x},r)\right) m(r,\theta) D_{s,y}u(\theta,r) drd \theta \right)^2 dy ds \right)^{\frac{p}{2}} ,
\\
I_3&:=&\E \left(\int_0^T \int_0^1 \left(  \int_{s}^{t}\int_{0}^{1} \left(G_{t-\theta}(x,r)-G_{\bar{t}-\theta}( \bar{x},r)\right) \hat{m}(r,\theta)D_{s,y}u(\theta,r) W(dr,d \theta) \right)^{2} dy ds \right)^{\frac{p}{2}} .
\end{eqnarray*}

We start by providing an estimate for $I_1$. 
Using that $(a+b)^p \le 2^{p-1} (a^p + b^p)$ for $p > 1$ and Corollary \ref{cor:space-time-diff}, we can bound the term $ I_ 1 $ by 
\begin{equation}
\begin{split}
I_1=\left(\int_{0}^{T} \int_0^1 |G_{t-s}(x,y)-G_{\bar{t}-s}(\bar{x},y)|^{2} dy ds \right)^{\frac{p}{2}} 
&\leq C_{\alpha,T,p} \left( |\bar{x}-x|^{2 \alpha}+|\bar{t}-t|^{\alpha} \right)^{\frac{p} {2}} \\
&\leq C_{\alpha,T,p} \left( |\bar{x}-x|^{\alpha}+|\bar{t}-t|^{\alpha/2} \right)^{p}.\label{I_1bound}
\end{split}
\end{equation}

We proceed with bounding the term $I_2$. We start by bounding the quantity inside the $dyds$-integral. In particular, we write
\small
\begin{equation} \label{eq:I2-1}
\begin{split}
&\left(\int_{s}^{t}\int_{0}^{1} \left( G_{t-\theta}(x,r)-G_{\bar{t}-\theta}( \bar{x},r)\right) m(r,\theta) D_{s,y}u(\theta,r) drd \theta \right)^2  \\
&\qquad\leq  
 C_{b} \left(\int_{s}^{t}\int_{0}^{1}|G_{t-\theta}(x,r)-G_{\bar{t}-\theta}( \bar{x},r)|^2 dr d \theta \right) \int_{s}^{t}\int_{0}^{1}|D_{s,y}u(\theta,r)|^2 drd \theta  \\
& \qquad\leq C_{b,\alpha,T} \left(|t-\bar{t}|^{\alpha}+|x-\bar{x}|^{2\alpha}\right)\int_{s}^{t}\int_{0}^{1}|D_{s,y}u(\theta,r)|^2 drd \theta .
\end{split}
\end{equation}
For the second line we used Cauchy-Schwarz and the fact that $m(r, \theta)$ is a random variable uniformly bounded by $\Lip(b)$. For the third line, we used Corollary \ref{cor:space-time-diff} with again $\alpha \in (0,1/2)$.
From \eqref{eq:I2-1}, we can bound the term $I_2$ by
\begin{align}
I_{2} & = \E\left[\left( \int_{0}^{T} \int_{0}^{1} \left(\int_{s}^{t}\int_{0}^{1} \left( G_{t-\theta}(x,r)-G_{\bar{t}-\theta}( \bar{x},r)\right) m(r,\theta) D_{s,y}u(\theta,r) drd \theta \right)^2 dy ds \right)^{\frac{p}{2}}\right] \nonumber \\
& \leq \E\left[\left( \int_{0}^{T} \int_{0}^{1} C_{b,\alpha,T} \left(|t-\bar{t}|^{\alpha}+|x-\bar{x}|^{2\alpha}\right)\int_{s}^{t}\int_{0}^{1}|D_{s,y}u(\theta,r)|^2 drd \theta dy ds \right)^{\frac{p}{2}}\right] \nonumber \\
& \leq C_{b,\alpha,T,p} \left(|t-\bar{t}|^{\alpha}+|x-\bar{x}|^{2\alpha}\right)^{\frac{p}{2}}\E\left[\left( \int_{0}^{T} \int_{0}^{1} \int_{s}^{t}\int_{0}^{1}|D_{s,y}u(\theta,r)|^2 drd \theta dy ds \right)^{\frac{p}{2}}\right]. \label{eq:i2-2}
\end{align}
Using Minkowski integral inequality for $p\ge 2$, we get 
\begin{align}
\mathbb{E}\left[\left( \int_0^T \int_0^1 \int_s^t\int_0^1|D_{s,y}u(\theta,r)|^2 drd \theta dy ds \right)^{\frac{p}{2}}\right]  \nonumber
&\leq
\left(\int_0^T \int_0^1 \int_s^t \int_0^1 \mathbb{E}\left( |D_{s,y} u (\theta,r)|^{p} \right)^{\frac{2}{p}} dr d \theta dy ds \right)^{\frac{p}{2}}  \nonumber
\\
& \leq C_{p,T} \left(\int_0^T \int_0^1 \int_s^t \int_0^1 p^2_{\theta-s}(r,y) dr d \theta dy ds \right)^{\frac{p}{2}} \nonumber \\
& \leq C_{p,T} \left(\int_0^T \int_0^1 \int_s^t \left(\frac{1}{\theta-s}\right)^{1/2} d \theta dy ds \right)^{\frac{p}{2}} \nonumber
\\
&\leq C_{p,T} , \label{eq:i2-3}
\end{align}
The second line follows by the estimate in Theorem \ref{thm:malliavin-derivative}; the third by \eqref{eq:p-L2}. Thus, combining \eqref{eq:i2-2}--\eqref{eq:i2-3} and using again that $(a+b)^p \le 2^{p-1} (a^p + b^p)$ for $p > 1$,
\begin{align} 
I_2&\leq C_{b,\alpha,T,p} \left(|t-\bar{t}|^{\alpha/2}+|x-\bar{x}|^{\alpha }\right)^{p}.
\label{I_2bound}
\end{align}
Next we bound the term $I_3$. First we apply Minkowski integral inequality and the Burkholder-Davis-Gundy inequality respectively with $p \geq 2$ to obtain
\begin{multline}
\E
\left(\left[\int_0^{t} \int_0^1 
\left(\int_s^{t}\int_0^1 \left(G_{t-\theta}(x,r)-G_{\bar{t}-\theta}(\bar{x},r)\right) \, \hat{m}(r,\theta) \, D_{s,y}u(\theta,r) \, W(dr,d\theta)\right)^2 dy ds\right]^{\frac{p}{2}}\right)
\\
\le C_T \left(\int_0^t \int_0^1 
\left(\E\left[\left(\int_s^{t}\int_0^1 \left(G_{t-\theta}(x,r)-G_{\bar{t}-\theta}(\bar{x},r)\right) \, \hat{m}(r,\theta) \, D_{s,y}u(\theta,r) \, W(dr,d\theta)\right)^p\right]\right)^{\frac{2}{p}} dy ds \right)^{\frac{p}{2}} \\
\le C_{T,p,\sigma}
\left(\int_0^t \int_0^1
\left(\E\left|\int_s^{t}\int_0^1 \left(G_{t-\theta}(x,r)-G_{{{\bar{t}}}-\theta}(\bar{x},r)\right)^2 \,| D_{s,y}u(\theta,r)|^2 drd\theta \right|^{p/2} \right)^{\frac{2}{p}}\, dy ds \right)^{\frac{p}{2}}  \label{eq:i3-2}
\end{multline}
For the last line, we also used that $|\hat{m}(\theta,r)| \le \Lip(\sigma)$ a.s.. Moreover, by another appeal to Minkowski's integral inequality for $p \geq 2$,
\begin{multline} \label{eq:i3-3}
    \left(\E\left|\int_s^{t}\int_0^1 \left(G_{t-\theta}(x,r)-G_{{{\bar{t}}}-\theta}(\bar{x},r)\right)^2 \,| D_{s,y}u(\theta,r)|^2 drd\theta \right|^{p/2} \right)^{\frac{2}{p}} \\
    \le  \int_s^{t}\int_0^1
\left( \left(G_{t-\theta}(x,r)-G_{\bar t-\theta}(\bar x,r)\right)^p \, \E\left[\big|D_{s,y} u(\theta,r)\big|^p\right] \right)^{\frac{2}{p}}dr\,d\theta \\
 \le  C_{T,p} \int_s^{t}\int_0^1
\left(G_{t-\theta}(x,r)-G_{\bar t-\theta}(\bar x,r)\right)^2
p_{(\theta-s)}(r,y)^2\,dr\,d\theta.
\end{multline}
The last line followed once again from Lemma \ref{thm:malliavin-derivative}. By plugging \eqref{eq:i3-3} in \eqref{eq:i3-2} and the form of $I_3$,
\begin{equation}\label{eq:preICH}
\begin{split}
I_3 &\leq  C_{T,p,\sigma} \, \left(\int_0^{t}\int_0^1
\int_s^{t}\int_0^1
\left(G_{t-\theta}(x,r)-G_{\bar t-\theta}(\bar x,r)\right)^2
p_{(\theta-s)}(r,y)^2\,dr\,d\theta dy\,ds \right)^{\frac{p}{2}} \\
& \leq C_{T,p,\sigma,\alpha} \left(\left( |x-\bar{x}|^{2\alpha} + |t-\bar{t}|^{\alpha} \right) \int_0^{t}\int_0^1  (t-s)^{-\frac{1}{2}-\alpha}dy\,ds \right)^{\frac{p}{2}} \\
&\leq C_{T,p,\sigma,\alpha} \left( |x-\bar{x}|^{\alpha }+|t-\bar{t}|^{\alpha/2} \right)^{p}.
\end{split}
\end{equation}
For the second line, we used \eqref{eq:product-bound}--\eqref{eq:time-increment-product-correct} again with $\alpha = 2\beta \in (0,1/2)$; the third line follows again by the fact that $(a+b)^p \le 2^{p-1} (a^p + b^p)$, and since $\int_0^t (t-s)^{-\frac{1}{2}-\alpha} ds \leq C_{T}$ for $\alpha \in (0,1/2)$.

Combining the bounds we obtained for $I_1,I_2,I_3$ in respectively \eqref{I_1bound}, \eqref{I_2bound}, and \eqref{eq:preICH}, we conclude using \eqref{choco1} that
\begin{equation} \label{eq:malliavin-difference-L2}
 \mathbb{E}\left[\left( \int_{0}^{T} \int_{0}^{1} |D_{s,y} u(t,x)- D_{s,y}u(\bar{t}, \bar{x})|^2 dy ds  \right)^{\frac{p}{2}}\right] \leq C_{T,p,b,\sigma,\alpha} \left(|x-\bar{x}|^{\alpha}+|t-\bar{t}|^{\alpha/2}\right)^{p},
\end{equation}
yielding \eqref{eq:cond-2-kolmogorov} with $\lambda_1 = \alpha/2$ and $\lambda_2 = \alpha, \;\alpha \in (0,1/2)$. Since these estimates hold for all $p \ge 2$, we can in particular fix $\alpha \in (0,1/2)$, and then $p \ge \frac{3}{\alpha}$, so that we can apply Kolmogorov's criterion. We conclude that the process $\left\{D_{\cdot,\cdot} u(t,x), t \in [0,T], x \in [0,1] \right\}$ possesses a continuous, $L^{2}(dy \, ds)$-valued modification.

\begin{remark}
    In fact, the computations rehearsed in this section prove the stronger statement that the process $\{ D_{\cdot,\cdot} u(t,x), (t,x) \in [0,T] \times [0,1]\}$ (as an $L^2$-valued random variable) posseses a modification that is $(\mu_1,\mu_2)$-Hölder continuous in $(t,x)$, for all $\mu_1 < 1/4$ and all $\mu_2 < 1/2$.
\end{remark}

To finish the proof of condition (ii) in Theorem \ref{thm:nualart-criterion-existence}, it suffices to show that 
\[\E \left( \sup_{(t,x) \in [0,T] \times [0,1]} \| D_{\cdot,\cdot} u(t,x) \|^2_{L^2([0,T]\times[0,1])} \right) < \infty.
\]
In turn, this follows by Theorem \ref{thm:kolmogorov} if there exist $t_0\in(0,T]$, $x_0\in[0,1]$, and $p\ge2$ such that
\[
\E\big\|Du(t_0,x_0)\big\|_{L^2([0,T]\times[0,1])}^{p}<\infty.
\]

We have 
\begin{multline}
    \E \|D_{\cdot,\cdot} u(t_0,x_0)\|^p_{L^2([0,T] \times [0,1])} =  \E \left( \int_0^T \int_0^1 | D_{s,y} u(t_0,x_0) |^2 dy ds \right)^{p/2} \leq  \left(  \int_0^{t_0} \int_0^1 \E \left(| D_{s,y} u(t_0,x_0) |^p \right)^{\frac{2}{p}} dy ds \right)^{p/2} \\
    \le C \left( \int_0^{t_0} \int_0^1 p_{t_0-s}^2 (x,y) dy ds   \right)^{p/2} \leq C_{T}
\end{multline}
where we applied Minkowski integral inequality and used Lemma \ref{thm:malliavin-derivative}.
 Then, Kolmogorov's criterion says that
\begin{equation} \label{eq:Kolmogorov-supremum}
    \E \left( \sup_{(t,x) \in [0,T] \times [0,1]} \int_0^T \int_0^1 | D_{s,y} u(t,x)|^2 dy ds \right) < \infty,
\end{equation}
finishing the proof of the second condition.

\subsection{Proof of the third condition}\label{sec:3rdcondk1}
We prove the third condition of Theorem \ref{thm:nualart-criterion-existence}. Define the random field
\begin{equation} \label{def-gamma}
    \gamma(t,x)=\int_{0}^{T}\int_0^1 |D_{s,y}u(t,x)|^2 dy ds, \quad (t,x) \in [0,T] \times [0,1].
\end{equation}
In Sections \ref{subsubsec:gammaholder} and \ref{subsubsec:k=0-small-ball} we prove some preliminary estimates for the field $\gamma$. In section \ref{subsubsec:k=0sup-compacts}, we prove the third condition in Theorem \ref{thm:nualart-criterion-existence} for compact subsets $K \subseteq (0,T] \times (0,1)$ under Regime \ref{:dirichlet}, and $K \subseteq (0,T] \times [0,1]$ under Regime \ref{case:neumann}. In Section \ref{subsubsec:k=0-whole-space}, we strengthen this to all compact sets $K \subseteq [0,T] \times [0,1]$ for both Regimes \ref{:dirichlet} and \ref{case:neumann} when in addition Assumption \ref{ass:uo-holder} is in place.

\subsubsection{H\"older continuity of $\gamma(t,x)$} \label{subsubsec:gammaholder}
In the next lemma we show that the process $\gamma$ in \eqref{def-gamma} has a continuous version.

\begin{lemma} \label{lemma-gamma-holder}
    The following estimate holds for $\alpha\in(0,1), p > \frac{3}{\alpha}$ and $(t,x),(\bar{t},\bar{x})\in[0,T]\times[0,1]$:
\begin{equation}\label{eq:w-inc}
 \E\Bigg|
\int_0^T\!\!\int_0^1\left(|D_{s,y}u(t,x)|^2-|D_{s,y}u(\bar t,\bar x)|^2\right)\,dy\,ds
\Bigg|^{p} \
\le C\left(|x-\bar{x}|^\alpha +|t-\bar{t}|^{\alpha /2}\right)^p.
\end{equation}
In particular, the field $\gamma$ defined in \eqref{def-gamma} is $(\mu_1,\mu_2)-$H\"older continuous in $(t,x)$ for all $\mu_1 < 1/4$ and all $\mu_2< 1/2$.
\end{lemma}

\begin{proof}
We rewrite 
\begin{equation*}
|D_{s,y}u(t,x)|^2-|D_{s,y}u(\bar t,\bar x)|^2
= (D_{s,y}u(t,x) + D_{s,y}u(\bar t,\bar x)) (D_{s,y}u(t,x) - D_{s,y}u(\bar t,\bar x)) := J_1 J_2
\end{equation*}
Now, by two applications of Cauchy--Schwarz, we can recast the left hand side of \eqref{eq:w-inc} as
\begin{multline} \label{eq:5.1}
    \E\Bigg|
\int_0^T\!\!\int_0^1\left(|D_{s,y}u(t,x)|^2-|D_{s,y}u(\bar t,\bar x)|^2\right)\,dy\,ds
\Bigg|^p
=
\E\Bigg|
\int_0^T\!\!\int_0^1 J_1J_2\,dy\,ds
\Bigg|^p  
\\
\le \E \left[
\Bigg(\int_0^T\!\!\int_0^1 J_1^2\,dy\,ds\Bigg)^{p/2}
\Bigg(\int_0^T\!\!\int_0^1 J_2^2\,dy\,ds\Bigg)^{p/2} \right] \\
\le
\Bigg(\E\left(\int_0^T\!\!\int_0^1 J_1^2\,dy\,ds\right)^p\Bigg)^{1/2}
\Bigg(\E\left(\int_0^T\!\!\int_0^1 J_2^2\,dy\,ds\right)^p\Bigg)^{1/2}.
\end{multline}
For the $J_2$ term, recalling \eqref{eq:malliavin-difference-L2} we have that
\begin{equation} \label{eq:diff-malliavin}
\begin{split}
\Bigg(\E\left(\int_0^T\!\!\int_0^1 J_2^2\,dy\,ds\right)^p\Bigg)^{1/2} &= \mathbb{E}\left[\left( \int_{0}^{T} \int_{0}^{1} |D_{s,y} u(t,x)- D_{s,y}u(\bar{t}, \bar{x})|^2 dy ds  \right)^{p}\right]^{1/2}  \\
&\le \left( C_{T,p,b,\sigma,\alpha} \left(|x-\bar{x}|^{\alpha}+|t-\bar{t}|^{\alpha/2}\right)^{2p} \right)^{1/2} \\
&\le C_{T,p,b,\sigma,\alpha} \left(|x-\bar{x}|^{\alpha}+|t-\bar{t}|^{\alpha/2}\right)^{p}.
\end{split}
\end{equation}
Moreover, 
\begin{equation} \label{eq:diff-malliavin-2}
\begin{split}
    \Bigg(\E\left(\int_0^T\!\!\int_0^1 J_1^2\,dy\,ds\right)^p\Bigg)^{1/2} &= \mathbb{E}\left[\left( \int_{0}^{T} \int_{0}^{1} |D_{s,y} u(t,x) + D_{s,y}u(\bar{t}, \bar{x})|^2 dy ds  \right)^{p}\right]^{1/2} \\
    &\le C_p \left(  \int_0^T \int_0^1  \left( (\E | D_{s,y} u(t,x)|^{2p})^{1/p} + (\E | D_{s,y} u(\bar{t}, \bar{x})|^{2p})^{1/p}  \right) dy ds     \right)^{1/2} \\
    &\le C_{p,T} \left(  \int_0^T \int_0^1  p^2_{t-s}(y,x) + p^2_{\bar{t}-s}(y,\bar{x}) dy ds     \right)^{1/2} \\
    &\le C_{p,T},
\end{split}
\end{equation}
where the second line follows by Minkowski's integral inequality and the third line by Lemma \ref{thm:malliavin-derivative}.

Plugging \eqref{eq:diff-malliavin} and \eqref{eq:diff-malliavin-2} into \eqref{eq:5.1}, we get \eqref{eq:w-inc}. The claim on the H\"older continuity follows by another application of Kolmogorov's criterion.
\end{proof}

\subsubsection{Small ball estimate for $\gamma(t,x)$} \label{subsubsec:k=0-small-ball}
Next, we show the following small ball estimate for $\gamma(t,x)$ in Regime \ref{:dirichlet}. 
Throughout the rest of this section, for any $\delta>0$ we set 
\begin{equation} \label{eq:Sdelta}
    S_{\delta}=[\delta, T] \times [\delta, 1-\delta].
\end{equation}
Remark \ref{rmk:small-ball-dirichlet-to-neumann} highlights the differences in passing to Regime \ref{case:neumann}.

\begin{lemma} \label{smallballestimatestochasticheat} Let $q \geq 1$. There exist constants $C_{q}>0$ and $ \psi_\delta > 0$ such that, for any $0 < y < \psi_\delta$
\begin{equation} \label{eq:gamma-small-ball}
\sup_{(t,x) \in S_{\delta}} \PP( \gamma(t,x) \leq y) \leq C_{q} y^q.
\end{equation}
\end{lemma}

\begin{proof}
Fix $(t,x)\in S_\delta$ and $\varepsilon>0$. Then, 
\[
\gamma(t,x)\ge
J_\varepsilon(t,x)
:=
\int_{t-\varepsilon}^t\int_0^1
|D_{s,y}u(t,x)|^2\,dy\,ds.
\]
For $s<t$, recalling the form of $D_{s,y} u(t,x)$ in \eqref{PPddfo3} and using the estimate $(a+b)^2\ge \frac12 a^2-b^2$, we can bound $J_\varepsilon$ as in \cite{farazakiskaralistavrianidi} by
\begin{equation} \label{eq:j1-r1-r2}
\begin{split}
J_\varepsilon(t,x)
&\ge \tfrac12 \mathcal{R}_{1}(\varepsilon;t,x)
- \mathcal{R}_{2}(\varepsilon;t,x),
\end{split}
\end{equation}
with
\begin{equation}
\begin{split}
    \mathcal{R}_{1}&:= \int_{t-\varepsilon}^t\; \int_0^1
\sigma(u(s,y))^2 G_{t-s}(x,y)^2\,dy\,ds,\\
\mathcal{R}_{2} &:= \int_{t-\varepsilon}^t\;\!\int_0^1 \left( \int_{s}^{t}\int_{0}^{1}G_{t-\theta}(x,r)m(r,\theta)D_{s,y}u(\theta,r) drd \theta \right. \\
&\quad  + \left. \int_{s}^{t}\int_{0}^{1} G_{t-\theta}(x,r) \hat{m}(r,\theta)D_{s,y}u(\theta,r) W(dr,d \theta) \right)^2 dy\,ds.
\end{split}
\end{equation}
Uniform ellipticity yields
\begin{equation} \label{eq:r1-uniform-elipticity}
\mathcal{R}_1\ge (C_\sigma)^{-2}\int_{t-\varepsilon}^t\;\int_0^1
G_{t-s}(x,y)^2\,dy\,ds = (C_\sigma)^{-2}
\int_0^\varepsilon\int_0^1 G_r(x,y)^2\,dy\,dr.
\end{equation}
Setting $
c_1:=\frac{\sqrt{2\pi}-1}{8\pi} $ and $ c_2 :=  (C_\sigma)^{-2} c_1$, the bound (B.2.11) in \cite{dalangsole} for Dirichlet boundary conditions says there exists $c(\delta)$ such that, for all $0<\varepsilon< r(\delta):=\min\{\frac{c_1}{c(\delta)}, \delta^2, (1-\delta)^2 \}$ 
 \begin{equation} \label{lowerbound}
\mathcal{R}_1\ge (C_\sigma)^{-2} \int_0^\varepsilon \int_{x - \sqrt{\varepsilon}}^{x+\sqrt{\varepsilon}} G_r(x,y)^2\,dy\,dr \ge c_2 \,\varepsilon^{1/2}.
\end{equation}
Now we proceed by bounding the moments of $\mathcal{R}_2$. For every $q\ge1$, we claim that there exists some $C_q > 0$ such that
\begin{equation}
\sup_{(t,x)\in S_\delta}
\E\!\left[\clr_2(\varepsilon;t,x)^q\right]
\le C_q\,\varepsilon^{q}.
\end{equation}
Using $(a+b)^2\le 2a^2+2b^2$,
\begin{eqnarray}
\clr_2(\varepsilon;t,x)
&\leq&
2\int_{t-\varepsilon}^t\int_0^1
 \left( \int_s^t\int_0^1
G_{t-r}(x,z)\,
m(r,z)\,
D_{s,y}u(r,z)\,dz\,dr   \right)^2 \,dy\,ds  
\nonumber\\
& &+2\int_{t-\varepsilon}^t\int_0^1
 \left( \int_s^t\int_0^1
G_{t-r}(x,z)\,
\hat{m}(r,z)\,
D_{s,y}u(r,z)\,
W(dz,dr) \right)^2 \,dy\,ds 
\nonumber\\
&=& 2\int_{t-\varepsilon}^t\int_0^1
 \clr_{2,1}(s,y;t,x)^2 \,dy\,ds 
+2\int_{t-\varepsilon}^t\int_0^1
 \clr_{2,2}(s,y;t,x)^2 \,dy\,ds, \label{eq:r2-r21-r22}
\end{eqnarray}
where
\begin{flalign*}
\clr_{2,1}(s,y;t,x)
&:=
\int_s^t\int_0^1
G_{t-r}(x,z)\,
m(r,z)\,
D_{s,y}u(r,z)\,dz\,dr,
\\
\clr_{2,2}(s,y;t,x)
&:=
\int_s^t\int_0^1
G_{t-r}(x,z)\,
\hat{m}(r,z)\,
D_{s,y}u(r,z)\,
W(dz,dr).
\end{flalign*}
Since $m$ is bounded a.s., for fixed $(s,y)$, Cauchy-Schwarz in $(z,r)$ yields
\begin{align*}
    \clr_{2,1}(s,y;t,x)^2 &\le C_b \left( \int_s^t\int_0^1
|G_{t-r}(x,z)|\,
|D_{s,y}u(r,z)|\,dz\,dr \right)^2 \\
&\le C_b
\left(
\int_s^t
\|G_{t-r}\|_{L^2([0,1])}^2\,dr
\right)
\left(
\int_s^t
\|D_{s,y}u(r,\cdot)\|_{L^2([0,1])}^2\,dr
\right).
\end{align*}
For $x\in[\delta,1-\delta]$, using Lemma \ref{lem:G-L2} and the fact that $t-\varepsilon \le s \le t$,
\[
\int_s^t
\|G_{t-r}\|_{L^2([0,1])}^2\,dr
=
\int_0^{t-s}
\|G_{\tau}\|_{L^2([0,1])}^2\,d\tau
\le C_T
\int_0^{\varepsilon}
\tau^{-1/2}\,d\tau
= C_T
\varepsilon^{1/2}.
\]
Thus, 
\begin{equation*}
\begin{split}
    \int_{t-\varepsilon}^t\;\int_0^1
 \clr_{2,1}(s,y;t,x)^2 \,dy\,ds 
&\le
C_{T,b}
\varepsilon^{1/2}
\int_{t-\varepsilon}^t\; \int_0^1
\int_s^t
\|D_{s,y}u(r,\cdot)\|_{L^2([0,1])}^2 
\,dr\,dy\,ds \\
&= C_{T,b}
\varepsilon^{1/2}
\int_{t-\varepsilon}^t
\int_{t-\varepsilon}^r
\int_0^1
\|D_{s,y}u(r,\cdot)\|_{L^2([0,1])}^2
\,dy\,ds\,dr,
\end{split}
\end{equation*}
where we have used Fubini's theorem. Moreover, Theorem \ref{thm:malliavin-derivative} and calculations analogous to \eqref{eq:i2-3} say that
\begin{equation} \label{eq:0234}
    \sup_{r\le T}
\E\Big[
\Big(
\int_0^r\!\!\int_0^1 
\|D_{s,y}u(r,\cdot)\|^2_{L^2([0,1])} dy\,ds
\Big)^q
\Big]
\leq C_{T,q}.
\end{equation}
Since $r-s\le\varepsilon$, we obtain
\begin{equation} \label{eq:R21-estimate}
\begin{split}
\E \left( \int_{t-\varepsilon}^t\;\int_0^1
 \clr_{2,1}(s,y;t,x)^2 \,dy\,ds \right)^q
&\le C_{T,b,q}\varepsilon^{q/2}
\E \left( \int_{t-\varepsilon}^t
\int_{t-\varepsilon}^r
\int_0^1
\|D_{s,y}u(r,\cdot)\|_{L^2([0,1])}^2
\,dy\,ds\,dr \right)^q \\ 
&\le C_{T,b,q}\varepsilon^{q/2}
\left( \int_{t-\varepsilon}^t
C_{T,q} dr \right)^q \\
&\leq C_{T,b,q}\varepsilon^{3q/2} \le C_{T,b,q}\varepsilon^{q},
\end{split}
\end{equation}
where for the second line we used Minkowski's integral inequality and \eqref{eq:0234}.
To estimate $\clr_{2,2}$, we write
\begin{align}
    \E &\left( \int_{t-\varepsilon}^t\;\int_0^1
 \clr_{2,2}(s,y;t,x)^2 \,dy\,ds \right)^q  \nonumber\\
 &\leq \left(\int_{t-\varepsilon}^{t} \int_0^1 \left( \E \left(\int_s^t\!\!\int_0^1
G_{t-r}(x,z)\,
\hat{m}(r,z,u(r,z))\,
D_{s,y}u(r,z)\,
W(dz,dr)\right)^{2q} \right)^{\frac{1}{q}} dy ds \right)^{q} \nonumber\\
 &\le C_{\sigma,q} \left(\int_{t-\varepsilon}^{t} \int_0^1   \E \left[\left(\int_s^t \int_0^1 G_{t-r}(x,z)^2 D_{s,y}u(r,z)^2 dz dr \right)^{q}\right]^{\frac{1}{q}} dy ds  \right)^{q} \nonumber \\
 &\le C_{\sigma,q} \left(\int_{t-\varepsilon}^{t} \int_0^1  \int_s^t \int_0^1 \E \left[G_{t-r}(x,z)^{2q} D_{s,y}u(r,z)^{2q}\right]^{\frac{1}{q}} dz dr dy ds  \right)^{q} \nonumber \\
 &\le C_{\sigma,q} \left(\int_{t-\varepsilon}^{t} \int_0^1  \int_s^t \int_0^1 G_{t-r}(x,z)^{2} p_{r-s}(y-z)^{2} dz dr dy ds  \right)^{q}. \label{eq:6789}
\end{align}
The second and fourth lines follows by Minkowski's integral inequality; the third line follows by BDG; \eqref{eq:6789} follows by Theorem \ref{thm:malliavin-derivative}.

Then we apply Fubini and get
\begin{equation*}
\begin{split}
    \int_{t-\varepsilon}^{t} \int_0^1  \int_s^t \int_0^1 G_{t-r}(x,z)^{2} p_{r-s}(y-z)^{2} dz dr dy ds &\le  \int_{t-\varepsilon}^{t} \int_0^1 G_{t-r}(x,z)^{2} \left( \int_{t-\varepsilon}^r  \int_0^1  p_{r-s}(y-z)^{2} dy ds \right)dz dr \\
    &\le C_{T} \int_{t-\varepsilon}^{t} \int_0^1 G_{t-r}(x,z)^{2} \sqrt{r-(t-\varepsilon)} dz dr \\
    &\leq C_{T} \sqrt{\varepsilon}  \int_{t-\varepsilon}^{t} \int_0^1 G_{t-r}(x,z)^{2} dz dr \leq C_T \varepsilon,
\end{split}
\end{equation*}
which, combined with \eqref{eq:6789}, leads to
\begin{equation} \label{eq:r22-estimate}
    \E \left( \int_{t-\varepsilon}^t\;\int_0^1
 \clr_{2,2}(s,y;t,x)^2 \,dy\,ds \right)^q\leq C_{T,\sigma,q} \varepsilon^{q}.
\end{equation}
Combining \eqref{eq:r2-r21-r22}, \eqref{eq:R21-estimate}, and \eqref{eq:r22-estimate}, we get that 
$$ \E[\clr_2(\varepsilon;t,x) ^q] \leq C_{T,q,b,\sigma} \varepsilon^{q} $$ for a constant $C_{T,q,b,\sigma}$ that does not depend on $(t,x)$. By Markov's inequality,
$$
\PP( \clr_2 > \varepsilon^{\frac{1}{2}}) \leq \frac{\E[\clr_2^q]}{\varepsilon^{\frac{q}{2}}} \leq C_{T,q,b,\sigma} \varepsilon^{\frac{q}{2}}.
$$
Now recalling the constant $c_2$ in \eqref{lowerbound}, \eqref{eq:j1-r1-r2} and the previous equation say that, for all $\varepsilon < r(\delta)$ (defined above \eqref{lowerbound}),
$$
\PP\left( J_{\varepsilon}(t,x) \le  \frac{c_{2}}{4}\varepsilon^{\frac{1}{2}} \right) \le  \PP\left( \frac{1}{2} \clr_1 - \clr_2 \le  \frac{c_{2}}{4}\varepsilon^{\frac{1}{2}} \right) \leq \PP\left(\clr_2 \ge \frac{c_{2}}{4} \varepsilon^{\frac{1}{2}}\right) \leq C_{T,q,b,\sigma} \frac{4^{q}}{c_{2}^{q}}\varepsilon^{\frac{q}{2}} .
$$
This lets us conclude that there exists a constant $C_{T,q,b,\sigma}>0$ not depending on $t,x$ such that, for any $0<y< \frac{c_{2}r(\delta)^{\frac{1}{2}}}{4}$, setting $0 < \varepsilon(y) := \left(\frac{4 y}{c_2}\right)^2 < r(\delta)$
\begin{multline*}
\sup_{(t,x) \in S_{\delta}} \PP( \gamma(t,x) \leq y) \le \sup_{(t,x) \in S_{\delta}} \PP(  J_{\varepsilon(y)}(t,x) \leq y) 
\le  \sup_{(t,x) \in S_{\delta}} \PP(  J_{\varepsilon(y)}(t,x) 
\leq \frac{c_2}{4} \varepsilon(y)^{1/2}) 
\le C_{T,q,b,\sigma} \varepsilon(y)^{q/2}
\\= C_{T,q,b,\sigma} y^{q},
\end{multline*}
i.e., \eqref{eq:gamma-small-ball} holds with $\psi_\delta := \frac{c_{2}r(\delta)^{\frac{1}{2}}}{4}$ and $C_q :=  C_{T,q,b,\sigma}$.
\end{proof}

\begin{remark} \label{rmk:small-ball-dirichlet-to-neumann}
    Lemma \ref{smallballestimatestochasticheat} is also true under Regime \ref{case:neumann}, with $S_\delta$ being replaced by sets of the form
    \begin{equation} \label{eq:Ldelta}
        L_\delta = [\delta,T] \times [0,1].
    \end{equation}
   The only difference in the proof is the lower bound in \eqref{lowerbound}, for which we obtain, from (B.3.7) in \cite{dalangsole}, that
   \begin{equation}
       \clr_1 \ge (C_\sigma)^{-2} \int_0^\varepsilon \int_0^1 G^{\operatorname{N}}_r(x,y)^2 dy dr \ge c_2 \varepsilon^{1/2},
   \end{equation}
   with the same constant $c_2$ unoformly over all $x \in [0,1]$. The rest of the proof is identical.
\end{remark}

\subsubsection{Deducing condition (iii) and existence of density for supremum over compact subsets} \label{subsubsec:k=0sup-compacts}

We will use the H\"older continuity and the small ball estimate we showed for $\gamma(t,x)$ to get the following proposition.
\begin{proposition} \label{prop:Sdelta-K}
Suppose that either (i) we are in Regime \ref{:dirichlet} and $K \subset (0,T]\times(0,1)$, or (ii) we are in Regime \ref{case:neumann} and $K \subset (0,T]\times[0,1]$. Then
\[
\PP \Big(
\exists (t,x)\in K:
\gamma(t,x)=0
\Big)=0.
\]
\end{proposition}

\begin{proof}
We only prove the statement under Regime \ref{:dirichlet}, since the proof of the statement in Regime \ref{case:neumann} is analogous.

Since $K$ is a compact subset of $(0,T]\times (0,1)$, there exists $\delta>0$ such that 
\[
K \subseteq S_{\delta}=[\delta,T]\times [\delta, 1- \delta].
\]
By Lemma \ref{lemma-gamma-holder}, the field $\gamma(t,x)$ admits a $(\mu_1,\mu_2)$--H\"older continuous version on $S_\delta$ for all $\mu_1<1/4$ and $\mu_2<1/2$. 
Hence there exists an a.s.\ finite random variable $C_\delta(\omega)$ such that for all $(t,x),(\bar t,\bar x)\in S_\delta$,
\begin{equation} \label{eq:gamma-holder-cd}
    |\gamma(t,x)-\gamma(\bar t,\bar x)|
\le 
C_\delta(\omega)\big(
|t-\bar t|^{\mu_1}+|x-\bar x|^{\mu_2}
\big).
\end{equation}
Fix such $\mu_1,\mu_2$.
We proceed by a localization argument. From the a.s.\ finiteness of $C_\delta$,
\begin{equation*}
    E_n := \{ C_\delta \le n \}, \qquad n\in\mathbb N, \quad \PP\left(\bigcup_{n=1}^\infty E_n\right)=1.
\end{equation*}
It therefore suffices to prove that for each fixed $n$,
\[
\PP\Big(
\exists (t,x)\in K : \gamma(t,x)=0
\ \cap E_n
\Big)=0.
\]
Fix $n\in\mathbb N$ and work on the event $E_n$. 
\eqref{eq:gamma-holder-cd} says that, on $E_n$, $\gamma$ is H\"older continuous with deterministic constant $n$.

Take $\eta>1$ such that $\frac{1}{\eta}< \psi_\delta = \frac{c_{2}r(\delta)^{1/2}}{4}$, where the constants were introduced in the proof of Lemma \ref{smallballestimatestochasticheat}; then, letting $\varepsilon_m:=\eta^{-m}$, we have that $0<\varepsilon_{m}< \frac{c_{2} r(\delta)^{\frac{1}{2}}}{4}$ for all $m$. Set $\rho_{m,n}:=\frac{\varepsilon_m}{2n}$. 
Then on $E_n$, if we select $(t,x), (\bar t, \bar x) \in S_\delta$ such that
\begin{equation} \label{eq:grid-estimate}
 |t-\bar t|< \rho_{m,n}^{1/\mu_1}
\quad\text{and}\quad
|x-\bar x|< \rho_{m,n}^{1/\mu_2},   \quad \text{then }|\gamma(t,x)-\gamma(\bar t,\bar x)| \le C_\delta \left( |t-\bar t|^{\mu_1} + |x-\bar x|^{\mu_2}  \right) \le \varepsilon_m.
\end{equation}

Let $N_1:=\left\lceil \frac{T-\delta}{\rho_{m,n}^{1/\mu_1}} \right\rceil$ and
$N_2:=\left\lceil \frac{1-2\delta}{\rho_{m,n}^{1/\mu_2}} \right\rceil$,
and set $t_i:=\delta+i\frac{T-\delta}{N_1}, i = 0,\dots,N_1$, $x_j:=\delta+j\frac{1-2\delta}{N_2}, j = 0,\dots N_2$.
Let $S_{m,n}:=\{(t_i,x_j):0\le i\le N_1,\ 0\le j\le N_2\}$. Then, the deterministic grid $S_{m,n}$ is finite and its size is upper bounded by
\[
|S_{m,n}| \le \left(  \left\lceil \frac{T}{\rho_{m,n}^{1/\mu_1}} \right\rceil +1 \right)  \left( \left\lceil \frac{1}{\rho_{m,n}^{1/\mu_2}} \right\rceil + 1  \right) \le C_T \rho_{m,n}^{-\left(\frac1{\mu_1}+\frac1{\mu_2}\right)},
\]
and moreover, for all $(t,x) \in S_\delta$, there exists a $(\bar t, \bar x) \in S_{m,n}$ such that \eqref{eq:grid-estimate} holds.
Recalling the form of $\rho_{m,n}$ we have that
\[
\text{setting }k=\frac1{\mu_1}+\frac1{\mu_2}, \quad \text{then }|S_{m,n}| \le C_{n,T}\,\varepsilon_m^{-k}.
\]

Now select an $\omega \in E_{n}$ such that there is a $(t,x) \in K$ with $\gamma(t,x;\omega) = 0$
Then for every $m$ there exists $(t_m,x_m)\in S_{m,n}$ such that, thanks to \eqref{eq:grid-estimate},
\[
\gamma(t_m,x_m) = |\gamma(t_m,x_m)- \gamma(t,x)| \le n \left( |t-t_m|^{\mu_1} + |x-x_m|^{\mu_2} \right) < \varepsilon_m.
\]
Hence
\begin{equation} \label{eq:limsup-amn}
\begin{split}
    \PP\Big(
\exists (t,x)\in K:\gamma(t,x)=0
\ \cap E_n
\Big)
&\le
\PP\Big(
\text{for all } m \text{ there exists } (t_m,x_m)\in S_{m,n}
:\gamma(t_m,x_m)<\varepsilon_m
\Big) \\
&\le \PP(\limsup_{m\to\infty} A_{m,n}),
\end{split}
\end{equation}
where we set
\[
A_{m,n}
:=
\{
\exists (t_m,x_m)\in S_{m,n}:
\gamma(t_m,x_m)<\varepsilon_m
\}.
\]
Since $\varepsilon_m < \psi_\delta$ by selection, Lemma \ref{smallballestimatestochasticheat} says that, for all $q \ge 1$,
\[
\PP(A_{m,n})
\le
\sum_{(t_m,x_m)\in S_{m,n}}
\PP(\gamma(t_m,x_m)<\varepsilon_m)
\le |S_{m,n}| C_{q} \varepsilon_m^{q} \le
C_{n,\delta,q}\,
\varepsilon_m^{-k}
\,\varepsilon_m^{q}.
\]
Selecting $q>k$, we obtain $\sum_{m=1}^{\infty} \PP(A_{m,n}) < \infty$.
Hence, by the Borel--Cantelli lemma, $\PP(\limsup_{m\to\infty} A_{m,n})=0$, and therefore \eqref{eq:limsup-amn} says that
\[
\PP\Big(
\exists (t,x)\in K:\gamma(t,x)=0
\ \cap E_n
\Big)=0.
\]
Since $\bigcup_{n=1}^\infty E_n$ has probability one, we conclude that
\[
\PP\Big(
\exists (t,x)\in K:\gamma(t,x)=0
\Big)=0,
\]
which finishes the proof.
\end{proof}

\begin{remark} \label{remarkonsuponthecompact}

Proposition \ref{prop:Sdelta-K} implies that $\PP \Big(
\exists (t,x)\in \mathcal{S}_{K}:
\gamma(t,x)=0
\Big)=0$, which in turn implies that Condition (iii) of Theorem \ref{thm:nualart-criterion-existence} is satisfied for $\phi_ K := \sup_{(s,y) \in K}u(s,y)$. Since conditions (i) and (ii) are also satisfied, Theorem \ref{thm:nualart-criterion-existence} yields that its law is absolutely continuous with respect to the Lebesgue measure.
\end{remark}

\subsubsection{Upgrading the result to $[0,T]\times [0,1]$} \label{subsubsec:k=0-whole-space}
Now we upgrade these results to get existence of a density for the supremum of $u$ on $[0,T] \times [0,1]$, when in addition Assumption \ref{ass:uo-holder} holds.

We first work under Regime \ref{:dirichlet}. Define the following random sets
\begin{align}
    \mathcal{L} &:=  \{(t,x) \in [0,T] \times [0,1] : u(t,x) = \sup_{(s,y) \in [0,T] \times [0,1]} u(s,y) \} \\
    \mathcal{L}_1 &:=  \{(0,x) \in \{0\} \times (0,1) : u(0,x) = \sup_{(s,y) \in [0,T] \times [0,1]} u(s,y) \} \\
    \mathcal{L}_2  &:= \{(t,x) \in [0,T] \times \{0,1\} : u(t,x) = \sup_{(s,y) \in [0,T] \times [0,1]} u(s,y) \} \\
    \mathcal{L}_3  &:= \{(t,x) \in (0,T] \times (0,1) : u(t,x) = \sup_{(s,y) \in [0,T] \times [0,1]} u(s,y) \},
\end{align}
and note that, for all $\omega \in \Omega$, it holds that $\mathcal{L}(\omega) = \mathcal{L}_1(\omega) \cup \mathcal{L}_2(\omega) \cup \mathcal{L}_3(\omega)$. Further, we can decompose $\mathcal{L}_3$ as the union of an increasing sequence of sets as
\begin{equation} \label{eq:l3-decomp-an}
    \mathcal{L}_3(\omega) := \bigcup_{n \ge 3} \left\{(t,x) \in \left[\frac{1}{n},T\right] \times \left[\frac{1}{n},1- \frac{1}{n}\right] : u(t,x;\omega) = \sup_{(s,y) \in [0,T] \times [0,1]} u(s,y;\omega) \right\} = \bigcup_{n \ge 3} A_n(\omega),
\end{equation}
with $A_n := \left\{(t,x) \in \left[\frac{1}{n},T\right] \times \left[\frac{1}{n},1- \frac{1}{n}\right] : u(t,x) = \sup_{(s,y) \in [0,T] \times [0,1]} u(s,y) \right\} \subseteq \left[\frac{1}{n},T\right] \times \left[\frac{1}{n},1- \frac{1}{n}\right]$. We can then estimate
\begin{equation}
   \PP\Big( \{ \omega \in \Omega: \gamma(t_0,x_0;\omega) = 0, \, \text{for some }
 (t_0,x_0)\in \mathcal{L}(\omega) \} ) 
\Big)  \le \sum_{i=1}^3 \PP(\Xi_i),
\end{equation}
with
\begin{equation} \label{eq:prob-decomp}
    \Xi_i := \{\omega \in \Omega: \gamma(t_0,x_0;\omega) = 0, \, \text{for some }(t_0,x_0) \in \mathcal{L}_i(\omega) \}.
\end{equation}
Moreover, in view of \eqref{eq:l3-decomp-an} and Proposition \ref{prop:Sdelta-K} (under Regime \ref{:dirichlet}),
\begin{multline} \label{eq:Xi-vanishes}
    \PP(\Xi_3) = \PP \left( \bigcup_{n \ge 3} \{\omega \in \Omega: \gamma(t_0,x_0;\omega) = 0, \, \text{for some }(t_0,x_0) \in A_n(\omega)\}  \right) \\
     \le\PP \left( \bigcup_{n  \ge 3} \left\{\omega \in \Omega: \gamma(t_0,x_0;\omega) = 0, \, \text{for some }(t_0,x_0) \in \left[\frac{1}{n},T   \right] \times \left[\frac{1}{n},1- \frac{1}{n}\right]\right\}  \right) \\
    \le \sum_{n \ge 3}  \PP \left(  \left\{\omega \in \Omega: \gamma(t_0,x_0;\omega) = 0, \, \text{for some }(t_0,x_0) \in \left[\frac{1}{n},T   \right] \times \left[\frac{1}{n},1- \frac{1}{n}\right]\right\}  \right) = 0.
\end{multline}
We further analyze the two sets $\Xi_1, \Xi_2$. Recall that, if $(t,x) \in \mathcal{L}_1(\omega)$, it follows that $\gamma(t,x) = 0$ since $u(0,x)$ is deterministic. Then,
\begin{multline} \label{eq:Xi1-analysis}
    \PP(\Xi_1) \\
    \le \PP\left(\left\{ \omega : \gamma(0,x_0;\omega) = 0  \text{ for some } (0,x_0) \in \left\{(0,x) \in \{0\} \times (0,1) : u(0,x) = \sup_{(s,y) \in [0,T] \times [0,1]} u(s,y) \right\} \right\} \right) \\
    \le \PP\left(\left\{ \omega : \text{ for some }x_0 \in (0,1),  u(0,x_0) = \sup_{(s,y) \in [0,T] \times [0,1]} u(s,y;\omega)  \right\} \right) \\
    \le \PP\left(\left\{ \omega :  u(0,x^*) = \sup_{(s,y) \in [0,T] \times [0,1]} u(s,y;\omega)  \right\} \right)  \\
    \le \PP\left( \sup_{(t,x)\in [0,T]\times [0,1]}u(t,x)=u_0(x^*) \right),
\end{multline}
where $x^* $ is the point in Assumption \ref{ass:uo-holder}.
Similarly, if $(t,x) \in \mathcal{L}_2(\omega)$, it follows that $\gamma(t,x) = 0$ and $u(t,0) = u(t,1) = 0$, implying that
\begin{multline}
    \PP(\Xi_2) \\
    \le \PP\left(\left\{ \omega : \gamma(t_0,x_0;\omega) = 0  \text{ for some } (t_0,x_0) \in \left\{(t,x) \in [0,T] \times \{0,1\} : u(t,x) = \sup_{(s,y) \in [0,T] \times [0,1]} u(s,y;\omega) \right\} \right\} \right) \\ 
   \le  \PP\left(\left\{ \omega : \text{ for some } (t_0,x_0) \in [0,T] \times \{0,1\} : u(t_0,x_0) = \sup_{(s,y) \in [0,T] \times [0,1]} u(s,y)  \right\} \right) \\
   \le \PP\left( \sup_{(s,y) \in [0,T] \times [0,1]} u(s,y) = 0   \right).
\end{multline}
Note that, if $x^*\in (0,1)$, then
Propositions \ref{dirichletneumannmaximizeratinitial} and \ref{dirichletmaximizeratboundary} below show respectively that $\PP(\Xi_1) = \PP(\Xi_2) = 0$ in Regime \ref{:dirichlet}. If $x^*\in \{0,1\}$, then $u_0(x^*) = 0$ and so again $\PP(\Xi_2) = \PP(\Xi_1) = 0$ by Proposition \ref{dirichletmaximizeratboundary}. Coupled with \eqref{eq:prob-decomp} and \eqref{eq:Xi-vanishes}, this proves that 
\begin{equation*}
    \PP \Big(
\exists (t,x)\in \mathcal{L}:
\gamma(t,x)=0
\Big) \le \sum_{i=1}^3 \PP(\Xi_i) = 0,
\end{equation*}
thus proving condition (iii) in Theorem \ref{thm:nualart-criterion-existence}. 

The proof of Theorem \ref{maintheorem} is then complete under Regime \ref{:dirichlet}.

For Regime \ref{case:neumann}, the situation is simpler. We can take 
\begin{align}
    \mathcal{L} &:=  \{(t,x) \in [0,T] \times [0,1] : u(t,x) = \sup_{(s,y) \in [0,T] \times [0,1]} u(s,y) \} \\
    \mathcal{L}_1 &:=  \{(0,x) \in \{0\} \times [0,1] : u(0,x) = \sup_{(s,y) \in [0,T] \times [0,1]} u(s,y) \} \\
    \mathcal{L}_2  &:= \{(t,x) \in (0,T] \times [0,1] : u(t,x) = \sup_{(s,y) \in [0,T] \times [0,1]} u(s,y) \},
\end{align}
and denote the induced decomposition of the events by $\Xi_1,\Xi_2$. Then, showing $\PP(\Xi_2) = 0$ follows by the same arguments as in \eqref{eq:Xi-vanishes} upon noticing that Proposition \ref{prop:Sdelta-K} (under Regime \ref{case:neumann}) gives the associated compact sets as subsets of the spatial set $[0,1]$. Moreover, the same estimates as in \eqref{eq:Xi1-analysis} yield
\begin{flalign*}
    &\PP(\Xi_1) \\
    &\le \PP\left(\left\{ \omega : \gamma(0,x_0;\omega) = 0  \text{ for some } (0,x_0) \in \left\{(0,x) \in \{0\} \times [0,1] : u(0,x) = \sup_{(s,y) \in [0,T] \times [0,1]} u(s,y) \right\} \right\} \right) \nonumber\\
    &\le \PP\left(\left\{ \omega : \text{ for some }x_0 \in [0,1],  u(0,x_0) = \sup_{(s,y) \in [0,T] \times [0,1]} u(s,y;\omega)  \right\} \right) \nonumber\\
    &\le \PP\left(\left\{ \omega :  u(0,x^*) = \sup_{(s,y) \in [0,T] \times [0,1]} u(s,y;\omega)  \right\} \right)  \nonumber\\
     &\le \PP\left( \sup_{(t,x)\in [0,T]\times [0,1]}u(t,x)=u_0(x^*) \right).
\end{flalign*}
Finally, Proposition \ref{dirichletneumannmaximizeratinitial}(ii) below shows that $\PP(\Xi_1) = 0$ in Regime \ref{case:neumann}, thus proving condition (iii) in Theorem~\ref{thm:nualart-criterion-existence}.

This also finishes the proof of Theorem \ref{maintheorem} under Regime \ref{case:neumann}.



\begin{proposition} \label{dirichletneumannmaximizeratinitial}
\begin{enumerate}[label=(\roman*)]
    \item Suppose Regime \ref{:dirichlet} holds. Then, if $x^* \in (0,1)$, we have that 
    \[
    \PP \left( \sup_{(t,x) \in [0,T] \times [0,1]} u(t,x) = u_0(x^*)\right) = 0.
    \]
    \item Suppose Regime \ref{case:neumann} holds. Then, if $x^* \in [0,1]$, we have that 
    \[
    \PP \left(\sup_{(t,x) \in [0,T] \times [0,1]} u(t,x) = u_0(x^*)\right) = 0.
    \]
\end{enumerate}

\end{proposition}
\begin{proof}
Since the proofs of the two items are very similar, we prove item (i). We will point out the difference for Regime \ref{case:neumann} in the proof.

We aim to show
\[
\mathbb{P}
\left(
\sup_{t\in(0,T],\,x\in[0,1]} u(t,x) > u_0(x^*) 
\right)
= 1.
\]
It suffices to prove the stronger statement that
\begin{equation*}
\mathbb{P}
\left(
u(t_n,x^{*}) > u_0(x^{*})
\ \text{i.o. as } t_n \downarrow 0
\right)
=1.
\end{equation*}
Define the germ $\sigma$-algebra generated by the Brownian sheet:
\[
\mathcal{F}_{0+}
:=
\bigcap_{\varepsilon>0}
\sigma\Big\{
W(s,x) :
s \le \varepsilon , 0 \leq x \leq 1
\Big\} \vee \mathcal{N}.
\]
The $\sigma$-algebra $\mathcal{F}_{0+}$ is trivial (0–1 law), see for instance Lemma 4.1 from \cite{dalangpu} and Proposition 3.2 from \cite{walsh01law}. Fix any deterministic sequence $t_n \to  0$ and define
\begin{equation} \label{eq:A-def}
    A :=
\left\{
u(t_n,x^{*}) > u_0(x^{*})
\text{ i.o.}
\right\}.
\end{equation}
Since $u(t_n,x^{*})$ is measurable with respect to the noise
restricted to $[0,t_n]\times[0,1]$, the event $A$
depends only on arbitrarily small times.
Hence $A \in \mathcal{F}_{0+}$, and by the 0–1 law,
$\mathbb{P}(A) \in \{0,1\}.$ It therefore remains to show that $\mathbb{P}(A)>0$.

Fix $t>0$ small. Write
\begin{equation} \label{eq:u-decomp-D-Z-R}
 u(t,x^{*}) - u_0(x^{*})
=
D(t) + Z(t)+R(t),   
\end{equation}
where
\begin{equation} \label{eq:D-Z-R-def}
\begin{split}
    D(t)
&:= \int_0^1 G_t(x^{*},y)u_0(y)dy - u_0(x^{*})
+ \int_0^t \int_0^1 G_{t-s}(x^{*},y) b(u(s,y)) dy \, ds, \\
Z(t)
&:= \int_0^t \int_0^1
G_{t-s}(x^{*},y)\,
\sigma(u_0(y))\,
W(ds,dy), \\
R(t)
&:= \int_0^t \int_0^1
G_{t-s}(x^{*},y)\,
\big(\sigma(u(s,y))-\sigma(u_0(y))\big)\,
W(ds,dy).
\end{split}
\end{equation}

Then $Z(t)$ is centered Gaussian with variance
\begin{align}
\mathrm{Var}(Z(t))=
\int_0^t \int_0^1
G_{t-s}(x^{*},y)^2
\sigma(u_0(y))^2
\,dy\,ds \notag \ge (C_\sigma)^{-2}
\int_0^t \int_0^1
G_r(x^{*},y)^2
\,dy\,dr \ge  c_{1} (C_\sigma)^{-2} t^{\frac{1}{2}}
\end{align}
Note that in the above inequality we used the lower bound of the Green's function which is valid for Dirichlet as long as $x^{*} \in (0,1)$ and valid for Neumann for any $x^{*} \in [0,1]$; see the argument performed in \eqref{eq:r1-uniform-elipticity} and Remark \ref{rmk:small-ball-dirichlet-to-neumann} for the Neumann case. This is also the point in the argument that reveals the admissible values for $x^*$ given in the statement.

Hence
\[
\sqrt{\mathrm{Var}(Z(t))} \geq c_1^{\frac{1}{2}}(C_\sigma)^{-1} t^{1/4}.
\]
Therefore there exist constants $c_1,c_2>0$ such that for $t$ small,
\begin{equation} \label{eq:zt-lower-bound}
\mathbb{P}\big(Z(t) \ge c_1 t^{1/4}\big) \ge c_2.
\end{equation}

We now deal with the $D$ term. Decompose the $D$ term to
\begin{equation}\label{eq:D-def}
D(t) = \int_0^1 G_t(x^{*},y)u_0(y)dy-u_0(x^{*})
\;+\;
\int_0^t \int_0^1 G_{t-s}(x^{*},y)b(u(s,y))) \,dy \ ds = D_1(t) + D_2(t),
\end{equation}
where
\begin{align*}
    D_1(t) &:=  \int_0^1 G_t(x^{*},y)u_0(y)dy-u_0(x^{*}),\\
    D_2(t) &:=\int_0^t \int_0^1 G_{t-s}(x^{*},y)b(u(s,y))) \,dy \ ds.
\end{align*}
Since $u_0(x^{*})$ is the maximum of $u_0$ and $\int_0^1 G_t(x^{*},y) \leq 1$, we have 
$
D_1(t)\le 0.
$
Moreover,
\begin{align}
-\,D_1(t)
&= \int_0^1 G_t(x^{*},y)\,\big(u_0(x^{*})-u_0(y)\big)\,dy \notag + \left( 1 -\int_0^1 G_{t}(x^{*},y) dy \right) u_0(x^{*})\\
&= \int_{|y-x^{*}|\le r_0} G_t(x^{*},y)\,\big(u_0(x^{*})-u_0(y)\big)\,dy
\;+\;
\int_{|y-x^{*}|> r_0} G_t(x^{*},y)\,\big(u_0(x^{*})-u_0(y)\big)\,dy \nonumber\\
&\quad+ \left( 1 -\int_0^1 G_{t}(x^{*},y) dy \right) u_0(x^{*}).
\label{eq:D1-split}
\end{align}
Note that for the Neumann kernel $\int_0^1 G_{t}(x^{*},y) dy=1$ so the last term in the third line in \eqref{eq:D1-split} is zero.
We denote 
$$I_{\mathrm{near}}(t):= \int_{|y-x^{*}|\le r_0} G_t(x^{*},y)\,\big(u_0(x^{*})-u_0(y)\big)\,dy,$$
and 
$$I_{\mathrm{far}}(t):= \int_{|y-x^{*}|> r_0} G_t(x^{*},y)\,\big(u_0(x^{*})-u_0(y)\big)\,dy.$$
For the near term, \eqref{eq:local-holder} yields, for some $\alpha > 1/2$,
\begin{align}
I_{\mathrm{near}}(t)
&\le C_0\int_{|y-x^{*}|\le r_0} G_t(x^{*},y)\,|y-x^{*}|^\alpha\,dy
\le C_0\int_0^1 G_t(x^{*},y)\,|y-x^{*}|^\alpha\,dy \nonumber \\
&\le C_{T,0} \int_{\mathbb R} p_t(z)\,|z|^\alpha\,dz = C_{T,0} t^{\alpha/2}\int_{\mathbb R}\frac{1}{\sqrt{4\pi}}e^{-u^2/4}|u|^\alpha\,du
=: C_{\alpha,T,0}\,t^{\alpha/2},  \label{eq:D1-near}
\end{align}
where we have used Lemma \ref{lem:Neumann-Gaussian-|x-y|}. In the calculations above, we use the subscript 0 in the constants to denote the dependence on $C_0$ from Assumption \ref{ass:uo-holder}.

For the far term, use $0\le u_0(x^{*})-u_0(y)\le 2\|u_0\|_\infty$ to obtain
\begin{align}
I_{\mathrm{far}}(t)
&\le 2\|u_0\|_\infty \int_{|y-x^{*}|>r_0} G_t(x^{*},y)\,dy
\le 2C_T \|u_0\|_\infty \int_{|z|>r_0} p_t(z)\,dz \le C_{T,r_0} e^{-c/t}, \label{eq:D1-far}
\end{align}
where we have used again Lemma \ref{lem:Neumann-Gaussian-|x-y|} and $c$ can be taken to be $r_0^2 / 4$.

The last term in \eqref{eq:D1-split} is nonzero only for Regime \ref{:dirichlet}, i.e. for the Dirichlet kernel. We set
\[
d:=\min\{x^{*},\,1-x^{*}\} \geq 0.
\]
Let $\tau$ be the hitting time of the set $\{0,1\}$ of a Brownian motion $\{B_t\}$. Using the probabilistic representation,
\[
\int_0^1 G_t^{\operatorname{D}}(x^{*},y)\,dy=\mathbb{P}_{x^{*}}(\tau>t),
\qquad
1-\int_0^1 G_t^{\operatorname{D}}(x^{*},y)\,dy=\mathbb{P}_{x^{*}}(\tau\le t).
\]
If \(\tau\le t\), then the Brownian motion has moved at least distance \(d\) from its starting point by time \(t\), hence
\begin{equation} \label{kernelbound}
\mathbb{P}_{x^{*}}(\tau\le t)
\le
\mathbb{P}_0\!\left(\sup_{0\le s\le t}|B_s|\ge d\right)
\le
2\,\mathbb{P}_0\!\left(\sup_{0\le s\le t} B_s \ge d\right) = 4 \Big(1-\Phi\!\big(\tfrac{d}{\sqrt{t}}\big)\Big) \le \frac{4}{\sqrt{2\pi}}\frac{\sqrt{t}}{d}\exp\!\left(-\frac{d^2}{2t}\right),
\end{equation}
where the equality follows by the reflection principle and the last estimate from Gaussian tail bound \(1-\Phi(a)\le \frac{1}{a\sqrt{2\pi}}e^{-a^2/2}\) for \(a>0\), with \(a=d/\sqrt{t}\). Here $\Phi$ denotes the CDF of a standard normal random variable.

Combining \eqref{eq:D1-near}, \eqref{eq:D1-far}, and \eqref{kernelbound}, we conclude that
\begin{equation}\label{eq:D1-final}
|D_1(t)|
\le C_{\alpha,T,0}\,t^{\alpha/2} + C_{T,r_0} e^{-c/t} + C_{x^*}\sqrt{t},
\qquad t\in(0,t_0],
\end{equation}
for small enough $t_0\in(0,1]$. In particular, since $\alpha>1/2$, as $ t\downarrow 0$,
\begin{equation}\label{eq:D1-o}
\frac{|D_1(t)|}{t^{1/4}}\xrightarrow{}0.
\end{equation}

Now we bound the drift term $D_2(t)$ in $L^2(\Omega)$. By Minkowski's integral inequality and nonnegativity of~$G$,
\begin{align}
\E\left(D_2(t)^2 \right)^{\frac{1}{2}}
&\le \int_0^t \int_0^1  \E\left( \left(G_{t-s}(x^{*},y)\,b(u(s,y))\, \right)^2 \right)^{\frac{1}{2}} dy ds \notag\\
&\le \int_0^t \int_0^1 G_{t-s}(x^{*},y)\,\E\left((b(u(s,y)))^2\right)^{\frac{1}{2}}\,dy\,ds \nonumber \\
&\le C\int_0^t \int_0^1 G_{t-s}(x^{*},y) \left(  |b(0)| + \Lip(b) \E \left((u(s,y))^2\right)^{\frac{1}{2}} \right) dy\,ds,
\label{eq:D2-minkowski}
\end{align}
where we used in the last line the Lipschitz growth of $b$ we have $|b(z)|\le |b(0)|+\Lip(b)|z|$. A standard argument for the stochastic heat equation with Lipschitz coefficients gives us
\begin{equation} \label{eq:sup-bound-E-u}
\sup_{s\in[0,T]}\sup_{y\in[0,1]}\E\left((u(s,y))^2 \right)^{\frac{1}{2}}\le C_T.
\end{equation}
Plugging this into \eqref{eq:D2-minkowski} and using $\int_0^1 G_{t-s}(x^{*},y)\,dy\le 1$,
\begin{equation}\label{eq:D2-final}
\E\left(D_2(t)^2 \right)^{\frac{1}{2}}
\le \int_0^t C_{T,b} \left(\int_0^1 G_{t-s}(x^{*},y)\,dy\right)\!ds
\le C_{T,b}\,t.
\end{equation}
Combining \eqref{eq:D1-final} and \eqref{eq:D2-final}, since $D_1(t)$ is deterministic, yields
\begin{equation}\label{eq:D-L2-bound}
\E\left(D(t)^2\right)^{\frac{1}{2}} \le |D_1(t)| + \E \left( D_2(t)^2 \right)^{\frac{1}{2}} \le C_{\alpha,T,0}\,t^{\alpha/2} + C_{T,r_0} e^{-c/t} + C_{x^*}\sqrt{t}+ C_{T,b}t,
\qquad t\in(0,t_0].
\end{equation}

Fix $c_1>0$. By Chebyshev's inequality and \eqref{eq:D-L2-bound} since $\alpha>\frac{1}{2}$ and as $t \downarrow 0$,
\begin{align}
\PP\Big(|D(t)|\ge \tfrac14 c_1 t^{1/4}\Big)
&\le \frac{\E \left(D(t)^2\right)}{(c_1 t^{1/4})^2}
\rightarrow 0, \quad \text{ hence }\quad \PP\Big(|D(t)|\ge \tfrac14 c_1 t^{1/4}\Big)\le \frac{c_2}{4},
\label{eq:D-prob-small}
\end{align}
for the prescribed $c_2\in(0,1),$ after possibly shrinking $t_0$.

We proceed with estimating the term $R(t)$. By It\^o isometry and Lipschitz continuity of $\sigma$,
\begin{align}
 \mathbb{E}|R(t)|^2
&\le
\Lip(\sigma)^2
\int_0^t \int_0^1
G_{t-s}(x^{*},y)^2
\,
\mathbb{E}|u(s,y)-u_0(y)|^2
\,dy\,ds \nonumber \\
&\le
\Lip(\sigma)^2
\int_0^t \int_{|y-x^{*}|<\frac{r_0}{2}}
G_{t-s}(x^{*},y)^2
\,
\mathbb{E}|u(s,y)-u_0(y)|^2
\,dy\,ds \nonumber \\
&\qquad+
\Lip(\sigma)^2
\int_0^t \int_{|y-x^{*}|\geq \frac{r_0}{2}}
G_{t-s}(x^{*},y)^2
\,
\mathbb{E}|u(s,y)-u_0(y)|^2
\,dy\,ds \label{eq:Rt-123}
\end{align}
For the last term in \eqref{eq:Rt-123}, 
\begin{align}
    \int_0^t \int_{|y-x^{*}|\geq \frac{r_0}{2}}
G_{t-s}(x^{*},y)^2
\,
\mathbb{E}|u(s,y)-u_0(y)|^2
\,dy\,ds &\le  C_T \int_0^t \int_0^1 \frac{1}{t-s} \exp \left( -2\frac{r_0^2}{t-s}  \right)  dy ds \nonumber \\
&\le C_T \int_0^t \frac{1}{s} \exp \left( -\frac{r_0^2}{s}  \right) ds \nonumber \\
&\le C_T \exp\left( - r_0^2 / t \right) \frac{t}{r_0^2} \nonumber \\
&\le C_{T,r_0 } t \exp\left( - r_0^2 / t \right), \label{eq:Rt-far} 
\end{align}
where the first inequality follows by using \eqref{eq:sup-bound-E-u} and extending the $dy$ integration to $[0,1]$. 

For $|y-x^{*}|< \frac{r_0}{2}$ we have that $\mathbb{E}|u(s,y)-u_0(y)|^2 \leq C_T s^{1/2}$. Indeed, we can decompose it as
\begin{equation} \label{eq:Rt-100}
\begin{split}
u(s,y)-u_0(y)&=\int_{|z-y|<\frac{r_0}{2}} G_{s}(y,z)u_0(z)dz -u_0(y)+ \int_{|z-y| \geq \frac{r_0}{2}} G_{s}(y,z)u_0(z)dz \\
&+\int_0^s \int_0^1 G_{s-r}(y,z)b(u(r,z))dz dr 
+\int_0^s \int_0^1 G_{s-r}(y,z)\sigma(u(r,z)) dW(r,z).
\end{split}
\end{equation}
Then, for the stochastic term
\begin{equation*}
\begin{split}
   \mathbb{E} \left(\int_0^s \int_0^1 G_{s-r}(y,z)\sigma(u(r,z)) dW(r,z) \right)^2 &= \mathbb{E} \left(\int_0^s \int_0^1 G_{s-r}^2(y,z)\sigma^2(u(r,z)) dr dz \right) \\
   &\leq C_\sigma \int_0^s \int_0^1 G_{s-r}^2(y,z) dr dz  \leq C_{T,\sigma} \sqrt{s}. 
\end{split}
\end{equation*}
The other terms in \eqref{eq:Rt-100} can be estimated by calculations identical to the ones for the terms $D_1$ and $D_2$ above, verifying that $\mathbb{E}|u(s,y)-u_0(y)|^2 \leq C_{T,\sigma,b} s^{1/2}$.

Plugging this in the term in the second line of \eqref{eq:Rt-123}, 
\begin{equation} \label{eq:Rt-near}
\begin{split}
&\int_0^t \int_{|y-x^{*}|<r_0}
G_{t-s}(x^{*},y)^2
\,
\mathbb{E}|u(s,y)-u_0(y)|^2
\,dy\,ds \\
&\leq C_{T,\sigma,b} \int_0^t \int_{|y-x^{*}|<r_0}
G_{t-s}(x^{*},y)^2 s^{1/2}
\,dy\,ds \leq C_{T,\sigma,b} \int_0^t s^{1/2} \int_0^1
G_{t-s}(x^{*},y)^2 
\,dy  \,ds \\
&\leq C_{T,\sigma,b} \int_0^t (t-s)^{-1/2} s^{1/2} \,ds= C_{T,\sigma,b} t.
\end{split}
\end{equation}
The last line uses Lemma \ref{lem:G-L2} and standard beta function calculations (cf. \eqref{eq:beta-example-2}). Hence, combining \eqref{eq:Rt-123}, \eqref{eq:Rt-far}, and \eqref{eq:Rt-near}, we get that
\begin{equation*}
     \mathbb{E}|R(t)|^2 \le C_{T,r_0,b,\sigma} \sqrt{t.}
\end{equation*}

Hence, by Chebyshev, for the same constant $c_1$ as in \eqref{eq:zt-lower-bound},
\begin{equation} \label{eq:R-estimate-prob}
\mathbb{P}
\big(
|R(t)| \ge \tfrac14 c_1 t^{1/4}
\big)
\leq C t^{1/2} \rightarrow 0, \text{ hence } \PP\Big(|R(t)|\ge \tfrac14 c_1 t^{1/4}\Big)\le \frac{c_2}{4},
\text{ for all  $t \leq t_0$ for some $t_0>0$.}
\end{equation}

We can conclude that for $t$ sufficiently small,
\[
\mathbb{P}
\Big(
Z(t)\ge c_1 t^{1/4},\,
|R(t)|\le\tfrac14 c_1 t^{1/4},\,
|D(t)|\le\tfrac14 c_1 t^{1/4}
\Big)
\ge \frac{c_2}{2}>0.
\]
Thus, from \eqref{eq:u-decomp-D-Z-R}, there exist $p_*>0$ and $t_0>0$ such that
\begin{equation*} 
\mathbb{P}
\big(
u(t,x^{*}) > u_0(x^{*})
\big)
\ge p_*,
\qquad
\forall t\in(0,t_0].
\end{equation*}
We have that, recalling the set $A$ defined in \eqref{eq:A-def},
\begin{align}
    \PP(A) &:= \PP(\left\{
u(t_n,x^{*}) > u_0(x^{*})
\text{ i.o.}
\right\})=\PP( \limsup \left\{u(t_n,x^{*}) > u_0(x^{*}) \right\}) \nonumber \\
&\geq \limsup_{n \rightarrow \infty} \PP( u(t_n, x^{*})> u_0(x^{*}))\geq p_{*} . \label{eq:prob-i.o.-event}
\end{align}
Hence, by the 0-1 law for the germ $\sigma$-algebra we obtain that $\PP(A)=1$ and thus \[
\mathbb{P}
\left(
\sup_{t\in(0,T],\,x\in[0,1]} u(t,x) > u_0(x^*)
\right)
= 1,
\]
which finishes the proof.
\end{proof}



\begin{proposition} \label{dirichletmaximizeratboundary}
Suppose Regime \ref{:dirichlet} holds. Then, if $x^* \in \{0,1\}$, we have that 
    \[
    \PP \left( \sup_{(t,x) \in [0,T] \times [0,1]} u(t,x) = 0 \right) = 0.
    \]
\end{proposition}
\begin{proof}
Assume that without loss of generality, $x^{*}=0$, hence $u_0$ is one-sided H\"older at $0$ with exponent $\alpha>1/2$ (cf. \eqref{eq:local-holder} with the one sided bound).
We choose a moving interior point $x(t)=t^{\theta}$ and choose $\theta$ such that
$\frac{1}{4\alpha}<\theta<\frac12.$
Then, since $\alpha\theta>1/4$, as $t \to 0$,
\begin{equation}\label{eq:u0xt-small}
    |u_0(x(t))|\le C_0 x(t)^\alpha=C_0 t^{\alpha\theta}=o(t^{1/4}), \quad \frac{x(t)}{\sqrt t}=t^{\theta-\frac12}\xrightarrow{}\infty, \quad \text{as }t\downarrow0.
\end{equation}

For each $t>0$ we write, at the point $x(t)$,
\begin{equation}\label{eq:main-decomp}
u(t,x(t))= D_1(t)+D_2(t)+Z(t)+R(t),
\end{equation}
where the terms $D_1(t),D_2(t),Z(t),R(t)$ are given in \eqref{eq:D-Z-R-def}--\eqref{eq:D-def} with $G^{\operatorname{D}}$ replacing $G$ and $x(t)$ replacing $x^*$.

Estimates for the terms $D_2(t)$ and $R(t)$ are analogous to the ones in Proposition~\ref{dirichletneumannmaximizeratinitial} (cf.\ \eqref{eq:D2-final} and \eqref{eq:R-estimate-prob}, with $x(t)$ in place of $x^*$). The estimate for $D_1(t)$ requires a separate argument, since here the evaluation point $x(t)$ is moving and is not itself a maximizer of $u_0$.

Since $x^*=0$ is a maximizer and $u_0(0)=0$, we have $u_0(y)\le 0$ for all $y\in[0,1]$. Hence
\begin{equation}\label{eq:D1-moving-start}
|D_1(t)|
\le
\left|\int_0^1 G_t^{\operatorname{D}}(x(t),y)u_0(y)\,dy\right|
+
|u_0(x(t))|.
\end{equation}
The second term is already controlled by \eqref{eq:u0xt-small}.
We now estimate the integral term in \eqref{eq:D1-moving-start}. Split
\[
\int_0^1 G_t^{\operatorname{D}}(x(t),y)|u_0(y)|\,dy
=
\int_0^{2x(t)} G_t^{\operatorname{D}}(x(t),y)|u_0(y)|\,dy
+
\int_{2x(t)}^{1} G_t^{\operatorname{D}}(x(t),y)|u_0(y)|\,dy
=: I_{\rm near}(t)+I_{\rm far}(t).
\]

For the near term, since $0\le y\le 2x(t)$ and $u_0$ is one-sided H\"older at $0$ with exponent $\alpha$, we have
\[
|u_0(y)|=|u_0(y)-u_0(0)|\le C_0 y^\alpha \le C\,x(t)^\alpha.
\]
Therefore, using that $\int_0^1 G_t^{\operatorname{D}}(x(t),y)\,dy\le 1$,
\begin{equation}\label{eq:D1-moving-near}
I_{\rm near}(t)
\le
C\,x(t)^\alpha \int_0^1 G_t^{\operatorname{D}}(x(t),y)\,dy
\le
C\,x(t)^\alpha
=
C\,t^{\alpha\theta}
=
o(t^{1/4}).
\end{equation}

For the far term, using the upper bound from Lemma~\ref{lem:Neumann-Gaussian-|x-y|},
\begin{multline}
   I_{\rm far}(t)
\le
\|u_0\|_\infty \int_{2x(t)}^1 G_t^{\operatorname{D}}(x(t),y)\,dy
\le
C_T \|u_0\|_\infty \int_{2x(t)}^1 p_t(x(t)-y)\,dy \\
\le C_T \int_{x(t)}^\infty p_t(z)\,dz
\le
C_T \exp\!\left(-c\frac{x(t)^2}{t}\right)
=
C_T \exp\!\left(-c\,t^{2\theta-1}\right), \label{eq:D1-moving-far}
\end{multline}
for some $c>0$. The second line follows since $y\ge 2x(t)$ implies $|y-x(t)|\ge x(t)$
 Because $\theta<1/2$, we have $x(t)/\sqrt t=t^{\theta-1/2}\to\infty$, and thus $\exp\!\left(-c\frac{x(t)^2}{t}\right)=o(t^{1/4})$.

Combining \eqref{eq:D1-moving-start}, \eqref{eq:D1-moving-near}, and \eqref{eq:D1-moving-far}, we conclude that, as $t\downarrow0$,
\begin{equation}\label{eq:D1-moving-final}
|D_1(t)|=o(t^{1/4}).
\end{equation}


It now remains to bound the term $Z(t)$, which is a centered Gaussian since the integrand is deterministic. By the uniform ellipticity of $\sigma$ and Lemma \ref{lem:G-L2},
\[
\mathrm{Var}(Z(t))
\ge (C_\sigma)^{-2} \int_0^t\int_0^1 G_r^{\operatorname{D}}(x(t),y)^2\,dy\,dr 
\ge (C_\sigma)^{-2} \int_0^t c r^{-1/2}\,dr
= c (C_\sigma)^{-2}\, t^{1/2}.
\]
Hence there exist constants $c_1,c_2>0$ and $t_0\in(0,1)$ such that
for all $t\in(0,t_0]$,
\begin{equation}\label{eq:Z-tail}
\mathbb P\big(Z(t)\ge c_1 t^{1/4}\big)\ge c_2.
\end{equation}
We can now get the desired lower bound for $\mathbb P(u(t,x(t))>0)$.

Fix $c_1,c_2$ from \eqref{eq:Z-tail}. Choose $t_0>0$ small so that, for $t \in (0,t_0]$,:
$$
|u_0(x(t))|\le \frac{1}{16} c_1 t^{1/4}, \;
|D_1(t)|\le \frac{1}{16} c_1 t^{1/4}, \;
\mathbb P\big(|D_2(t)|>\tfrac14 c_1 t^{1/4}\big)\le \frac{c_2}{4},  \;
\mathbb P\big(|R(t)|>\tfrac14 c_1 t^{1/4}\big)\le \frac{c_2}{4}.
$$

Define the event
\[
E_t:=\Big\{Z(t)\ge c_1 t^{1/4}\Big\}
\cap \Big\{|D_2(t)|\le \tfrac14 c_1 t^{1/4}\Big\}
\cap \Big\{|R(t)|\le \tfrac14 c_1 t^{1/4}\Big\}.
\]
On $E_t$, using \eqref{eq:main-decomp} and the deterministic bounds on $u_0(x(t))$ and $D_1(t)$,
\[
u(t,x(t))
\ge -\frac{1}{16} c_1 t^{1/4}-\frac{1}{16} c_1 t^{1/4}-\frac14 c_1 t^{1/4}
+ c_1 t^{1/4}-\frac14 c_1 t^{1/4}= \frac{1}{4}c_1 t^{\frac{1}{4}}>0.
\]
Moreover,
\begin{align*}
\mathbb P(E_t)
&\ge \mathbb P\big(Z(t)\ge c_1 t^{1/4}\big)
-\mathbb P\big(|D_2(t)|>\tfrac14 c_1 t^{1/4}\big)
-\mathbb P\big(|R(t)|>\tfrac14 c_1 t^{1/4}\big)\\
&\ge c_2-\frac{c_2}{4}-\frac{c_2}{4}=\frac{c_2}{2}.
\end{align*}
Thus we obtain the uniform lower bound
\begin{equation*}
\exists\,p_\star>0\ \text{such that}\ \mathbb P\big(u(t,x(t))>0\big)\ge p_\star
\quad\forall\,t\in(0,t_0],
\qquad p_\star:=\frac{c_2}{2}.
\end{equation*}
Once more, 
$$
\PP(A) := \PP(\left\{
u(t_n,x(t_n))> 0
\text{ i.o.}
\right\})=\PP( \limsup \left\{u(t_n,x(t_n))> 0 \right\}) \geq \limsup_{n \rightarrow \infty} \PP( \left\{u(t_n, x(t_n))>0 \right\})\geq p_{*}>0.
$$
Hence, by the 0-1 law for the germ $\sigma$-algebra we obtain that $\PP(A)=1$ and thus $\mathbb{P}
\left(
\sup_{t\in(0,T],\,x\in[0,1]} u(t,x) > 0
\right)
= 1$.
\end{proof}




\section{Proof of the theorem for $\kappa > 0$} \label{sec:kappa-neq-0}

In this section we treat the proof of Theorem \ref{thm:4th-order} under Regime \ref{case:fourthorder}, i.e., with $\kappa > 0$. In that case we take $\kappa=1 $ and $ \rho \ge 0$ in \eqref{sCH} without loss of generality.

The proof of Theorem \ref{thm:4th-order} parallels the proof of Theorem \ref{maintheorem} under Regime \ref{case:neumann}, but we use the alternative lemmas that were developed in Section \ref{subsec:kappa>0}. We outline the proof here, but emphasize the differences in the proof in detail.

We have to prove the conditions in Theorem \ref{thm:nualart-criterion-existence}. Condition (i) follows as in Section \ref{sec:proof-kappa=0} by splitting the solution to $u$ to its deterministic and stochastic components and bounding these separately by using the continuity of $u_0$ and an application of Kolmogorov's criterion respectively. In section \ref{subsec:second-cond-fourth} we prove condition (ii) by means of a Kolmogorov theorem, thereby obtaining that $\sup_{(t,x) \in [0,T] \times [0,1]} u(t,x) \in \mathbb{D}^{1,2}$. Condition (iii) is the subject of Section \ref{subsec:third-cond-fourth}. In particular, Condition (iii) for all compact sets $K \subseteq (0,T] \times [0,1]$ is shown in Section \ref{subsubsec:compact-subsets-fourth}. When taking $K = [0,T] \times [0,1]$, Condition (iii) is proved in Section \ref{subsubsec:upgrading-whole-spacek-k>0}.

\subsection{ Proof of the second Condition} \label{subsec:second-cond-fourth}
For the second condition we proceed exactly as in Section \ref{sec:proof-kappa=0}, obtainig the terms $I_1,I_2, I_3$, with the Green's function $H$ replacing $G$. We use Lemma \ref{greenincrementsintegral} to bound the terms $I_1, I_2$ in the same way as we did in Section \ref{sec:proof-kappa=0}. For any $0<\alpha< \frac{3}{4}$,
\begin{eqnarray*}
I_1&=&\Big(\int_{0}^{T} \int_0^1 |H_{t-s}(x,y)-H_{\bar{t}-s}(\bar{x},y)|^{2} dy ds \Big)^{\frac{p}{2}} \leq C_{\gamma, p, \rho} \left( |\bar{x}-x|^{2}+|\bar{t}-t|^{\alpha} \right)^{\frac{p}{2}} ,
\end{eqnarray*}
and
\begin{align*}
I_2 & = \E\left[\left( \int_0^T \int_0^1 \left(\int_s^t\int_0^1 \left( H_{t-\theta}(x,r)-H_{\bar{t}-\theta}( \bar{x},r)\right) m(r,\theta) D_{s,y}u(\theta,r) drd \theta \right)^2 dy ds \right)^{\frac{p}{2}}\right] \nonumber \\
& \leq C_{b,\alpha,T,p,\rho} \left(|t-\bar{t}|^{\alpha}+|x-\bar{x}|^{2}\right)^{\frac{p}{2}}\E\left[\left( \int_{0}^{T} \int_{0}^{1} \int_{s}^{t}\int_{0}^{1}|D_{s,y}u(\theta,r)|^2 drd \theta dy ds \right)^{\frac{p}{2}}\right]. 
\end{align*}
Using Minkowski integral inequality for $p\ge 2$ and Lemma \ref{lmalden} 
\begin{align}
\mathbb{E}\left[\left( \int_0^T \int_0^1 \int_s^t\int_0^1|D_{s,y}u(\theta,r)|^2 drd \theta dy ds \right)^{\frac{p}{2}}\right]  \nonumber
& \leq C_{p,T} \left(\int_0^T \int_0^1 \int_s^t \left(\frac{1}{\theta-s}\right)^{1/2} d \theta dy ds \right)^{\frac{p}{2}} \nonumber \leq C_{p,T}. \label{eq:i2-3}
\end{align}

So we can conclude that 
\begin{equation*}
I_2 \leq C_T\left(|t-\bar{t}|^{\alpha}+|x-\bar{x}|^{2}\right)^{\frac{p}{2}} .
\end{equation*}

For the term $I_{3}$ as in Section \ref{sec:proof-kappa=0} we apply Minkowski integral inequality, the Burkholder-Davis-Gundy inequality respectively with $p \geq 2$ and use Lemmas \ref{lmalden} and \ref{lemmagdif} to obtain
$$
\begin{aligned}
& I_{3}=\E
\left(\left[\int_0^{t} \int_0^1 
\left(\int_s^{t}\int_0^1 \left(H_{t-\theta}(x,r)-H_{\bar{t}-\theta}(\bar{x},r)\right) \, \hat{m}(r,\theta) \, D_{s,y}u(\theta,r) \, W(dr,d\theta)\right)^2 dy ds\right]^{\frac{p}{2}}\right)
\\
& \le C_{T,p,\sigma} \left[\int_0^t \int_0^1 \int_s^{t}\int_0^1
\left( |H_{t-\theta}(x,r)-H_{\bar t-\theta}(\bar x,r)|^p \, \E\left[\big|D_{s,y} u(\theta,r)\big|^p\right] \right)^{\frac{2}{p}}dr\,d\theta dy ds \right]^{\frac{p}{2}}\\
& \le C_{T,p,\sigma} \left[\int_0^t \int_0^1 \int_s^{t}\int_0^1
\left(H_{t-\theta}(x,r)-H_{\bar t-\theta}(\bar x,r)\right)^2 \frac{1}{(\theta-s)^{\frac{1}{2}}}\,dr\,d\theta dy ds \right]^{\frac{p}{2}} \\
&\leq C_{T,p,\sigma\rho} \left[\int_0^t \int_0^1 \frac{\left(|t-\bar{t}|^{\alpha}+|x-\bar{x}|^2 \right)}{\sqrt{t-s}} dy ds \right]^{\frac{p}{2}} \\
&\leq C_{T,p,\sigma,\rho} \left(|t-\bar{t}|^{\alpha}+|x-\bar{x}|^{2}\right)^{\frac{p}{2}} .
\end{aligned}
$$

Combining those bounds for $I_1,I_2,I_3$ we obtain that for any $0<\alpha < \frac{3}{4}$
\begin{equation} \label{eq:malliavin-difference-L2-4thorder}
 \mathbb{E}\left[\left( \int_{0}^{T} \int_{0}^{1} |D_{s,y} u(t,x)- D_{s,y}u(\bar{t}, \bar{x})|^2 dy ds  \right)^{\frac{p}{2}}\right] \leq C_{T,p,b,\sigma,\alpha,\rho} \left(|x-\bar{x}| +|t-\bar{t}|^{\frac{\alpha}{2}}\right)^{p}.
\end{equation}

Since this holds for any $\alpha < \frac{3}{4}$ it thus suffices to choose $p> 1+\frac{2}{\alpha}> 1+\frac{8}{3}=\frac{11}{3}$ to obtain the desired Kolmogorov criterion estimate. We can conclude that the process $\left\{D_{s,y} u(t,x), t \in (s,T), x \in (0,1)\right\}$ possesses a continuous version as an $L^{2}(dy \, ds)$-valued process. 
We also have, after fixing some $(t_0,x_0) \in [0,T] \times [0,1]$, and using Minkowski integral inequality and Lemma \ref{lmalden} 
\begin{equation*}
\E \|D_{\cdot,\cdot} u(t_0,x_0)\|^p_{L^2([0,T] \times [0,1])} = \E \left( \int_0^T \int_0^1 | D_{s,y} u(t_0,x_0) |^2 dy ds \right)^{p/2} 
\le C \left( \int_0^{t_0} \int_0^1 \frac{1}{(t_0-s)^{\frac{1}{2}}} dy ds   \right)^{p/2} \leq C_{t_0}.
\end{equation*}
We can thus conclude that, by Theorem \ref{thm:kolmogorov}
\begin{equation} 
    \E \left( \sup_{(t,x) \in [0,T] \times [0,1]} \int_0^T \int_0^1 | D_{s,y} u(t,x)|^2 dy ds \right) < \infty,
\end{equation}
and thus the second condition for the case $\kappa >  0$ is proved.

\subsection{Proof of the third condition} \label{subsec:third-cond-fourth} For the proof of the third condition of Theorem \ref{thm:nualart-criterion-existence} we follow the same steps as in Section \ref{sec:3rdcondk1}. In Sections \ref{subsubsec:holder-gamma-fourth} and \ref{subsubsec:small-ball-gamma-fourth} we develop some needed preliminary lemmas, while we prove the condition for compact sets $K \subset (0,T] \times [0,1]$ in Section \ref{subsubsec:compact-subsets-fourth}. We upgrade this to $K = [0,T] \times [0,1]$ in Section \ref{subsubsec:upgrading-whole-spacek-k>0}.

\subsubsection{H\"older continuity of $\gamma(t,x)$} \label{subsubsec:holder-gamma-fourth}
Proceeding as in \ref{lemma-gamma-holder} we can show that for $\kappa > 0$, $\gamma(t,x)=\int_0^T \int_0^1 |D_{s,y}u(t,x)|^2\,dy\,ds$ is $(\mu_1,\mu_2)-$H\"older continuous in $(t,x)$ for all $\mu_1 < \frac{3}{8}$ and all $\mu_2< 1$.

Indeed we can write 

\begin{multline*}
\E\Bigg|
\int_0^T\!\!\int_0^1\left(|D_{s,y}u(t,x)|^2-|D_{s,y}u(\bar t,\bar x)|^2\right)\,dy\,ds
\Bigg|^p
\le
\Bigg(\E\left(\int_0^T\!\!\int_0^1 J_1^2\,dy\,ds\right)^p\Bigg)^{1/2}
\Bigg(\E\left(\int_0^T\!\!\int_0^1 J_2^2\,dy\,ds\right)^p\Bigg)^{1/2},
\end{multline*}
where $J_1 := D_{s,y}u(t,x) + D_{s,y}u(\bar t,\bar x)$, $J_2 := D_{s,y}u(t,x) - D_{s,y}u(\bar t,\bar x)$.
For the $J_2$ term, we get for $0 < \alpha < 3/4$
\begin{equation*}
\begin{split}
&\Bigg(\E\left(\int_0^T\!\!\int_0^1 J_2^2\,dy\,ds\right)^p\Bigg)^{1/2} = \mathbb{E}\left[\left( \int_{0}^{T} \int_{0}^{1} |D_{s,y} u(t,x)- D_{s,y}u(\bar{t}, \bar{x})|^2 dy ds  \right)^{p}\right]^{1/2}  \\
&\le \left( C_{T,p,b,\sigma,\alpha,\rho} \left(|x-\bar{x}|+|t-\bar{t}|^{\alpha/2}\right)^{2p} \right)^{1/2} \le C_{T,p,b,\sigma,\alpha,\rho} \left(|x-\bar{x}|+|t-\bar{t}|^{\alpha/2}\right)^{p}.
\end{split}
\end{equation*}
Moreover, applying Minkowski integral inequality and Lemma \ref{lmalden},
\begin{flalign*} 
\Bigg(\E\left(\int_0^T\!\!\int_0^1 J_1^2\,dy\,ds\right)^p\Bigg)^{1/2} &= \mathbb{E}\left[\left( \int_{0}^{T} \int_{0}^{1} |D_{s,y} u(t,x) + D_{s,y}u(\bar{t}, \bar{x})|^2 dy ds  \right)^{p}\right]^{1/2}
\nonumber\\
&\le C_p \left(  \int_0^T \int_0^1  \left( (\E | D_{s,y} u(t,x)|^{2p})^{1/p} + (\E | D_{s,y} u(\bar{t}, \bar{x})|^{2p})^{1/p}  \right) dy ds     \right)^{p/2}
\nonumber\\
&\le C_{p,T} \left(  \int_0^t \int_0^1  \frac{1}{(t-s)^{\frac{1}{2}}} dy ds + \int_0^{\bar{t}} \int_0^1 \frac{1}{(\bar{t}-s)^{\frac{1}{2}}} dy ds     \right)^{p/2} \le C_{p,T}.
\end{flalign*}
We can thus conclude 
\begin{equation*}
 \E\Bigg|
\int_0^T\!\!\int_0^1\left(|D_{s,y}u(t,x)|^2-|D_{s,y}u(\bar t,\bar x)|^2\right)\,dy\,ds
\Bigg|^{p} \
\le C_{T,p,b,\sigma,\alpha,\rho}\left(|x-\bar{x}| +|t-\bar{t}|^{\alpha /2}\right)^p
\end{equation*}
and obtain the H\"older continuity  by another application of Kolmogorov's criterion.

\subsubsection{Small ball estimate for $\gamma(t,x)$} \label{subsubsec:small-ball-gamma-fourth}
Recall the regions associated with Neumann boundary conditions
\begin{equation*}
        L_\delta = [\delta,T] \times [0,1].
\end{equation*}

We can also obtain the small ball estimate for $\gamma(t,x)$ as in Lemma \ref{smallballestimatestochasticheat}. For any $\varepsilon>0$ we have the lower bound 

\begin{equation} \label{lowerboundmalliavincovariance4thorder}
\begin{split}
\gamma(t,x) \ge \tfrac12 \mathcal{R}_{1}(\varepsilon;t,x)
- \mathcal{R}_{2}(\varepsilon;t,x),
\end{split}
\end{equation}
with
\begin{equation*}
\begin{split}
\mathcal{R}_1&:= \int_{t-\varepsilon}^t\; \int_0^1
\sigma(u(s,y))^2 H_{t-s}(x,y)^2\,dy\,ds,\\
\mathcal{R}_2 &:= \int_{t-\varepsilon}^t\;\!\int_0^1 \left( \int_{s}^{t}\int_{0}^{1}H_{t-\theta}(x,r)m(r,\theta)D_{s,y}u(\theta,r) drd \theta \right. \\
&\qquad \qquad+ \left. \int_{s}^{t}\int_{0}^{1} H_{t-\theta}(x,r) \hat{m}(r,\theta)D_{s,y}u(\theta,r) W(dr,d \theta)\right)^2 dy\,ds.
\end{split}
\end{equation*}
Uniform ellipticity and a lower bound for the $L^2$ notm of $H$ (see for instance (3.24) from \cite{AFK}) yield
\begin{equation*}
    \mathcal{R}_1\ge (C_\sigma)^{-2}\int_{t-\varepsilon}^t\;\int_0^1
H_{t-s}(x,y)^2\,dy\,ds \geq (C_\sigma)^{-2} C \varepsilon^{\frac{3}{4}} .
\end{equation*}
We set $c_2=(C_\sigma)^{-2} C>0. $
We can also bound the moments of $\mathcal{R}_2$ and show that for any $q\ge1$, there exists some $C_q > 0$ such that
\begin{equation} \label{eq:r2-estimate-c2}
    \sup_{(t,x)\in L_\delta}
\E\!\left[\clr_2(\varepsilon;t,x)^q\right]
\le C_{q,T,b,\sigma} \,\varepsilon^{5q/4}.
\end{equation}
We have 
\begin{equation}
 \clr_2(\varepsilon;t,x) \le 2\int_{t-\varepsilon}^t\int_0^1
 \clr_{2,1}(s,y;t,x)^2 \,dy\,ds 
+2\int_{t-\varepsilon}^t\int_0^1
 \clr_{2,2}(s,y;t,x)^2 \,dy\,ds, \label{eq:r2-r21-r224thorder}
\end{equation}
where
\begin{eqnarray*}
\clr_{2,1}(s,y;t,x)
&:=&
\int_s^t\int_0^1
H_{t-r}(x,z)\,
m(r,z)\,
D_{s,y}u(r,z)\,dz\,dr,
\nonumber\\
\clr_{2,2}(s,y;t,x)
&:=&
\int_s^t\int_0^1
H_{t-r}(x,z)\,
\hat{m}(r,z)\,
D_{s,y}u(r,z)\,
W(dz,dr).
\end{eqnarray*}
Using Corollary \ref{4thgreenbounds}, we bound
\begin{align*}
\clr_{2,1}(s,y;t,x)^2 &\le C_b
\left(
\int_s^t
\|H_{t-r}\|_{L^2([0,1])}^2\,dr
\right)
\left(
\int_s^t
\|D_{s,y}u(r,\cdot)\|_{L^2([0,1])}^2\,dr
\right) \\
&= C_b \left( \int_0^{t-s}
\|H_{\tau}\|_{L^2([0,1])}^2\,d\tau \right) \left(
\int_s^t
\|D_{s,y}u(r,\cdot)\|_{L^2([0,1])}^2\,dr
\right) \\
&\le C_{b,T}
\varepsilon^{3/4} \left(
\int_s^t
\|D_{s,y}u(r,\cdot)\|_{L^2([0,1])}^2\,dr
\right).
\end{align*}
Thus, 
\begin{equation*}
\begin{split}
    \int_{t-\varepsilon}^t\;\int_0^1
 \clr_{2,1}(s,y;t,x)^2 \,dy\,ds 
&\le
 C_{T,b}
\varepsilon^{3/4}
\int_{t-\varepsilon}^t
\int_{t-\varepsilon}^r
\int_0^1
\|D_{s,y}u(r,\cdot)\|_{L^2([0,1])}^2
\,dy\,ds\,dr,
\end{split}
\end{equation*}
where we have used Fubini's theorem. Moreover, using Minkowski's integral inequality and Lemma \ref{lmalden}
\begin{equation*} 
\begin{split}
&\sup_{r\le T}
\E\left[
\left(
\int_0^r\int_0^1 
\|D_{s,y}u(r,\cdot)\|^2_{L^2([0,1])}  dy\,ds
\right)^q
\right]
\leq \sup_{r\le T}
\left(
\int_0^r\int_0^1 \int_0^1 \E\Big[
|D_{s,y}u(r,z)|^{2q}\Big]^{\frac{1}{q}} dz dy\,ds
\right)^q \\
& \leq \sup_{r\le T} \left(
\int_0^r\int_0^1 \int_0^1 \E\Big[
|D_{s,y}u(r,z)|^{2q}\Big]^{\frac{1}{q}} dy\,ds \right)^{q} \leq \left(
\int_0^T \int_0^1 \int_0^1 \frac{1}{(r-s)^{\frac{1}{2}}} dz dy\,ds
\right)^q \leq C_{T,q} .
\end{split}
\end{equation*}
Since $r-s\le\varepsilon$, we obtain, using again Minkowski integral inequality,
\begin{equation} \label{1eq:R21-estimate}
\begin{split}
\E \left( \int_{t-\varepsilon}^t\;\int_0^1
 \clr_{2,1}(s,y;t,x)^2 \,dy\,ds \right)^q
&\le C_{T,b,q}\varepsilon^{3q/4}
\E \left( \int_{t-\varepsilon}^t
\int_{t-\varepsilon}^r
\int_0^1
\|D_{s,y}u(r,\cdot)\|_{L^2([0,1])}^2
\,dy\,ds\,dr \right)^q \\ 
&\le C_{T,b,q}\varepsilon^{3q/4}
\left( \int_{t-\varepsilon}^t
C_{T,q} dr \right)^q \leq C_{T,b,q}\varepsilon^{7q/4} \le C_{T,b,q} \varepsilon^{5q/4}.
\end{split}
\end{equation}
To estimate $\clr_{2,2}$, proceeding as in Lemma \ref{smallballestimatestochasticheat} we have
\begin{align}
\E &\left( \int_{t-\varepsilon}^t\;\int_0^1
 \clr_{2,2}(s,y;t,x)^2 \,dy\,ds \right)^q  \nonumber\\
 &\le C_{\sigma,q} \left(\int_{t-\varepsilon}^t \int_0^1  \int_s^t \int_0^1 \E \left[H_{t-r}(x,z)^{2q} D_{s,y}u(r,z)^{2q}\right]^{\frac{1}{q}} dz dr dy ds  \right)^{q} \nonumber \\
 &\le C_{\sigma,q} \left(\int_{t-\varepsilon}^{t} \int_0^1  \int_s^t \int_0^1 H_{t-r}(x,z)^2 \frac{1}{(r-s)^{\frac{1}{2}}} dz dr dy ds  \right)^q, \label{eq:6789p}
\end{align}
where the last line follows from Lemma \ref{lmalden}.

Then we get
\begin{equation*}
\begin{split}
    &\int_{t-\varepsilon}^{t} \int_0^1  \int_s^t \int_0^1 H_{t-r}(x,z)^{2} \frac{1}{(r-s)^{\frac{1}{2}}} dz dr dy ds \\
    &\qquad\le  \int_{t-\varepsilon}^{t} \int_0^1 H_{t-r}(x,z)^{2} \left( \int_{t-\varepsilon}^r  \int_0^1  \frac{1}{(r-s)^{\frac{1}{2}}}  dy ds \right)dz dr \\
    &\qquad\le C_{T} \int_{t-\varepsilon}^{t} \int_0^1 H_{t-r}(x,z)^{2} \sqrt{r-(t-\varepsilon)} dz dr \\
    &\qquad\leq C_{T} \sqrt{\varepsilon}  \int_{t-\varepsilon}^{t} \int_0^1 H_{t-r}(x,z)^{2} dz dr \\
    &\qquad\leq C_{T} \sqrt{\varepsilon}  \int_{t-\varepsilon}^{t} (t-r)^{-\frac{1}{4}} dr= C_{T}\varepsilon^{\frac{5}{4}}
\end{split}
\end{equation*}
which, combined with \eqref{eq:6789p}, leads to
\begin{equation} \label{eq:r22-estimate-2}
    \E \left( \int_{t-\varepsilon}^t\;\int_0^1
 \clr_{2,2}(s,y;t,x)^2 \,dy\,ds \right)^q\leq C_{T,\sigma,q} \varepsilon^{5q/4}.
\end{equation}
Combining \eqref{eq:r2-r21-r224thorder}, \eqref{1eq:R21-estimate}, and \eqref{eq:r22-estimate-2}, we get that 
$$ \E[\clr_2(\varepsilon;t,x) ^q] \leq C_{T,q,b,\sigma} \varepsilon^{5q/4} $$ for a constant $C_{T,q,b,\sigma}$ that does not depend on $(t,x)$. By Markov's inequality,
$$
\PP( \clr_2 > \varepsilon^{\frac{3}{4}}) \leq \frac{\E[\clr_2^q]}{\varepsilon^{\frac{3q}{4}}} \leq C_{T,q,b,\sigma} \varepsilon^{\frac{q}{2}}.
$$
Now for the constant $c_2$ introduced above \eqref{eq:r2-estimate-c2} we get from the previous equation that for all $\varepsilon < r(\delta)$,
$$
\PP\left( \gamma(t,x) \le  \frac{c_{2}}{4}\varepsilon^{\frac{3}{4}} \right) \le  \PP\left( \frac{1}{2} \clr_1 - \clr_2 \le  \frac{c_{2}}{4}\varepsilon^{\frac{3}{4}} \right) \leq \PP\left(\clr_2 \ge \frac{c_{2}}{4} \varepsilon^{\frac{3}{4}}\right) \leq C_{T,q,b,\sigma} \frac{4^{q}}{c_{2}^{q}}\varepsilon^{\frac{q}{2}} .
$$
This lets us conclude that there exists a constant $C_{T,q,b,\sigma}>0$ not depending on $t,x, $ such that, for any $0<y< \frac{c_2}{4}r(\delta)^{3/4}$, setting $ \varepsilon(y) := \left(\frac{4 y}{c_2}\right)^{\frac{4}{3}},$
\begin{multline*}
\sup_{(t,x) \in L_{\delta}} \PP( \gamma(t,x) \leq y) \le \sup_{(t,x) \in L_{\delta}} \PP(  J_{\varepsilon(y)}(t,x) \leq y) 
\le  \sup_{(t,x) \in L_{\delta}} \PP(  J_{\varepsilon(y)}(t,x) \leq \frac{c_2}{4} \varepsilon(y)^{3/4}) \le C_{T,q,b,\sigma} \varepsilon(y)^{q/2}\\= C_{T,q,b,\sigma} y^{2q/3}.
\end{multline*}

\subsubsection{Deducing condition (iii) and existence of density for supremum over compact subsets} \label{subsubsec:compact-subsets-fourth}
Using the H\"older continuity and the small ball estimate, the proof of Proposition \ref{prop:Sdelta-K} holds verbatim for sets $K \subseteq (0,T]\times [0,1]$. This gives us that for any such compact set $K$
\[
\PP \Big(
\exists (t,x)\in K:
\gamma(t,x)=0
\Big)=0.
\]

As a result, Conditions (i), (ii) and (iii) are satisfied, so Theorem \ref{thm:nualart-criterion-existence} yields that $\phi_ K=\sup_{(s,y) \in K}u(s,y) \in \mathbb{D}^{1,2}$, and that it is absolutely continuous with respect to the Lebesgue measure.  The proof of the first part of Theorem \ref{thm:4th-order} is thereby complete.

\subsubsection{Upgrading the result to $[0,T]\times [0,1]$} \label{subsubsec:upgrading-whole-spacek-k>0}
Let $x^{*}$ be a maximizer of $u_0$. The setting is exactly the same as in Section \ref{subsubsec:k=0-whole-space} under Regime \ref{case:neumann}, and so it suffices to prove that
\[
    \PP \left(\sup_{(t,x) \in [0,T] \times [0,1]} u(t,x) = u_0(x^*)\right) = 0.
    \]
Letting $u_0(x^{*}) = m_0 := \max_{x\in[0,1]} u_0(x)$, we aim to show
$
\mathbb{P}
\left(
\sup_{t\in(0,T],\,x\in[0,1]} u(t,x) > m_0
\right)
= 1,
$
and as before, we will prove the stronger statement:
\begin{equation}
\mathbb{P}
\left(
u(t_n,x^{*}) > u_0(x^{*})
\ \text{i.o.\ as } t_n \downarrow 0
\right)
=1.
\end{equation}
Assumption \ref{ass:u0-fourth} is assumed to hold throughout this subsection.

Fixing $t>0$ small we write
\begin{equation} \label{eq:u-decomp-D-Z-R-2}
 u(t,x^{*}) - u_0(x^{*})
=
D(t) + Z(t)+R(t),   
\end{equation}
where
\begin{align*}
D(t)
&:= \int_0^1 H_t(x^{*},y)u_0(y)dy - u_0(x^{*})
+ \int_0^t \int_0^1 H_{t-s}(x^{*},y) b(u(s,y)) dy \, ds, \\
Z(t)
&:= \int_0^t \int_0^1
H_{t-s}(x^{*},y)\,
\sigma(u_0(y))\,
W(ds,dy), \\
R(t)
&:= \int_0^t \int_0^1
H_{t-s}(x^{*},y)\,
\big(\sigma(u(s,y))-\sigma(u_0(y))\big)\,
W(ds,dy).
\end{align*}

Then $Z(t)$ is centered Gaussian, and by using uniform ellipticity and again (3.24) from \cite{AFK},
\begin{align}
\sqrt{\mathrm{Var}(Z(t))}= \sqrt{
\int_0^t \int_0^1
H_{t-s}(x^{*},y)^2
\sigma(u_0(y))^2
\,dy\,ds } \ge (C_\sigma)^{-1} \sqrt{
\int_0^t \int_0^1
H_r(x^{*},y)^2
\,dy\,dr} \ge (C_\sigma)^{-1} t^{\frac{3}{8}}.
\end{align}
Therefore there exist constants $c_1,c_2>0$ such that for $t$ small enough,
\begin{equation}
\mathbb{P}\big(Z(t) \ge c_1 t^{3/8}\big) \ge c_2.
\end{equation}

We decompose
$
D(t) = D_1(t) + D_2(t)
$
with
$$
D_1(t)
:=
\int_0^1 H_t(x^{*},y)u_0(y)dy-u_0(x^{*}),\quad{}
D_2(t)
:=
\int_0^t \int_0^1 H_{t-s}(x^{*},y)b(u(s,y))dy\,ds .
$$
The term $D_2(t)$ can be estimated as in \eqref{eq:D2-minkowski} to give
\begin{align}
\E\left(D_2(t)^2 \right)^{\frac{1}{2}}
\le
C_{T,b}\int_0^t\int_0^1
|H_{t-s}(x^{*},y)|
\,dy\,ds
\le C_{T,b} t . \label{eq:D2-fourth-order}
\end{align}

Turning now to the term $D_1(t)$, we have 
\begin{align}
-\,D_1(t)
&= \int_0^1 H_t(x^{*},y)\,\big(u_0(x^{*})-u_0(y)\big)\,dy \notag \\
&= \int_{|y-x^{*}|\le r_0} H_t(x^{*},y)\,\big(u_0(x^{*})-u_0(y)\big)\,dy
\;+\;
\int_{\{|y-x^{*}|> r_0\}\cap\{ 0<y<1 \}} H_t(x^{*},y)\,\big(u_0(x^{*})-u_0(y)\big)\,dy ,\nonumber\\
\label{eq:D1-split-kappa>0}
\end{align}
where $r_0$ is as in Assumption \ref{ass:u0-fourth} and $r_0<\frac{1}{3}$. Denoting
$$I_{\mathrm{near}}(t) := \int_{|y-x^{*}|\le r_0} H_t(x^{*},y)\,\big(u_0(x^{*})-u_0(y)\big)\,dy,$$
and 
$$I_{\mathrm{far}}(t) := \int_{\{|y-x^{*}|> r_0\}\cap\{ 0<y<1 \}} H_t(x^{*},y)\,\big(u_0(x^{*})-u_0(y)\big)\,dy$$
For the near term, we use that Assumption \ref{ass:u0-fourth} holds, and thus the derivative of $u_0$ is $C^{\alpha}$ for some $\alpha > 1/2$ locally around $x^{*}$. Note that $u_0'(x^{*})=0$. Indeed, if $x^* \in (0,1)$, this follows by the fact that it is an interior point and that $x^*$ is a maximizer. If $x^* = 0$ or $1$, then it follows by Assumption \ref{ass:u0-fourth}. Hence, 
\begin{align}
&|I_{\mathrm{near}}(t)|= \left|\int_{|y-x^{*}|\le r_0} H_t(x^{*},y)\,\int_{y}^{x^{*}}u_0'(z)dz dy \right| \\
&\qquad\le \int_{|y-x^{*}|\le r_0} |H_t(x^{*},y)|\,\int_{\min\{y,x^*\}}^{\max\{y,x^*\}}\left|u_0'(z)-u_0'(x^{*}) \right| dz \,dy \\
&\qquad\le C_\alpha \int_{|y-x^{*}|\le r_0} |H_t(x^{*},y)|\,\int_{\min\{y,x^*\}}^{\max\{y,x^*\}}|z-x^{*}|^{\alpha} dz \,dy \\
&\qquad\leq 
C_\alpha \int_{|y-x^{*}|\le r_0} |H_t(x^{*},y)|\,|x^{*}-y|^{\alpha+1} \,dy \leq C_{T,\rho,\alpha,r_0} t^{(\alpha+1)/4}.
\label{eq:12345}
\end{align}
In the last inequality we used Lemma \ref{integralboundholderforH}.

For the far term, using $0\le u_0(x^{*})-u_0(y)\le 2\|u_0\|_\infty$ and Lemma \ref{pointwiseboundonH} we get 
\begin{align}
&I_{\mathrm{far}}(t) \le 2\|u_0\|_\infty \int_{\{|y-x^{*}|> r_0\}\cap\{ 0<y<1 \}} |H_t(x^*,y)|\,dy \nonumber \\
&\le C_{T,\rho} \|u_0\|_\infty \int_{\{|y-x^{*}|> r_0\}\cap\{ 0<y<1 \}} 
t^{-1/4}
\left(
e^{-c|x^*-y|^{4/3}/t^{1/3}}
+
e^{-c\,m(x^*,y)^{4/3}/t^{1/3}}
\right)\,dy \label{eq:Ifar-1}
\end{align}
We estimate the two integrals separately. First, since $\alpha > -1$,
\begin{align}
   \int_{\{|y-x^{*}|> r_0\}\cap\{ 0<y<1 \}} 
t^{-1/4}
e^{-c|x^*-y|^{4/3}/t^{1/3}}
\,dy &\le \int_{|z|>r_0} t^{-1/4}e^{-c|z|^{4/3}/t^{1/3}}\,dz \nonumber  \\
&= \int_{|u|>r_0 t^{-1/4}} e^{-c|u|^{4/3}}\,du \leq C_{T,r_0} e^{-c_1 r_0^{4/3}/t^{1/3}} \leq C_{T,r_0,\alpha} t^{\frac{\alpha+1}{4}} \label{eq:Ifar-2}
\end{align}
The first inequality follows by the change of variables $z=y-x^*$. The second by letting $z=t^{1/4}u$, and hence $dz=t^{1/4}du$ and
$
\frac{|z|^{4/3}}{t^{1/3}}=|u|^{4/3}.
$

For the second term, we have that on the set $\{|y-x^*|>r_0\}$, $m(x^*,y)\ge r_0$. Therefore
\begin{align} 
    \int_{\{|y-x^{*}|> r_0\}\cap\{ 0<y<1 \}} 
t^{-1/4}
e^{-c\,m(x^*,y)^{4/3}/t^{1/3}} \,dy  \nonumber
&\le
\int_{\{|y-x^{*}|> r_0\}\cap\{ 0<y<1 \}} t^{-1/4}e^{-c r_0^{4/3}/t^{1/3}}\,dy \\
&\le
t^{-1/4}e^{-c r_0^{4/3}/t^{1/3}} \leq C_{T,r_0, \alpha} t^{\frac{\alpha+1}{4}}. \label{eq:Ifar-3}
\end{align}
Combining \eqref{eq:Ifar-1}, \eqref{eq:Ifar-2}, and \eqref{eq:Ifar-3},
\begin{equation} \label{IfarboundforH}
I_{\mathrm{far}}(t)\le 
C_{T,\rho,r_0} \|u_0\|_\infty t^{\frac{\alpha+1}{4}}.
\end{equation}

From \eqref{eq:12345} and \eqref{IfarboundforH} we obtain 
\[
|D_1(t)|
\le
C_{T,\rho,r_0,\alpha} t^{\frac{\alpha+1}{4}}.
\]
Consequently,  since $D_1(t)$ is deterministic and from \eqref{eq:D2-fourth-order}, we have, for small $t$,
\begin{equation*}
\E(D(t)^2)^{1/2}
\le |D_1(t)|+\E(D_2(t)^2)^{1/2}
\le C_{T,r_0,\rho,u_0,b,\alpha} t^{\frac{\alpha+1}{4}}
\end{equation*}
Applying Chebyshev's inequality with threshold \(\frac14 c_1 t^{3/8}\), we obtain, since $\alpha > \frac{1}{2}$
\begin{align*}
\PP\Bigl(|D(t)|\ge \tfrac14 c_1 t^{3/8}\Bigr)
&\le
\frac{\E[D(t)^2]}{\bigl(\frac14 c_1 t^{3/8}\bigr)^2} \leq 
C_{T,r_0,\rho,u_0,b,\alpha}\,\frac{t^{\frac{\alpha+1}{2}}}{t^{3/4}}
=
C_{T,r_0,\rho,u_0,b,\alpha}\, t^{\frac{\alpha}{2}-\frac{1}{4}} \rightarrow 0
\end{align*}

We proceed with estimating the term $R(t)$ as in \eqref{eq:Rt-123}. We have 
\begin{align}
 \mathbb{E}|R(t)|^2
&\le
\Lip(\sigma)^2
\int_0^t \int_0^1
H_{t-s}(x^{*},y)^2
\,
\mathbb{E}|u(s,y)-u_0(y)|^2
\,dy\,ds \nonumber \\
&\le
\Lip(\sigma)^2
\int_0^t \int_{|y-x^{*}|<\frac{r_0}{2}}
H_{t-s}(x^{*},y)^2
\,
\mathbb{E}|u(s,y)-u_0(y)|^2
\,dy\,ds \nonumber \\
&\qquad+
\Lip(\sigma)^2
\int_0^t \int_{|y-x^{*}|\geq \frac{r_0}{2}}
H_{t-s}(x^{*},y)^2
\,
\mathbb{E}|u(s,y)-u_0(y)|^2
\,dy\,ds \label{eq:Rt-1234thorder}
\end{align}
For the last term in \eqref{eq:Rt-1234thorder}, by using the bound $\sup_{s \in [0,T]} \sup_{ y \in [0,1]} \mathbb{E}(|u(s,y)|^2) \le C_T$ and Lemma \ref{pointwiseboundonH}
\begin{align}
&\int_0^t \int_{|y-x^{*}|\ge \frac{r_0}{2}}
H_{t-s}(x^{*},y)^2
\,\mathbb{E}\lvert u(s,y)-u_0(y)\rvert^2
\,dy\,ds  \nonumber \\
&\le C_T \int_0^t \int_{\{|y-x^{*}|> \frac{r_0}{2}\}\cap\{0<y<1\}} (t-s)^{-1/2}
\left(
e^{-c |y-x^*|^{4/3}/(t-s)^{1/3}}
+
e^{-c\,\frac{\min\{x^*+y,\; 2-x^*-y\}^{4/3}}{(t-s)^{1/3}}}
\right)
\,dy\,ds
\notag\\
&\le C_T \int_0^t \int_{\{|y-x^{*}|> \frac{r_0}{2}\}\cap\{0<y<1\}} (t-s)^{-1/2}
e^{-c r_0^{4/3}/(t-s)^{1/3}}
\,dy\,ds.
\label{eq:Rt-far4thorder}
\end{align}
The last line follows since, for \(0<x^*,y<1\),
$
x^*+y\ge |x^*-y|,
2-x^*-y\ge |x^*-y|.
$
Then, using the estimate $e^{-x} \le C_\beta x^{-\beta}$ for $x > 0$ and $\beta = 3/2$, we get
\begin{align}
    \int_0^t \int_{|y-x^{*}|\ge \frac{r_0}{2}}
H_{t-s}(x^{*},y)^2
\,\mathbb{E}\lvert u(s,y)-u_0(y)\rvert^2
\,dy\,ds &\le C_T \int_0^t  (t-s)^{-1/2}
e^{-c r_0^{4/3}/(t-s)^{1/3}} ds \nonumber \\
&\le C_{T} \int_0^t (t-s)^{-1/2} \frac{(t-s)^{1/2}}{r_0^{2}} ds \nonumber \\
&\le C_{T,r_0} t. \nonumber
\end{align}

For $|y-x^{*}|< \frac{r_0}{2}$ we claim that $\mathbb{E}|u(s,y)-u_0(y)|^2 \leq C_T s^{1/2}$. Indeed, first decompose it as
\begin{equation} \label{eq:Rt-1004thorder}
\begin{split}
u(s,y)-u_0(y)&=\int_{|z-y|<\frac{r_0}{2}} H_{s}(y,z)(u_0(z)-u_0(y))dz  + \int_{|z-y| \geq \frac{r_0}{2}} H_{s}(y,z)(u_0(z)-u_0(y))dz \\
&+\int_0^s \int_0^1 H_{s-r}(y,z)b(u(r,z))dz dr 
+\int_0^s \int_0^1 H_{s-r}(y,z)\sigma(u(r,z)) dW(z,r).
\end{split}
\end{equation}
Then, for the stochastic term
\begin{equation*}
\begin{split}
   \mathbb{E} \left(\int_0^s \int_0^1 H_{s-r}(y,z)\sigma(u(r,z)) dW(z,r) \right)^2 &= \mathbb{E} \left(\int_0^s \int_0^1 H_{s-r}^2(y,z)\sigma^2(u(r,z)) dz dr \right) \\
   &\leq C_\sigma \int_0^s \int_0^1 H_{s-r}^2(y,z) dz dr  \leq C_{T,\sigma} s^{3/4}. 
\end{split}
\end{equation*}
For the first term we have 
\begin{equation}
\begin{split}
 \left|\int_{|z-y|<\frac{r_0}{2}} H_{s}(y,z)(u_0(z)-u_0(y))dz \right| &\leq \int_{|z-y|<\frac{r_0}{2}} |H_{s}(y,z)| |u_0(z) -u_0(y)| dz \\
 &\leq C\int_{|z-y|<\frac{r_0}{2}}  |H_{s}(y,z)| |z-y| dz \leq C_{T} s^{1/4},
\end{split}
\end{equation}
which follows from a direct modification of Lemma \eqref{integralboundholderforH} and from $u_0$ being Lipschitz in the interval $(x^{*}-r_0, x^{*}+r_0)$.
The other terms in \eqref{eq:Rt-1004thorder} can be estimated as the $I_{\mathrm{far}}$ part of $D_1$ and as $D_2$ above, verifying that $\mathbb{E}|u(s,y)-u_0(y)|^2 \leq C_{T,\sigma,b} s^{1/2}$.

Plugging this in the term in the second line of \eqref{eq:Rt-1234thorder}, 
\begin{equation} \label{eq:Rt-near4thorder}
\begin{split}
&\int_0^t \int_{|y-x^{*}|<\frac{r_0}{2}}
H_{t-s}(x^{*},y)^2
\,
\mathbb{E}|u(s,y)-u_0(y)|^2
\,dy\,ds \\
&\leq C_{T,\sigma,b} \int_0^t \int_{|y-x^{*}|<\frac{r_0}{2}}
H_{t-s}(x^{*},y)^2 s^{1/2}
\,dy\,ds \leq C_{T,\sigma,b} \int_0^t s^{1/2} \int_0^1
H_{t-s}(x^{*},y)^2 
\,dy  \,ds \\
&\leq C_{T,\sigma,b} \int_0^t (t-s)^{-1/4} s^{1/2} \,ds= C_{T,\sigma,b} t^{5/4}.
\end{split}
\end{equation}
We used Lemma \ref{4thgreenbounds} and standard beta function calculations (cf. \eqref{eq:beta-example-2}). Hence, combining \eqref{eq:Rt-1234thorder}, \eqref{eq:Rt-far4thorder}, and \eqref{eq:Rt-near4thorder}, we get that
\begin{equation}
     \mathbb{E}|R(t)|^2 \le C_{T,r_0,b,\sigma} t.
\end{equation}

 Hence, by Chebyshev inequality,
\begin{equation}
\mathbb{P}
\big(
|R(t)| \ge \tfrac14 c_1 t^{3/8}
\big)
\leq C_{T,\sigma} t^{1/4}.
\end{equation}

Finally, combining the bounds for $Z(t)$, $R(t)$ and $D(t)$, we have that for $t$ sufficiently small,
\[
\mathbb{P}
\Big(
Z(t)\ge c_1 t^{3/8},\,
|R(t)|\le\tfrac14 c_1 t^{3/8},\,
|D(t)|\le\tfrac14 c_1 t^{3/8}
\Big)
\ge \frac{c_2}{2}>0.
\]
Therefore from \eqref{eq:u-decomp-D-Z-R-2},
there exist $p_*>0$ and $t_0>0$ such that
\[
\mathbb{P}
\big(
u(t,x^{*}) > u_0(x^{*})
\big)
\ge p_*,
\qquad
\forall t\in(0,t_0].
\]
By appealing to the 0-1 law for the germ $\sigma$-algebra $\mathcal{F}_{0+}$ as in Proposition \ref{dirichletneumannmaximizeratinitial} and \eqref{eq:prob-i.o.-event}, we conclude that $\mathbb{P}
\left(
u(t_n,x^{*}) > u_0(x^{*})
\ \text{i.o.\ as } t_n \downarrow 0
\right)
=1$ and thus 
$\PP \left(\sup_{(t,x) \in [0,T] \times [0,1]} u(t,x) = u_0(x^*)\right) = 0$. The proof of the second part of Theorem \ref{thm:4th-order} is thereby complete.

\section*{Acknowledgements}
The authors are grateful to Robert Dalang and Lluís Quer-Sardanyons for fruitful discussions and helpful communication. The work of AS and PZ was Funded by the Deutsche Forschungsgemeinschaft (DFG, German Research Foundation) under Germany's Excellence Strategy EXC 2044/2 –390685587, Mathematics Münster: Dynamics–Geometry–Structure.



\bibliographystyle{acm}
\bibliography{mybibliography}

\end{document}